\documentclass[10pt]{amsart}

\usepackage{graphicx, amssymb, enumerate, pstricks, pst-node}
\usepackage{pst-all}
\usepackage{latexsym}
\usepackage{amsmath}
\usepackage{epsfig}
\usepackage{epic}
\usepackage{amssymb}
\usepackage{url}

\newtheorem{theorem}{Theorem}[section]
\newtheorem*{hypothesis}{Hypothesis~VII}
\newtheorem{lemma}[theorem]{Lemma}
\newtheorem{corollary}[theorem]{Corollary}

\newtheorem{sublemma}{}[theorem]

\newcommand{\ba}{\backslash}
\newcommand{\cl}{{\rm cl}}
\newcommand{\fcl}{{\rm fcl}}

\newcommand{\cls}{{\cl^*}}

\newcommand{\thc}{$3$-connected}
\newcommand{\ths}{$3$-separation}
\newcommand{\ifc}{internally $4$-connected}
\newcommand{\sfc}{sequentially $4$-connected}

\newcommand{\ffsc}{$(4,4,S)$-connected}

\newcommand{\cn}{contradiction}
\newcommand{\btu}{\bigtriangleup}
\newcommand{\ftv}{$(4,3)$-violator}
\newcommand{\ffsv}{$(4,4,S)$-violator}
\newcommand{\ort}{orthogonality}

\newcommand{\ffspc}{$(4,5,S,+)$-connected}
\newcommand{\ffspv}{$(4,5,S,+)$-violator}
\newcommand{\ns}{non-sequential}
\newcommand{\al}{\alpha}
\newcommand{\be}{\beta}
\newcommand{\ga}{\gamma}
\newcommand{\de}{\delta}
\begin{document}

\title[A Splitter Theorem for Internally $4$-connected Binary Matroids VII]{Towards a Splitter Theorem for Internally $4$-connected Binary Matroids VII}

\thanks{The first author was supported by NSF IRFP Grant OISE0967050, an LMS Scheme 4 grant, and an AMS-Simons travel grant. The second author was supported by the National Security Agency.}

\author{Carolyn Chun}
\address{School of Mathematical Sciences,
Brunel University,
London,  
England}
\email{chchchun@gmail.com}

\author{James Oxley}
\address{Department of
 Mathematics, Louisiana State University, Baton Rouge, Louisiana, USA}
\email{oxley@math.lsu.edu}

\date{\today}

\begin{abstract} 
Let $M$ be a $3$-connected binary matroid; $M$ is  internally $4$-connected if one side of every $3$-separation is a triangle or a triad, and $M$ is \ffsc\ if one side of every $3$-separation is a triangle, a triad, or a $4$-element fan. 
Assume $M$ is   \ifc\   
 and that neither $M$ nor its dual is  a cubic 
M\"{o}bius or planar ladder or a certain coextension thereof. 
Let $N$ be an \ifc\ proper minor of $M$.  
Our aim is to show that  $M$   has a proper \ifc\ minor with an $N$-minor that can be obtained from $M$ either by removing at most four elements, or by   removing elements in an easily described way from a special substructure of $M$. When this aim cannot be met, the earlier papers in this series showed that, 
 up to duality, $M$ has a good bowtie, that is, a pair, $\{x_1,x_2,x_3\}$ and $\{x_4,x_5,x_6\}$, of disjoint triangles and a cocircuit, $\{x_2,x_3,x_4,x_5\}$, where $M\ba x_3$ has an $N$-minor and is \ffsc. 
We also showed that, when $M$ has a good bowtie, either $M\ba x_3,x_6$ has an $N$-minor and $M\ba x_6$ is \ffsc; or $M\ba x_3/x_2$ has an $N$-minor and is \ffsc. 
In this paper, we show that, when $M\ba x_3,x_6$ has no $N$-minor, $M$ 
has an \ifc\ proper minor  with an $N$-minor that can be obtained  from $M$ by removing at most three elements, or by   removing elements in a well-described way from a special substructure of $M$. 
This is the penultimate step towards obtaining a splitter theorem for the class of \ifc\ binary matroids.
\end{abstract}

\maketitle

\section{Introduction}
\label{introduction}

Seymour's Splitter Theorem~\cite{seymour} proved that if $N$ is a $3$-connected proper minor of a $3$-connected matroid $M$, then $M$ has a proper $3$-connected minor $M'$ with an $N$-minor such that $|E(M)  - E(M')| \le 2$.  Furthermore, such an $M'$ can be found with $|E(M)|-|E(M')|=1$ unless $r(M)\geq 3$ and $M$ is a wheel or a whirl.  
This result has been extremely useful in inductive and constructive arguments for \thc\ matroids.  
In this paper, we prove the penultimate step in obtaining a corresponding result for \ifc\ binary matroids.  
Specifically, we will prove that if $M$ and $N$ are
\ifc\ binary matroids, and $M$ has
a proper $N$-minor, then $M$ has a proper minor $M'$ such that
$M'$ is internally $4$-connected with an $N$-minor, and
$M'$ can be produced from $M$ by a small number of
simple operations. 

 Any unexplained matroid terminology used here will follow \cite{oxrox}. The only $3$-separations allowed in an internally $4$-connected matroid have a triangle or a triad on one side. A $3$-connected matroid  $M$ is {\it $(4,4,S)$-connected} if, for every $3$-separation $(X,Y)$ of $M$, one of $X$ and $Y$ is a triangle, a triad, or a $4$-element fan, that is, a $4$-element set $\{x_1,x_2,x_3,x_4\}$ that can be ordered so that $\{x_1,x_2,x_3\}$ is a triangle and $\{x_2,x_3,x_4\}$ is a triad. 

To provide a context for the main theorem of this paper, we briefly describe our progress towards obtaining the desired splitter theorem.  Johnson and Thomas \cite{johtho} showed that, even for graphs, a splitter theorem in the internally $4$-connected case must take account of some special examples. For $n \ge 3$, let  $G_{n+2}$  be the {\it biwheel} with $n+2$ vertices, that is, $G_{n+2}$ consists of  an $n$-cycle $v_1,v_2,\ldots,v_{n},v_1$, the {\it rim}, and two additional  vertices, $u$ and $w$, both of which are adjacent to every $v_i$. Thus  the dual of $G_{n+2}$ is a cubic planar ladder. Let $M$ be the cycle matroid of $G_{2n+2}$ for some $n \ge 3$ and 
let $N$ be the cycle matroid of the graph that is obtained by proceeding around the rim of $G_{2n+2}$ and alternately deleting the edges from the rim vertex to $u$ and to $w$. Both $M$ and $N$ are internally $4$-connected but there is no   internally $4$-connected proper minor of $M$ that has a proper $N$-minor. We can modify $M$ slightly and still see the same phenomenon. Let $G_{n+2}^+$ be obtained from $G_{n+2}$ by adding a new edge $z$ joining the hubs $u$ and $w$. Let $\Delta_{n+1}$ be the binary matroid that is obtained from 
$M(G_{n+2}^+)$ by deleting the edge $v_{1}v_n$ and adding the third element on the line spanned by $wv_n$ and $uv_{1}$. This new element is also on the line spanned by $uv_n$ and $wv_{1}$. For $r \ge 3$, Mayhew, Royle, and Whittle~\cite{mayroywhi} call $\Delta_r$ the {\it rank-$r$ triangular M\"{o}bius matroid} and note that $\Delta_r \ba z$ is the dual of the cycle matroid of a cubic M\"{o}bius ladder. 
The following is the main result of \cite[Theorem~1.2]{cmoIII}. 

\begin{theorem}
\label{44S}
Let $M$ be an \ifc~binary matroid with an \ifc~proper minor $N$ such that $|E(M)|\geq 15$ and $|E(N)|\geq 6$.  
Then 
\begin{itemize}
\item[(i)] $M$ has a proper minor $M'$ such that $|E(M)-E(M')|\leq 3$ and $M'$ is \ifc\ with an $N$-minor; or 
\item[(ii)] for some $(M_0,N_0)$ in $\{(M,N), (M^*,N^*)\}$, the matroid $M_0$ has a triangle $T$ that contains an element $e$ such that  $M_0\ba e$  is $(4,4,S)$-connected with an $N_0$-minor; or 
\item[(iii)] $M$ is isomorphic to $M(G_{r+1}^+)$, $M(G_{r+1})$, $\Delta_r$,   or 
$\Delta_r \ba z$ for some $r \ge 5$.
\end{itemize}
\end{theorem}

\begin{figure}[htb]
\centering
\includegraphics[scale = 0.7]{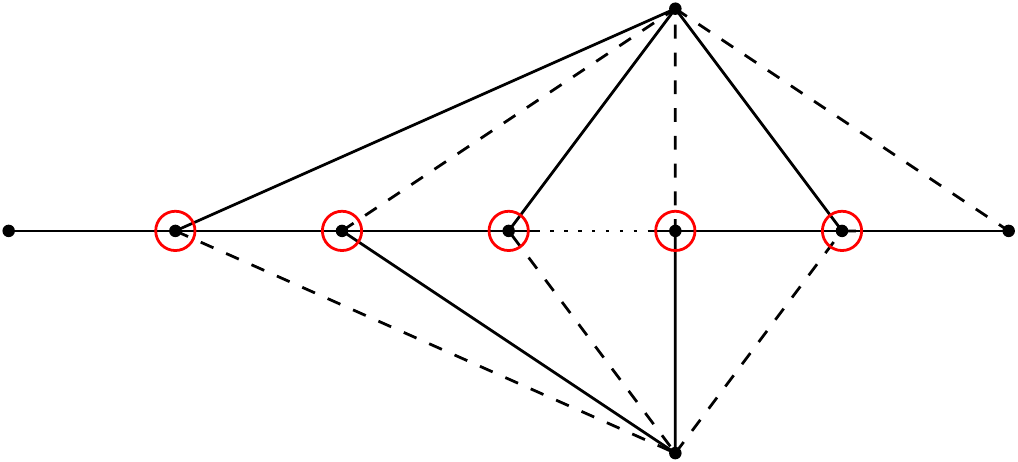}
\caption{All the elements shown are distinct. There are at least three dashed elements; and all dashed elements are deleted.}
\label{gmcdashed}
\end{figure}

That theorem led us to consider those matroids for which the second outcome in the theorem holds. In order to state the next result, we need to define some special structures. Let $M$ be an \ifc\ binary matroid and $N$ be an \ifc\ proper minor of $M$. Suppose $M$ has 
 disjoint triangles $T_1$ and $T_2$ and a $4$-cocircuit $D^*$ contained in their union. 
We call this structure a {\it bowtie} and denote it by $(T_1,T_2,D^*)$. If   $D^*$ has an element $d$ such that $M\ba d$ has an $N$-minor and $M\ba d$ is $(4,4,S)$-connected, then $(T_1,T_2,D^*)$ is a {\it good bowtie}. Motivated by (ii)  of the last theorem, 
we aim to discover more about the structure of $M$ when it has a triangle containing an element $e$ such that $M\ba e$ is \ffsc\ with an $N$-minor. One possible outcome here is that $M$ has a good bowtie. Indeed, as the next result shows, if that outcome or its dual does not arise, we get a small number of easily described alternatives. We shall need two more definitions.  A {\it terrahawk} is the graph  that is obtained from a  cube by adjoining a new vertex and adding edges from the new  vertex to each of the four vertices that bound some fixed face of the cube. Figure~\ref{gmcdashed} shows  a modified graph diagram, which we will use to keep track of some of the circuits and cocircuits in $M$, even though $M$ need not be graphic.   
Each of the cycles in such a graph diagram corresponds to a circuit  of $M$ while a circled vertex indicates a known cocircuit of $M$.
We refer to the structure in Figure~\ref{gmcdashed} as an {\it open rotor chain} noting that all of the elements in the figure are distinct and, for some $n \ge 3$, there are $n$ dashed edges.  The figure may suggest that $n$ must be even but we impose no such restriction. 
We will refer to deleting the dashed elements from Figure~\ref{gmcdashed} as {\it trimming an open rotor chain}.   

{   
\begin{figure}[h]
\center
\includegraphics{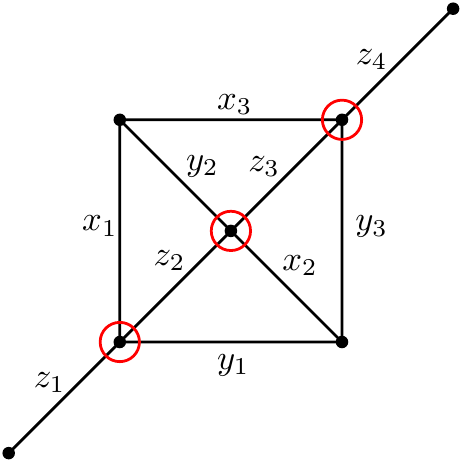}

\caption{An augmented $4$-wheel.}
\label{a4w}
\end{figure}

We need to define another special structure.  
An {\it augmented $4$-wheel} consists of a $4$-wheel restriction of $M$ with triangles $\{z_2,x_1,y_2\}, \{y_2,x_3,z_3\}, \{z_3,y_3,x_2\}, \{x_2,y_1,z_2\}$ along with two additional distinct elements $z_1$ and $z_4$ such that $M$ has $\{x_1,y_1,z_1,z_2\}, \{x_2,y_2,z_2,z_3\}$, and $\{x_3,y_3,z_3,z_4\}$ as cocircuits.  We call $\{x_2,y_2,z_2,z_3\}$ the {\it central cocircuit} of the augmented $4$-wheel. A diagrammatic representation of an augmented $4$-wheel is shown in Figure~\ref{a4w}. 
}

The following is~\cite[Corollary~1.4]{cmoV}.
 
 \begin{theorem}
\label{mainone4}
Let $M$ and $N$ be \ifc\ binary matroids such that $|E(M)| \ge 16$ and $|E(N)|\geq 6$. Suppose that $M$ has a triangle $T$ containing an element $e$ for which $M\ba e$ is \ffsc\ with an $N$-minor. Then one of the following holds.
\begin{itemize}
\item[(i)] $M$ has an \ifc\ minor $M'$ that has an $N$-minor such that  $1 \le |E(M) - E(M')| \le 3$; or {  $|E(M) - E(M')| = 4$ and, for some $(M_1,M_2)$ in $\{(M,M'), (M^*,(M')^*)\}$, the matroid $M_2$ is obtained from $M_1$ by deleting the central cocircuit of an augmented $4$-wheel; or }
\item[(ii)] $M$ or $M^*$  has a good bowtie; or 
\item[(iii)]  $M$ is the cycle matroid of a terrahawk; or
\item[(iv)] for some $(M_0,N_0)$ in $\{(M,N), (M^*,N^*)\}$, the matroid $M_0$ contains 
an open rotor chain that can be trimmed to obtain an \ifc\ matroid with an $N_0$-minor.  
\end{itemize}
\end{theorem}

Note that there is a small error in \cite[Theorem 1.1]{cmoV} since it requires at least five elements to be removed when trimming an open rotor chain. But, as the proof there makes clear, trimming exactly four elements is a possibility. Trimming exactly three elements is also possible but that is included under (i) of \cite[Theorem~1.1]{cmoV}. 

\begin{figure}[b]
\center
\includegraphics{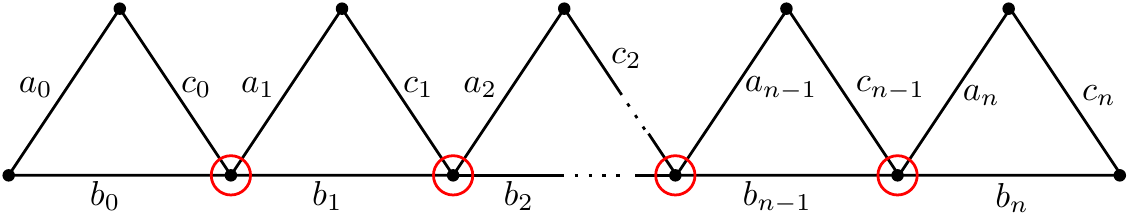}
\caption{A  string of bowties.  All elements are distinct except that $a_0$ may be the same as $c_n$.}
\label{bowtiechainfig}
\end{figure}



This theorem leads us to consider a good bowtie $(\{x_1,x_2,x_3\}, \linebreak \{x_4,x_5,x_6\}, \{x_2,x_3,x_4,x_5\})$ in an \ifc\ binary matroid $M$ where $M\ba x_3$ is \ffsc\ with an $N$-minor. In $M\ba x_3$, we see that $\{x_5,x_4,x_2\}$ is a triad and $\{x_6,x_5,x_4\}$ is a
triangle, so
$\{x_6,x_5,x_4,x_2\}$ is a $4$-element fan. By \cite[Lemma 2.5]{cmoIV}, which is included below as Lemma~\ref{airplane}, either
\begin{itemize}
\item[(i)] $M\ba x_3,x_6$ has an $N$-minor; or 
\item[(ii)] $M\ba x_3,x_6$ does not have an $N$-minor, but $M\ba x_3/x_2$ is \ffsc\ with an $N$-minor.
\end{itemize}
In~\cite{cmoVI}, we considered the case when (i) holds and  $M\ba x_6$ is not \ffsc.  
In this paper, we focus on the case when (ii) holds. 
The next and final paper in this series will complete the work to obtain the splitter theorem by considering the case when $M\ba x_3,x_6$ has an $N$-minor and $M\ba x_6$ is \ffsc.  
Before stating the main result  of \cite{cmoVI}, we define some structures that require special attention.

In a matroid $M$, a {\it string of bowties}  is a sequence $\{a_0,b_0,c_0\},\linebreak \{b_0,c_0,a_1,b_1\},\{a_1,b_1,c_1\}, 
\{b_1,c_1,a_2, 
 b_2\},\dots,\{a_n,b_n,c_n\}$ with $n \ge 1$ such that   
 \begin{itemize} 
 \item[(i)] $\{a_i,b_i,c_i\}$ is a triangle for all $i$ in $\{0,1,\dots, n\}$;  
 \item[(ii)] $\{b_j,c_j,a_{j+1},b_{j+1}\}$ is a cocircuit for all $j$ in $\{0,1,\dots ,n-1\}$; and 
 \item[(iii)] the elements $a_0,b_0,c_0,a_1,b_1,c_1,\ldots,a_n,b_n$, and $c_n$ are distinct except that $a_0$ and $c_n$ may be equal.
 \end{itemize}

\begin{figure}[b]
\center
\includegraphics[scale=0.72]{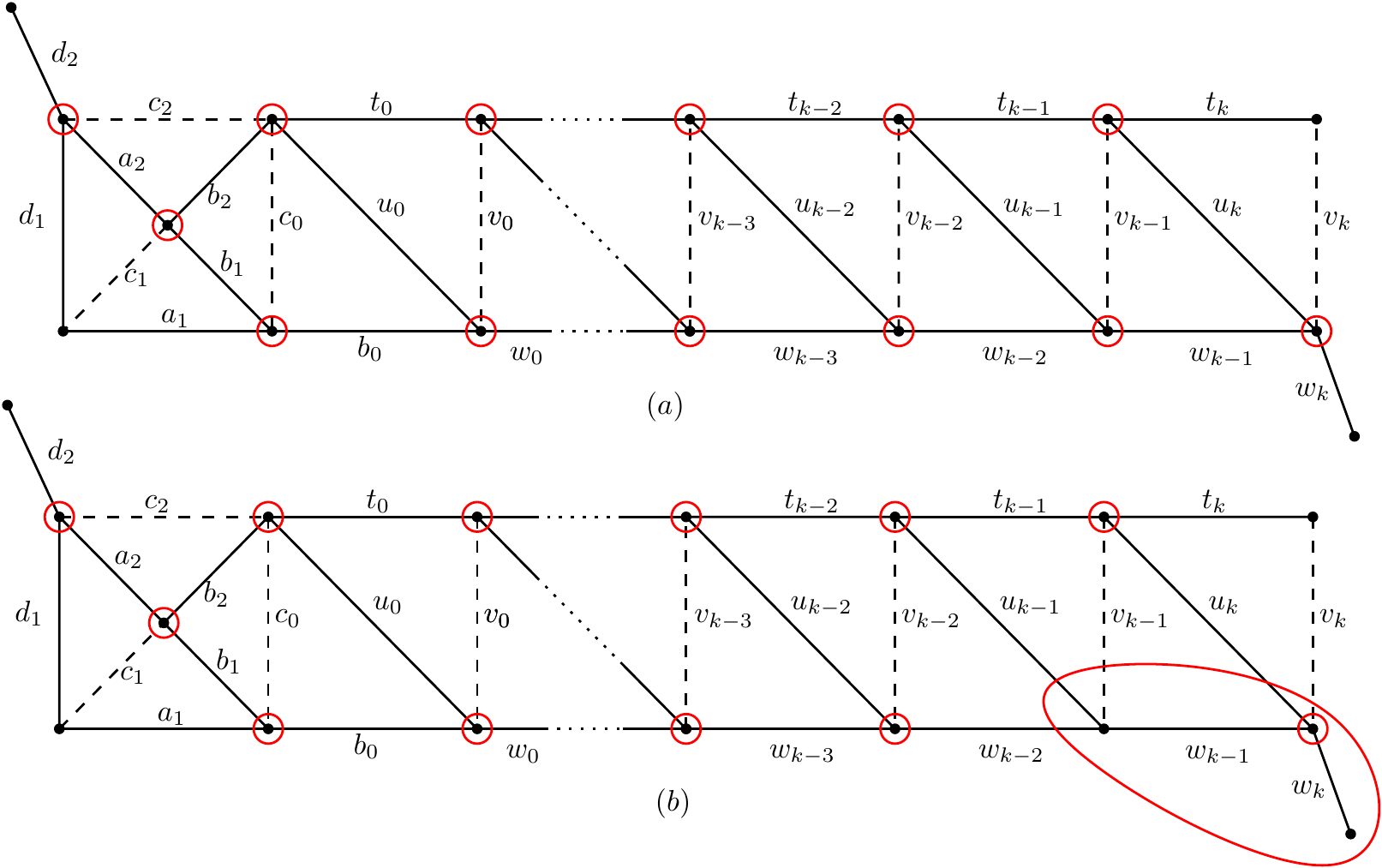}
\caption{In both (a) and (b), the elements shown are distinct, except that $d_2$ may be $w_k$.  Furthermore, in (a), $k\geq 0$; and in (b), $k \geq 1$ and $\{w_{k-2},u_{k-1},v_{k-1},u_k,v_k\}$ is a cocircuit.}
\label{bonesaws0}
\end{figure}

The reader should note that this differs slightly from the definition we gave in \cite{cmochain} in that  here we allow $a_0$ and $c_n$ to be equal instead of requiring all of the elements to be distinct. 
Figure~\ref{bowtiechainfig} illustrates a string of bowties, but this diagram may obscure the  potential complexity of such a string. Evidently $M\ba c_0$ has $\{c_1,b_1,a_1,b_0\}$ as a 4-fan. Indeed, $M\ba c_0,c_1,\ldots,c_i$ has a $4$-fan for all $i$ in $\{0,1,\ldots,n-1\}$. We shall say that the matroid $M\ba c_0,c_1,\ldots,c_n$ has been obtained from $M$ by {\it trimming a string of bowties}. This operation plays a prominent role in our main theorem, and is the underlying operation in trimming an open rotor chain. 
Before stating this theorem, we introduce the other operations that incorporate this process of trimming a string of bowties. 
Such a string can attach to the rest of the matroid in a variety of ways. In most of these cases, the operation of trimming the string will produce an \ifc\ minor of $M$ with an $N$-minor. But, when the bowtie string is embedded in a modified  quartic ladder in certain ways, we need to adjust the trimming process.

\begin{figure}[htb]
\center
\includegraphics[scale=0.72]{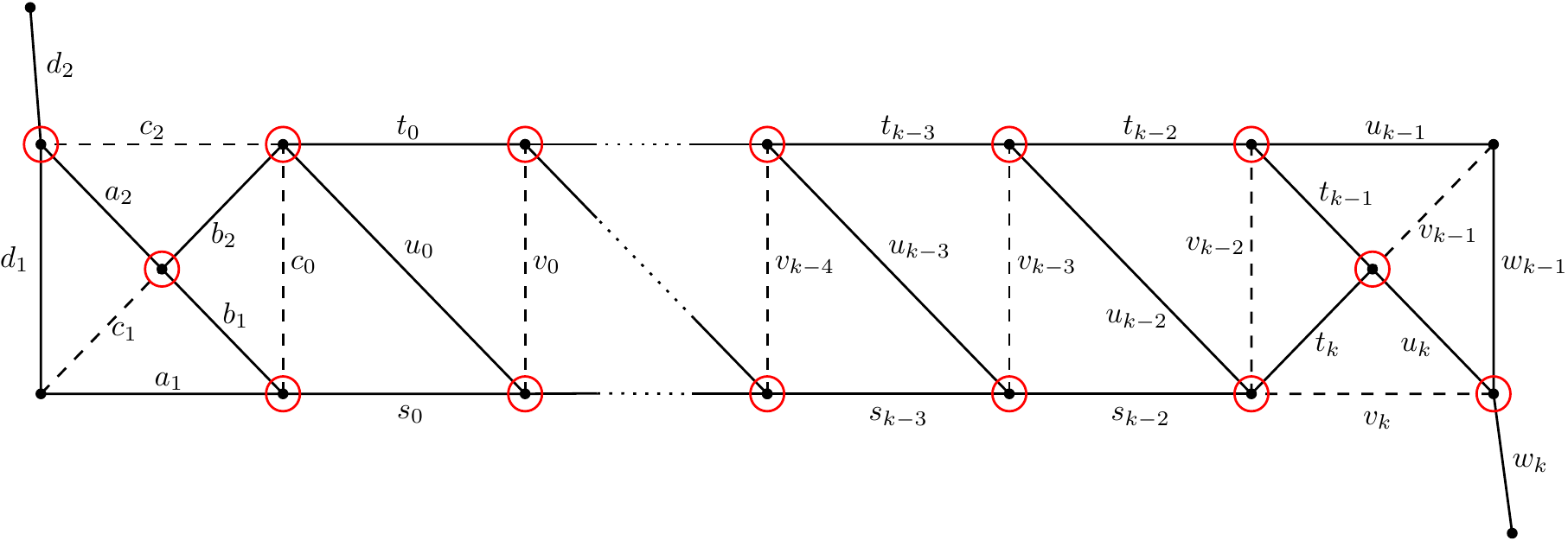}
\caption{In this configuration, $k\geq 2$ and the elements are all distinct except that $d_2$ may be $w_k$.}
\label{caterpillarwhole0}
\end{figure}


Consider the three configurations shown in  Figure~\ref{bonesaws0} and Figure~\ref{caterpillarwhole0} where the elements in each configuration are distinct except that $d_2$ may equal $w_k$. We refer to each of these configurations as an {\it enhanced quartic ladder}. Indeed, in each configuration, we can see a portion of a quartic ladder, which can be thought of as two interlocking bowtie strings, one pointing up and one pointing down. In each case, we focus on $M\ba c_2,c_1,c_0,v_0,v_1,\ldots,v_k$ saying that this matroid has been obtained from $M$ by an {\it enhanced-ladder move}. 

\begin{figure}[t]
\center
\includegraphics[scale=0.72]{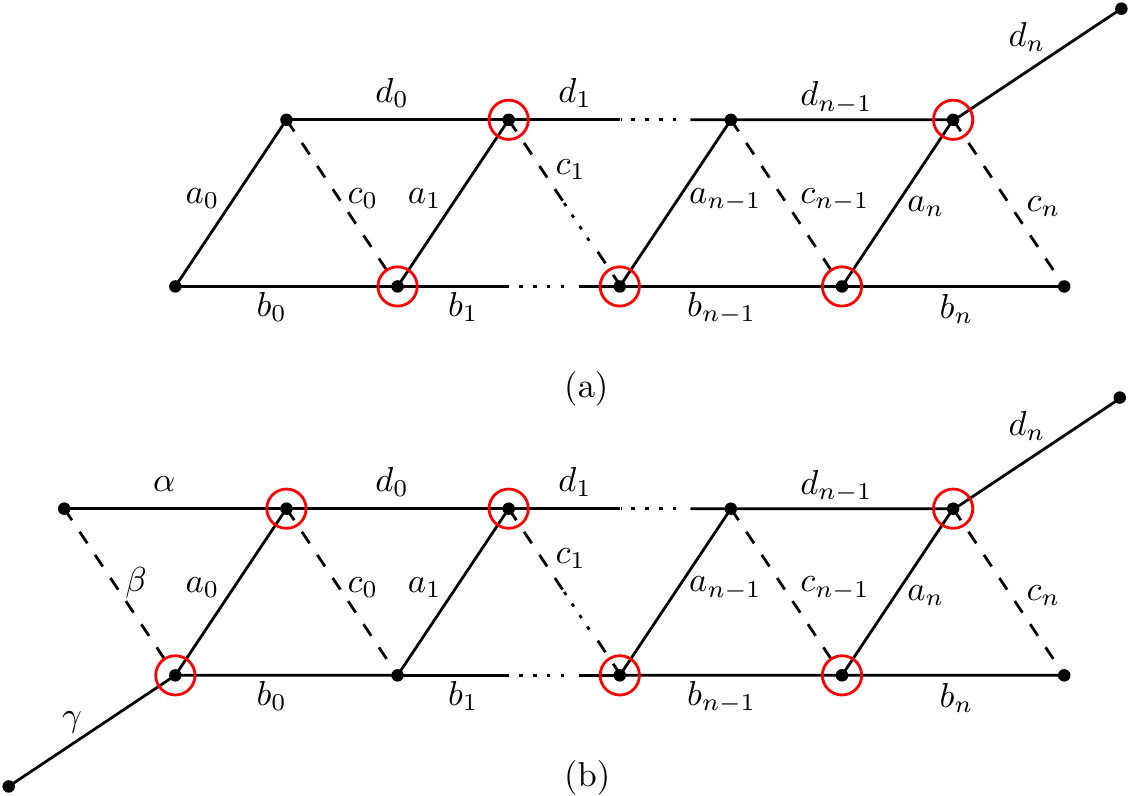}
\caption{In (a) and (b), $n\geq 2$ and the elements shown are distinct, with the exception that $d_n$ may be the same as $\gamma$ in (b).  Either $\{d_{n-2},a_{n-1},c_{n-1},d_{n-1}\}$ or $\{d_{n-2},a_{n-1},c_{n-1},a_n,c_n\}$ is a cocircuit in (a) and (b).  Either $\{b_0,c_0,a_1,b_1\}$ or $\{\be,a_0,c_0,a_1,b_1\}$ is also a cocircuit in (b).}
\label{laddery}
\end{figure}

\begin{figure}[b]
\center
\includegraphics[scale=0.72]{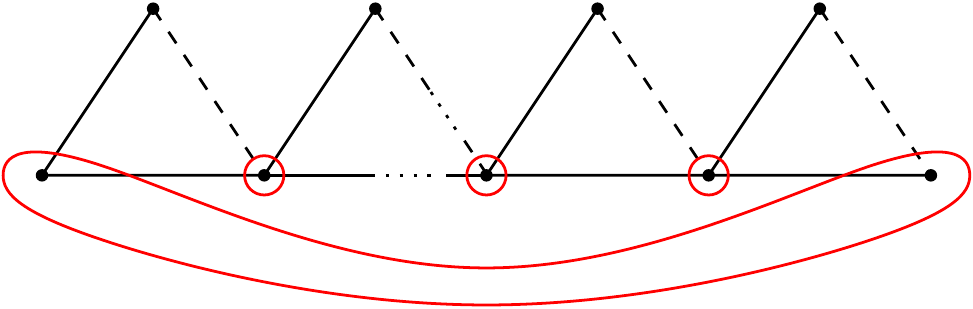}
\caption{A bowtie ring.  All elements are distinct.  The ring contains at least three triangles.}
\label{btringfig}
\end{figure}

Suppose that $\{a_0,b_0,c_0\}, \{b_0,c_0,a_1,b_1\}, \{a_1,b_1,c_1\},\dots, \{a_n,b_n,c_n\}$ is a bowtie string for some $n \ge 2$.  Assume, 
in addition, that $\{b_n,c_n,a_0,b_0\}$ is a cocircuit. Then the string of bowties has wrapped around on itself as in Figure~\ref{btringfig}, and we call the resulting structure a {\it ring of bowties}. 
We refer to each of the structures in Figure~\ref{laddery} as a {\it ladder structure} and we refer to removing the dashed elements in Figure~\ref{btringfig}  and Figure~\ref{laddery} as {\it trimming a ring of bowties} and  {\it trimming a ladder structure}, respectively.  

In the case that trimming a string of bowties in $M$ yields an \ifc\ matroid with an $N$-minor, we are able to ensure that the string of bowties belongs to one of the more highly structured objects we have described above.  
The following theorem is the main result of~\cite[Theorem~1.3]{cmoVI}.  

\begin{theorem}
\label{killcasek}
Let $M$ and $N$ be \ifc\ binary matroids such that $|E(M)|\geq 13$ and $|E(N)|\geq 7$.  
Assume that $M$ has  a bowtie $(\{x_0,y_0,z_0\},\{x_1,y_1,z_1\},\{y_0,z_0,x_1,y_1\})$, where $M\ba z_0$ is \ffsc, $M\ba z_0,z_1$ has an $N$-minor, and $M\ba z_1$ is not \ffsc.  
Then one of the following holds.  
\begin{itemize}
\item[(i)] $M$ has a proper minor $M'$ such that $|E(M)|-|E(M')|\leq 3$ and $M'$ is \ifc\ with an $N$-minor; or 
\item[(ii)] $M$ contains an open rotor chain, a ladder structure, or a ring of bowties that can be trimmed to obtain an \ifc\ matroid with an $N$-minor; or
\item[(iii)] $M$ contains an enhanced quartic ladder from which an \ifc\ minor of $M$ with an $N$-minor can be obtained by an enhanced-ladder move.
\end{itemize}
\end{theorem}

\begin{figure}[t]
\center
\includegraphics[scale = 0.8]{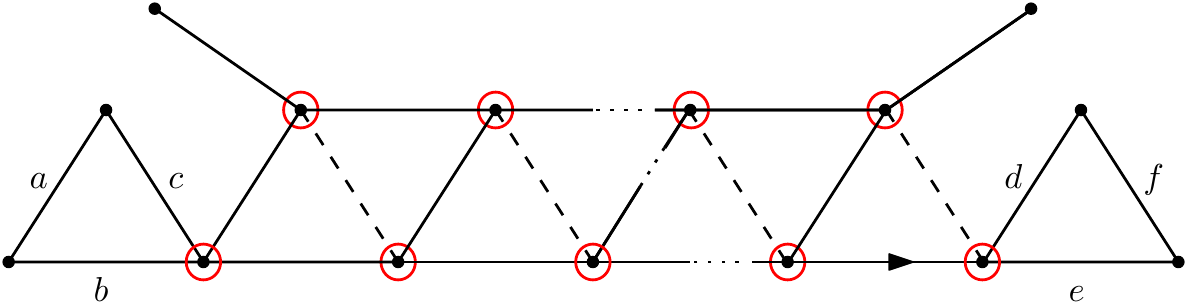}
\caption{All of the elements are distinct except that $a$ may be $f$, or $\{a,b,c\}$ may be $\{d,e,f\}$. There are at least three dashed elements.}
\label{lobster}
\end{figure}

In Theorem~\ref{killcasek}, not all of the moves  that we perform on $M$ to obtain an intermediate \ifc\ binary matroid with an $N$-minor are bounded in size, but each unbounded move is highly structured. 
In this paper, we shall require one more such unbounded move. 
When $M$ contains the structure in Figure~\ref{lobster}, where the elements are all distinct except that $a$ may be $f$,  or $\{a,b,c\}$ may be $\{d,e,f\}$, we say that $M$ contains an {\it open quartic ladder}. We will refer to deleting the dashed elements and contracting the arrow edge as a {\it mixed ladder move}.  This is the only unbounded move that uses a contraction as well as a number of deletions. Note that both of the   vertices of degree one in the diagram differ from the vertices closest to them. 

The following theorem is the main result of this paper.  

\begin{theorem}
\label{mainguy}
Let $M$ and $N$ be \ifc\ binary matroids such that  $|E(M)|\geq 16$ and $|E(N)|\geq 7$.  
Let $M$ have a bowtie $(\{1,2,3\},\{4,5,6\},\{2,3,4,5\})$, where $M\ba 4$ is \ffsc\ with an $N$-minor, and $M\ba 1,4$ has no $N$-minor.  
Then, for some  $(M_0,N_0)$ in $\{(M,N), (M^*,N^*)\}$, one of the following holds. 
\begin{itemize}
\item[(i)] $M_0$ has a proper \ifc\ minor $M'$ such that $M'$ has an $N_0$-minor and either $|E(M)|-|E(M')|\leq 3$,  or {  $|E(M) - E(M')| = 4$ and $M'$ is obtained from $M_0$ by deleting the central cocircuit of an augmented $4$-wheel; or }
\item[(ii)] $M_0$ contains an open rotor chain, a ladder structure, or a ring of bowties that can be trimmed to obtain an \ifc\ matroid with an $N_0$-minor; or 
\item[(iii)] $M_0$ contains an open quartic ladder and an \ifc\ matroid with an $N_0$-minor can be obtained by a mixed ladder move; or 
\item[(iv)] $M_0$ contains an enhanced quartic ladder from which an \ifc\ minor of $M_0$ with an $N_0$-minor can be obtained by an enhanced-ladder move.  
\end{itemize}
\end{theorem}

An outline of the proof of this theorem is given in Section~\ref{out}. That section separates the argument into three subcases and these cases are treated in the three subsequent sections.  The results from those three sections are combined in Section~\ref{pomt} to complete the proof of the theorem.  Before all of that, Section~\ref{prelim} gives some basic preliminaries while 
Sections~\ref{bqr}  and \ref{bl} present some properties of, respectively, bowties and quasi rotors, and bowties and ladders.

The next corollary follows immediately by combining Theorem~\ref{mainguy} with Theorem~\ref{killcasek}.  This corollary provides the context for the next and final paper~\cite{cmoVIII} in this series, which proves a splitter theorem for \ifc\ binary matroids.  

\begin{corollary}
\label{chapter7}
Let $M$ and $N$ be \ifc\ binary matroids such that  $|E(M)|\geq 16$ and $|E(N)|\geq 7$.  
If $M$ has a bowtie $(\{1,2,3\},\{4,5,6\},\{2,3,4,5\})$, where $M\ba 4$ is \ffsc\ with an $N$-minor, then either $M\ba 1,4$ has an $N$-minor  and   $M\ba 1$ is \ffsc,  or one of the following 
holds for some  $(M_0,N_0)$ in $\{(M,N), (M^*,N^*)\}$.  
\begin{itemize}
\item[(i)] $M_0$ has a proper \ifc\ minor $M'$ such that $M'$ has an $N_0$-minor and either $|E(M)|-|E(M')|\leq 3$,  or {  $|E(M) - E(M')| = 4$ and $M'$ is obtained from $M_0$ by deleting the central cocircuit of an augmented $4$-wheel; or } 
\item[(ii)] $M_0$ contains an open rotor chain, a ladder structure, or a ring of bowties that can be trimmed to obtain an \ifc\ matroid with an $N_0$-minor; or 
\item[(iii)] $M_0$ contains an open quartic ladder and an \ifc\ matroid with an $N_0$-minor can be obtained by a mixed ladder move; or 
\item[(iv)] $M_0$ contains an enhanced quartic ladder from which an \ifc\ minor of $M_0$ with an $N_0$-minor can be obtained by an enhanced-ladder move.  
\end{itemize}
\end{corollary}

\section{Preliminaries}
\label{prelim}

In this section, we give some basic definitions mainly relating to matroid connectivity. 
Let $M$ and $N$ be matroids. We shall sometimes write $N\preceq M$ to indicate that $M$ has an $N$-minor, that is, a minor isomorphic to $N$. 
Now let $E$ be the ground set of $M$ and $r$ be its rank function.   
The {\it connectivity function} $\lambda_M$ of $M$ is defined on all subsets $X$ of $E$ by $\lambda_M(X) = r(X) + r(E-X) - r(M)$.  Equivalently, $\lambda_M(X) = r(X) + r^*(X) - |X|.$ We will sometimes abbreviate $\lambda_M$ as $\lambda$.
For a positive integer $k$, a subset $X$ or a partition $(X,E-X)$ of $E$ is 
{\em $k$-separating} if $\lambda_M(X)\le k-1$.  A $k$-separating partition $(X,E-X)$ of $E$ is a {\it $k$-separation} if  $|X|,|E-X|\ge k$. If $n$ is an integer exceeding one, a matroid  is {\it $n$-connected} if it has  no $k$-separations for all $k < n$.  This definition \cite{wtt} has the attractive property that a matroid is $n$-connected if and only if its dual is. Moreover, this matroid definition of $n$-connectivity is relatively compatible with the graph notion of $n$-connectivity when $n$ is $2$ or $3$. For example, when $G$ is a graph with at least four vertices and with no isolated vertices, $M(G)$ is a $3$-connected matroid if and only if $G$ is a $3$-connected simple graph. But the link between $n$-connectivity for matroids and graphs breaks down for $n \ge 4$. In particular, a $4$-connected matroid with at least six elements cannot have a triangle. Hence, for $r \ge 3$, neither $M(K_{r+1})$ nor $PG(r-1,2)$ is $4$-connected. This motivates the consideration of other types of $4$-connectivity  in which certain $3$-separations are allowed.

A matroid is \emph{internally $4$-connected} if it is $3$-connected and, whenever $(X,Y)$ is a $3$-separation, either $|X|=3$ or $|Y|=3$.  Equivalently, a $3$-connected matroid $M$ is \ifc\ if and only if, for every $3$-separation $(X,Y)$ of $M$, either $X$ or $Y$ is a triangle or a triad of $M$.  A graph $G$ without isolated vertices is {\it internally $4$-connected} if $M(G)$ is internally $4$-connected. 
 
Let $M$ be a  matroid. A subset $S$  of $E(M)$ is a {\em fan}   in $M$ if $|S| \ge 3$ and there is an ordering $(s_1,s_2,\ldots,s_n)$
of $S$ such that $\{s_1,s_2,s_3\},\{s_2,s_3,s_4\},\ldots,\{s_{n-2},s_{n-1},s_n\}$  alternate between triangles and triads.   
We call $(s_1,s_2,\ldots,s_n)$   a {\it fan ordering} of 
$S$. 
For convenience, we will often refer to the fan ordering as the fan.  
We will be mainly concerned with $4$-element and $5$-element fans. By convention, we shall always view a fan ordering of a $4$-element fan as beginning with a triangle and we shall use the term {\it $4$-fan} to refer to both the $4$-element fan and such a fan ordering of it. Moreover, we shall use the terms {\it $5$-fan} and {\it $5$-cofan} to refer to the two different types of $5$-element fan where the first contains two triangles and the second two triads. Let $(s_1,s_2,\ldots,s_n)$ be a fan ordering of a fan $S$. When $M$ is $3$-connected and $n \ge 4$, every fan ordering of $S$ has its first and last elements in $\{s_1,s_n\}$. We call these elements  the {\it ends} of the fan while the elements of $S - \{s_1,s_n\}$  are called the {\it internal elements} of the fan.  When $(s_1,s_2,s_3,s_4)$ is a $4$-fan, our convention is that $\{s_1,s_2,s_3\}$ is a triangle, and we call $s_1$  the {\it guts    element} of the fan and $s_4$ the {\it coguts    element} of the fan since $s_1 \in \cl(\{s_2,s_3,s_4\})$ and $s_4 \in \cls(\{s_1,s_2,s_3\})$.

A set $U$ in a matroid $M$ is {\em fully closed} if it is closed in both $M$ and $M^*$. Let $(X,Y)$ be a partition of $E(M)$. If $(X,Y)$ is $k$-separating in $M$ for some positive integer $k$,  and $y$ is an element of  $Y$ that is also in $\cl(X)$ or $\cl ^*(X)$, then it is well known and easily checked that $(X \cup y,Y-y)$ is $k$-separating, and we say that we have {\it moved} $y$ into $X$. More generally, $(\fcl(X),Y-\fcl(X))$ is $k$-separating in $M$. Let $n$ be an   integer exceeding one. If $M$ is $n$-connected, an $n$-separation $(U,V)$ of $M$ is {\it sequential} if $\fcl(U)$ or $\fcl(V)$ is $E(M)$. In particular, when $\fcl(U) = E(M)$, there is an ordering $(v_1,v_2,\ldots,v_m)$ of the elements of $V$ such that 
$U \cup \{v_m,v_{m-1},\ldots,v_i\}$ is $n$-separating for all $i$ in $\{1,2,\ldots,m\}$. When this occurs, the set $V$ is called {\it sequential}. Moreover, if $n \le m$, then $\{v_1,v_2,\ldots,v_n\}$ is a circuit or a cocircuit of $M$. 
A \thc\ matroid is {\it \sfc} if all of its $3$-separations are sequential. It is straightforward to check that, when $M$ is  binary, a sequential set with $3, 4$, or $5$ elements is a fan. Let $(X,Y)$ be a \ths\ of a \thc\ binary matroid $M$. We shall frequently be interested in $3$-separations that indicate that $M$ is, for example, not \ifc. We call $(X,Y)$ or $X$ a {\it \ftv} if $|Y| \ge |X| \ge 4$. Similarly, $(X,Y)$ is a {\it \ffsv} if, for each $Z$ in $\{X,Y\}$, either $|Z| \ge 5$, or $Z$ is \ns. 
We also say that $(X,Y)$ is a {\it \ffspv} if, for each $Z\in\{X,Y\}$, either $|Z|\geq 6$, or $Z$ is \ns, or $Z$ is a $5$-cofan.  
A binary matroid that has no \ffsv\ is {\it \ffsc}, as we defined in the introduction, and it is {\it \ffspc} if its has no \ffspv.

Next we note another special structure from~\cite{zhou}, which has arisen frequently in our work towards the desired splitter theorem. In an internally $4$-connected binary matroid $M$, we call 
$(\{1,2,3\},\{4,5,6\},\{7,8,9\},\{2,3,4,5\},\{5,6,7,8\},\{3,5,7\})$  a {\it quasi rotor} with {\it central triangle} $\{4,5,6\}$ and {\it central element $5$} 
if $\{1,2,3\}, \{4,5,6\},$ and $\{7,8,9\}$ are disjoint triangles in $M$ such that $\{2,3,4,5\}$ and $\{5,6,7,8\}$ are cocircuits and $\{3,5,7\}$ is a triangle.  
The next section is dedicated to results concerning bowties and quasi rotors.  

\begin{figure}[htb]
\centering
\includegraphics{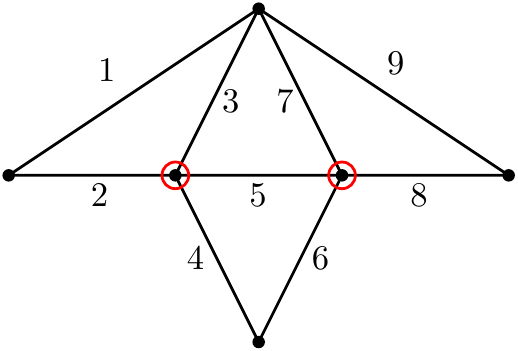}
\caption{A quasi rotor, where $\{2,3,4,5\}$ and $\{5,6,7,8\}$ are cocircuits.}
\label{rotorfig}
\end{figure}

For all non-negative integers $i$, it will be convenient  to adopt the convention throughout the paper of using $T_i$ and $D_i$ to denote, respectively, a triangle $\{a_i,b_i,c_i\}$ and a cocircuit $\{b_i,c_i,a_{i+1},b_{i+1}\}$.  
Let $M$ have $(T_0,T_1,T_2,D_0,D_1,\{c_0,b_1,a_2\})$ as a quasi rotor.  
Now $T_2$ may also be the central triangle of a quasi rotor.  
In fact, we may have a structure like one of the two depicted in Figure~\ref{rotorchainabc}.  
If $T_0,D_0,T_1,D_1,\dots ,T_n$ is a string of bowties in $M$, for some $n\geq 2$, and $M$ has the additional structure that $\{c_i,b_{i+1},a_{i+2}\}$ is a triangle for all $i$ in  $\{0,1,\dots ,n-2\}$, then we say that $((a_0,b_0,c_0),(a_1,b_1,c_1),\dots ,(a_n,b_n,c_n))$ is a {\it rotor chain}.  Clearly, deleting $a_0$ from a rotor chain gives an open rotor chain. 
Note that every  three consecutive triangles within a rotor chain have the structure of a quasi rotor;   
that is, for all $i$  in $\{0,1,\dots ,n-2\}$, the sequence $(T_i,T_{i+1},T_{i+2},D_i,D_{i+1},\{c_i,b_{i+1},a_{i+2}\})$ is a quasi rotor.  
Zhou~\cite{zhou} considered a similar structure that he called a {\it double fan of 
length $n-1$};  it  consists of all of the  elements in the rotor chain except for $a_0,b_0,b_n$, and $c_n$. 

\begin{figure}[b]
\centering
\includegraphics[scale=0.75]{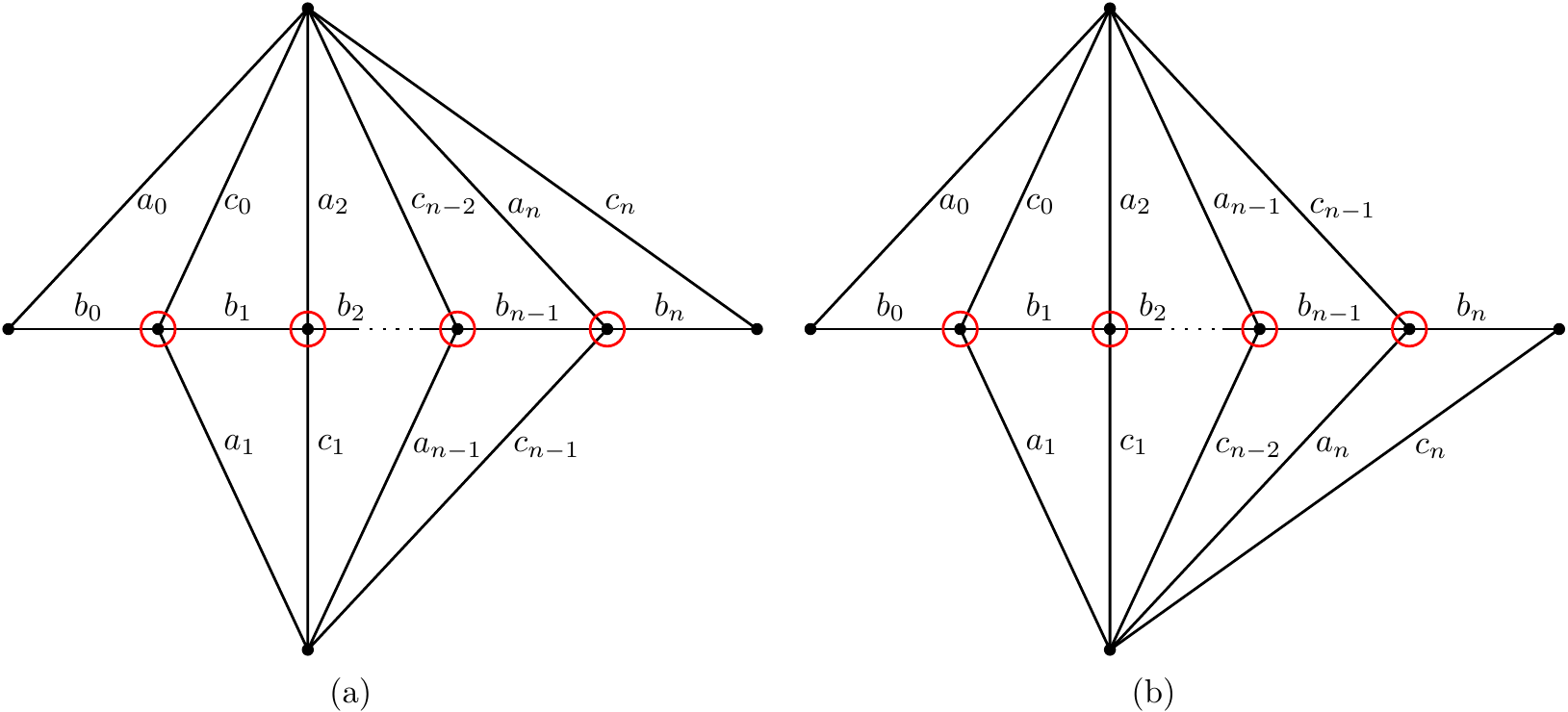}
\caption{Right-maximal rotor chain configurations.  In the case that $n$ is even, the rotor chain is depicted on the left.  If $n$ is odd, then the rotor chain has the form on the right.}
\label{rotorchainabc}
\end{figure}

If a rotor chain $((a_0,b_0,c_0),(a_1,b_1,c_1),\dots ,(a_n,b_n,c_n))$ cannot be extended to a rotor chain of the form $((a_0,b_0,c_0),(a_1,b_1,c_1),\dots ,(a_{n+1},b_{n+1},c_{n+1}))$, then we call it a {\it right-maximal rotor chain}.  

In the introduction, we defined a string of bowties. 
We say that such a string $T_0,D_0,T_1,D_1,\dots ,T_n$ is a {\it right-maximal bowtie string} in $M$ if $M$ has no triangle $\{u,v,w\}$ such that $T_0,D_0,T_1,D_1,\dots,T_n,\{x,c_n,u,v\},\{u,v,w\}$ is a bowtie string for some $x$ in $\{a_n,b_n\}$.  



{  For each positive integer $n\ge 3$, let $M_n$ be the binary matroid that is obtained from a wheel of rank $n$ by adding a single element $\gamma$ such that if $B$ is the basis of $M({\mathcal W}_n)$ consisting of the set of spokes of the wheel, then the fundamental circuit $C(\gamma, B)$ is $B \cup \gamma$. Observe that $M_3 \cong F_7$ and $M_4 \cong M^*(K_{3,3})$. Assume that the spokes of $M({\mathcal W}_n)$, in cyclic order, are $x_1,x_2,\ldots,x_n$ and that $\{x_i,y_i,x_{i+1}\}$ is a triangle of $M({\mathcal W}_n)$ for all $i$ in $\{1,2,\dots,n\}$ where we interpret all subscripts modulo $n$. Then, for all $i$ in $\{1,2,\dots,n\}$, the set $\{y_{i-1},x_i,y_i\}$ is a triad of $M({\mathcal W}_n)$  and $\{\gamma,y_{i-1},x_i,y_i\}$ is a cocircuit of $M_n$. It is straightforward to check that $M_n$ is \ifc. Kingan and Lemos~\cite{kinglem} denote $M_n$ by $F_{2n+1}$. 
When $n$ is odd, which is the case that will be of most interest to us here, $M_n^*$ is isomorphic to what Mayhew, Royle, and Whittle~\cite{mayroywhi} call the {\it rank-$(n+1)$ triadic M\"{o}bius matroid}, $\Upsilon_{n+1}$.} 

\section{Some results for bowties and quasi rotors}
\label{bqr}

In this section, we gather together a number of results that will be needed to prove the main theorem beginning with Lemma 2.5 from \cite{cmoIV} and Lemmas 4.1 and 4.2 from \cite{cmoVI}.  {  The second of these  will often be used implicitly without reference.}

\begin{lemma}
\label{airplane}
Let $M$ and $N$  be  \ifc~ binary matroids  and $\{e,f,g\}$ be a triangle of $M$ such that $N\preceq M\ba e$ and $M\ba e$ is \ffsc. Suppose $|E(N)| \ge 7$ and $M\ba e$ has $(1,2,3,4)$ as a $4$-fan. 
Then either 
\begin{itemize}
\item[(i)]  $N\preceq M\ba e\ba 1$; or 
\item[(ii)] $N\preceq M\ba e/4$ and $M\ba e/4$ is \ffsc.
\end{itemize}
\end{lemma}

\begin{lemma}
\label{2.2} 
Let $N$ be an \ifc\  matroid having at least seven elements and $M$ be a binary matroid with an $N$-minor. If $(s_1,s_2,s_3,s_4)$ is a $4$-fan in $M$, then $M\ba s_1$ or $M/s_4$ has an $N$-minor. If $(s_1,s_2,s_3,s_4,s_5)$ is a $5$-fan in $M$, then either $M\ba s_1,s_5$ has an $N$-minor, or both $M\ba s_1/s_2$ and $M\ba s_5/s_4$ have $N$-minors. 
\end{lemma}

\begin{lemma}
\label{bowwow}
Let $M$ be  an \ifc\ matroid having at least ten elements. If $(\{1,2,3\}, \{4,5,6\}, \{2,3,4,5\})$ is a bowtie in $M$, then $\{2,3,4,5\}$ is the unique $4$-cocircuit of $M$ that meets both $\{1,2,3\}$ and  $\{4,5,6\}$.
\end{lemma}

When dealing with bowtie structures, we will repeatedly use the following   result 
\cite[Lemma 4.3]{cmoVI}, a modification of \cite[Lemma 6.3]{cmochain} .  

\begin{lemma}
\label{6.3rsv}
Let $(\{1,2,3\}, \{4,5,6\}, \{2,3,4,5\})$ be a bowtie in an \ifc\ binary matroid $M$ with $|E(M)| \ge 13$. Then $M\ba 6$ is \ffsc\ unless $\{4,5,6\}$ is the central triangle of 
a quasi rotor $(\{1,2,3\},\{4,5,6\},\{7,8,9\},\{2,3,4,5\},\{y,6,7,8\},\{x,y,7\})$ for some $x$ in $\{2,3\}$ and some $y$ in $\{4,5\}$.  
In addition, when $M\ba 6$ is \ffsc, one of the following holds. 
\begin{itemize}
\item[(i)] $M\ba 6$ is \ifc; or
\item[(ii)] $M$ has a triangle $\{7,8,9\}$ disjoint from $\{1,2,3,4,5,6\}$ such that $(\{4,5,6\}, \{7,8,9\}, \{a,6,7,8\})$ is a bowtie for some $a$ in $\{4,5\}$; or
\item[(iii)] every \ftv\ of $M\ba 6$ is a $4$-fan of the form $(u,v,w,x)$, where $M$ has a triangle $\{u,v,w\}$ and a cocircuit $\{v,w,x,6\}$ for some $u$ and  $v$ in $\{2,3\}$ and $\{4,5\}$, respectively, and $|\{1,2,3,4,5,6,w,x\}|=8$; or
\item[(iv)] $M\ba 1$ is \ifc\ and $M$ has a triangle $\{1,7,8\}$ and a cocircuit $\{a,6,7,8\}$ where $|\{1,2,3,4,5,6,7,8\}| = 8$ and  $a \in \{4,5\}$. 
\end{itemize}
\end{lemma}

In Theorem~\ref{killcasek}, we dealt with the case when $M$ has a bowtie $(\{a_0,b_0,c_0\},\{a_1,b_1,c_1\},\{b_0,c_0,a_1,b_1\})$ such that $M\ba c_0$ is \ffsc\ with an $N$-minor and $M\ba c_1$ has an $N$-minor but is not \ffsc.  
We will therefore use the following hypothesis in the next lemma, and throughout this paper.  

\begin{hypothesis} 
If, for $(M_1,N_1)\in \{(M,N),(M^*,N^*)\}$, the matroid $M_1$ has a bowtie $(\{a_0,b_0,c_0\},\{a_1,b_1,c_1\},\{b_0,c_0,a_1,b_1\})$, where $M_1\ba c_0$ is \ffsc\ and $M_1\ba c_0,c_1$ has an $N_1$-minor, then $M_1\ba c_1$ is \ffsc.  
\end{hypothesis}

The next lemma is related to the previous lemma.  
We begin with the same structure in $M$, a bowtie, but we add the additional consideration of preserving an $N$-minor, and we eliminate one outcome by adding Hypothesis~VII.  

\begin{lemma}
\label{ILK}
Let $(\{1,2,3\}, \{4,5,6\}, \{2,3,4,5\})$ be a bowtie in an \ifc\ binary matroid $M$ with $|E(M)| \ge 13$.  Let $N$ be an \ifc\ minor of $M$ having at least seven elements. Suppose that $M\ba 4$ is \ffsc, that $N\preceq M\ba 1,4$, and that Hypothesis~VII holds. Then $M\ba 1$ is \ffsc\ with an $N$-minor and 
\begin{itemize}
\item[(i)] $M\ba 1$ is \ifc; or 
\item[(ii)] $M$ has a triangle $\{7,8,9\}$ such that $(\{1,2,3\}, \{7,8,9\}, \{7,8,1,s\})$ is a bowtie for some $s$ in $\{2,3\}$ and $|\{1,2,\dots ,9\}|=9$; or 
\item[(iii)] every \ftv\ of $M\ba 1$ is a $4$-fan of the form $(4,t,7,8)$, for some $t$ in $\{2,3\}$ where $\{1,2,3,4,5,6,7,8\}| = 8$; or 
\item[(iv)]  $M\ba 6$ is \ifc\ with an $N$-minor.
\end{itemize}
\end{lemma}

\begin{proof}  
By Hypothesis~VII, we may assume that $M\ba 1$ is \ffsc.  
Now, if part (i) or (ii)  of Lemma~\ref{6.3rsv} holds, then (i) or (ii)  of the current lemma holds. 
Moreover, if (iii) of Lemma~\ref{6.3rsv} holds, then, since $M\ba 4$ is \ffsc, part (iii) of the current lemma holds. 
Thus we may assume that (iv) of  Lemma~\ref{6.3rsv} holds.  Then, by symmetry, we may assume that $M$ has a triangle $\{6,7,8\}$ and a cocircuit $\{1,3,7,8\}$ where 
$|\{1,2,3,4,5,6,7,8\}| = 8$ and $M\ba 6$ is \ifc. 
We may also assume that $M\ba 6$ has no $N$-minor otherwise (iv) holds.  

Now $M\ba 1,4$ is $3$-connected having $(6,7,8,3)$ as a $4$-fan and  $N \preceq M\ba 1,4$.  As $N \not \preceq M\ba 6$, it follows, by Lemma~\ref{2.2}, that $N\preceq M\ba 1,4/3$. But $M\ba 1,4/3 \cong M\ba 2,4/3 \cong M\ba 2,4/5 \cong M\ba 2/5\ba 6$. Hence $N \preceq M\ba 6$; a \cn. 
\end{proof} 

The next two results \cite[Lemmas 5.2 and 5.7]{cmoVI} are helpful when dealing with bowtie strings.

\begin{lemma}
\label{stringswitch}
Let $T_0,D_0,T_1,D_1,\dots, T_n$ be a string of bowties in a  matroid $M$.  Then
\begin{align*}
M\ba c_0,c_1,\dots ,c_n/b_n &\cong M\ba a_0,a_1,\dots ,a_n/b_0\\
&\cong M\ba c_0,c_1,\dots ,c_{k-1}/b_k\ba a_k,a_{k+1},\dots ,a_n\\
&\cong M\ba c_0,c_1,\dots ,c_{k-1}/b_{k-1}\ba a_k,a_{k+1},\dots ,a_n
\end{align*}
for all $k$ in $\{1,2,\dots ,n\}$.  
\end{lemma}

\begin{lemma}
\label{stringybark}
Let $M$ be a binary matroid with an \ifc\  minor $N$ where $|E(N)|\geq 7$.  
Let $T_0,D_0,T_1,D_1,\dots, T_n$ be a string of bowties in  $M$.  
Suppose $M\ba c_0,c_1$ has an $N$-minor but $M\ba c_0,c_1/b_1$ does not. Then $M\ba c_0,c_1,\dots ,c_n$ has an $N$-minor, but $M\ba c_0,c_1,\dots ,c_i/b_i$ has no $N$-minor for all $i$ in $\{1,2,\dots ,n\}$, and   $M\ba c_0,c_1,\dots ,c_j/a_j$ has no $N$-minor for all $j$ in $\{2,3,\ldots,n\}$.
\end{lemma}

In the following result, we consider a short bowtie string.  

\begin{lemma}
\label{realclaim1}
Let {  $T_0,D_0,T_1,D_1,T_2$} 
be a string of bowties in an \ifc~ binary matroid $M$.
Suppose $M\ba c_1$ is \ffsc. 
Then $M\ba c_1/b_1$ is \ffspc. Moreover, either $M\ba c_1/b_1$ is \ifc, or  $M\ba c_1/b_1$ has a $4$-fan and, whenever $(\alpha,\beta,\gamma,\delta)$ is such a $4$-fan, $a_1 \in \{\beta,\gamma,\delta\}$, and $\{\beta,\gamma,\delta,c_1\}$ is a cocircuit of $M$.
\end{lemma}

\begin{proof}
As $M\ba c_1$ is \ffsc\ having $(c_2,b_2,a_2,b_1)$ as a $4$-fan,  $M\ba c_1/b_1$ is \thc.  
Suppose $M\ba c_1/b_1$ has a  $(4,5,S,+)$-violator $(U,V)$. Then, without loss of generality, $|T_2\cap U|\geq 2$. It is not difficult to check that $(U\cup T_2 \cup b_1,V-T_2)$ is a $(4,4,S)$-violator of $M\ba c_1$; a \cn.  
We conclude that $M\ba c_1/b_1$ is \ffspc.

Suppose that $(\alpha,\beta,\gamma,\delta)$ is a $4$-fan in $M\ba c_1/b_1$ such that $a_1\notin \{\beta,\gamma,\delta\}$.  
Orthogonality with $T_1$ implies that $\{\beta,\gamma,\delta,c_1\}$ is not a cocircuit of $M$. Hence $\{\beta,\gamma,\delta\}$ is a triad of $M$.  
As $M$ is \ifc, we deduce that $\{\alpha,\beta,\gamma,b_1\}$ is a circuit of $M$, so, by \ort, $\{\alpha,\beta,\gamma\}$ meets $\{b_0,c_0,{  a_1}\}$ and $\{a_2,b_2\}$.  
Thus $\{\beta,\gamma,\delta\}$ meets {  a triangle of $M$}; a \cn.  
We conclude that $a_1\in\{\be,\ga,\de\}$.  
Since $a_1$ is in a triangle of $M$, we deduce that $\{c_1,\be,\ga,\de\}$ is a $4$-cocircuit of $M$.  
\end{proof}

We continue on this theme with the following lemma.  

\begin{lemma}
\label{claim2}
Let $M$ be an \ifc\ binary matroid 
and suppose that $M$ has $T_0,D_0,T_1,D_1,T_2$ as a string of bowties 
and that $M\ba c_1$ is \ffsc.  
Then 
\begin{itemize}
\item[(i)] $M\ba c_1/b_1$ is \ifc; or 
\item[(ii)] $M\ba c_1/b_1$ is  \ffspc\ and $M$ has a triangle $\{1,2,3\}$ that avoids $T_1$ such that $\{2,3,a_1,c_1\}$ is a cocircuit; or
\item[(iii)] $M\ba c_1/b_1$ is  \ffspc\ and $M$ has elements $d_0$ and $d_1$ such that $\{d_0,d_1\}$ avoids $T_0\cup T_1\cup T_2$ where $\{d_0,a_1,c_1,d_1\}$ is a cocircuit, and $\{d_0,a_1,s\}$ or $\{d_1,c_1,t\}$ is a triangle for some $s$ in $\{b_0,c_0\}$ or $t$  in 
$\{a_2,b_2\}$.  
\end{itemize}
\end{lemma}

\begin{proof}
Suppose (i) does not hold. 
By~\ref{realclaim1}, 
 $M\ba c_1/b_1$ has a $4$-fan, $(1,2,3,4)$, where $a_1\in\{2,3,4\}$, and $\{2,3,4,c_1\}$ is a cocircuit.  
Lemma~\ref{bowwow} implies that $\{2,3,4\}$ avoids $T_0$ and $T_2$. 
Now $M$ has $\{1,2,3\}$ or $\{1,2,3,b_1\}$ as a circuit. 
Suppose that $a_1=4$.  
If $\{1,2,3\}$ is a triangle, then (ii) holds, so we assume not.  
Then $\{1,2,3,b_1\}$ is a circuit.  
Now \ort\ implies that $\{1,2,3\}$ meets both $\{b_0,c_0\}$ and $\{a_2,b_2\}$, so $\{2,3\}$ meets $T_0$ or $T_2$; a \cn.  
We deduce that $a_1\neq 4$.  
Without loss of generality, $a_1=3$.  
If $\{1,2,a_1\}$ is a triangle in $M$, then \ort\ implies that $\{1,2\}$ meets $\{b_0,c_0\}$. Hence $1\in\{b_0,c_0\}$ and, {  relabelling $(1,2,4)$ as $(s,d_0,d_1)$, we see that} (iii) holds.  
If $\{1,2,a_1,b_1\}$ is a circuit of $M$, then \ort\ with $D_1$ implies that $\{1,2\}$ meets $\{a_2,b_2\}$. Thus $1\in\{a_2,b_2\}$ and  $\{1,2,a_1,b_1\}\btu T_1$ is $\{1,2,c_1\}$, a triangle, so, {  relabelling $(1,2,4)$ as $(t,d_1,d_0)$},   (iii) holds.  
\end{proof}

Next, we prove a stronger version of~\cite[Lemma~8.4]{cmochain}.  

\begin{lemma}
\label{3peaks}
Let $M$ be an internally $4$-connected binary matroid having 
{  $T_0,D_0,T_1,D_1,T_2$}
as a string of bowties. 
Then
\begin{itemize}
\item[(i)] $M/T_1$ is internally $4$-connected; or  
\item[(ii)] $T_1$ is the central triangle of a quasi rotor; or
\item[(iii)] {  $M\ba c_1/b_1$ is internally $4$-connected; or
\item[(iv)] $M\ba c_1/b_1$ is  \ffspc\ and $M$ has elements $d_0$ and $d_1$ such that $\{d_0,d_1\}$ avoids $T_0\cup T_1\cup T_2$, and $\{d_0,a_1,c_1,d_1\}$ is a cocircuit, and $\{d_0,a_1,s\}$ or $\{d_1,c_1,t\}$ is a triangle for some $s$ in $\{b_0,c_0\}$ or $t$ in $\{a_2,b_2\}$.}  
\end{itemize}
\end{lemma}

\begin{proof}
{  Assume the lemma does not hold.  
By Lemma~\ref{6.3rsv}, since $T_1$ is not the central triangle of a quasi rotor, $M\ba c_1$ is \ffsc.  
By Lemma~\ref{claim2}, $M$ has a triangle $\{1,2,3\}$ avoiding $T_1$ such that $\{2,3,a_1,c_1\}$ is a cocircuit.  
Lemma~\ref{bowwow} implies that $\{1,2,3\}$ avoids $D_0$ and $D_1$, and that $T_0$ and $T_2$ avoid $\{2,3\}$.  
Then (i) or (ii) holds by~\cite[Lemma~8.3]{cmochain}; a \cn.}  

%
%
\end{proof}

{  To conclude this section, we recall \cite[Lemma~4.5]{cmoVI}, which is useful for dealing with  quasi rotors.  

\begin{lemma}
\label{rotorwin} Let $M$ be an \ifc~ binary matroid having $(\{1,2,3\},\{4,5,6\},\{7,8,9\},\{2,3,4,5\},\{5,6,7,8\},\{3,5,7\})$ as a quasi rotor   
and having at least thirteen elements. Let $N$ be an \ifc~ matroid containing at least seven elements such that $M/e$ has an $N$-minor for some $e$ in $\{3,5,7\}$. Then one of $M\ba 1$, $M\ba 9$, $M\ba 1/2$,  $M\ba 9/8$, or $M\ba 3,4/5$ is \ifc\ with an $N$-minor.
\end{lemma}}

\section{Bowties and Ladders}
\label{bl}

In this section, we consider how bowties can interact with ladders.  We begin with a lemma that builds from the configuration in Figure~\ref{drossfig2}.

\begin{figure}[htb]
\center
\includegraphics[scale=0.9]{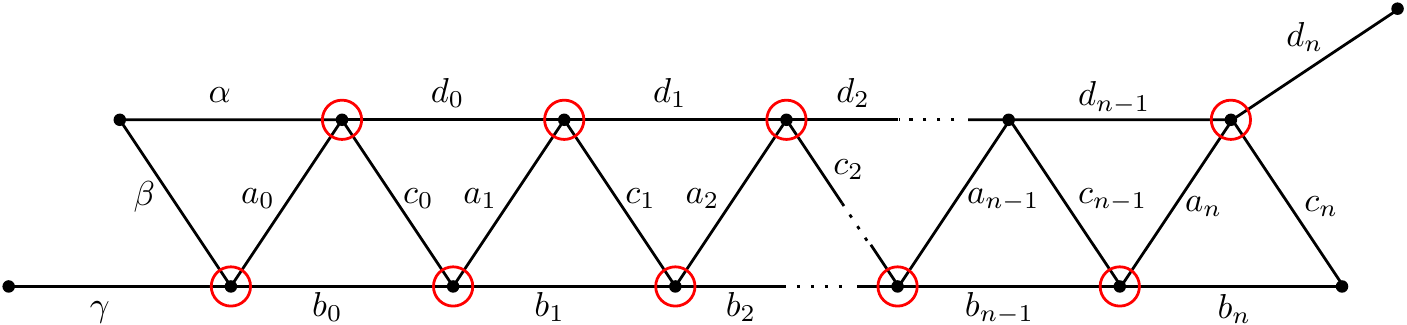}
\caption{$n\ge 1$ and $M$ has either $\{d_{n-2},a_{n-1},c_{n-1},d_{n-1}\}$ or $\{d_{n-2},a_{n-1},c_{n-1},a_n,c_n\}$ as a cocircuit,   where $d_{-1} = \alpha$ if $n = 1$.}
\label{drossfig2}
\end{figure}

\begin{lemma}
\label{moreteeth}
Let $M$ be an \ifc\ binary matroid that has at least {  thirteen} elements. 
Assume that $M$ contains the configuration shown in Figure~\ref{drossfig2} where $n\geq 1$,   all the elements shown are distinct except that $d_n$ and $\gamma$ may be equal, and, in addition to the cocircuits shown, exactly one of 
$\{d_{n-2},a_{n-1},c_{n-1},d_{n-1}\}$ and $\{d_{n-2},a_{n-1},c_{n-1},a_n,c_n\}$ is a cocircuit of $M$. 
Assume also that $M$ is not isomorphic to the cycle matroid of a quartic M\"obius ladder and that $M\ba c_n$ is \ffsc.  
Then $M\ba c_0,c_1,\ldots,c_n,\beta$ is \ffsc\ and if it has a \ftv, then one side of that \ftv\ is a $4$-fan $F$ where either 
\begin{itemize}
\item[(i)] $F$ is a $4$-fan in $M\ba c_n$ {  with $b_n$ as its coguts element}; or 
\item[(ii)] $F$ is a $4$-fan in $M\ba \beta$ {  with $\al$ as its coguts element}.  
\end{itemize}
\end{lemma}

\begin{proof} 
First we show the following. 
\begin{sublemma}
\label{all4s}
When $n = 1$, neither $\{d_0,d_1\}$ nor $\{b_0,b_1\}$ is contained in a triangle of $M$. Moreover, none of $a_1, b_0, b_1, d_0$, nor $d_1$ is in a triangle of $M\ba c_0,c_1$.
\end{sublemma}

If $\{d_{n-1},d_n\}$ is in a triangle, then $M\ba c_n$ has a $5$-fan; a \cn.  
If $\{b_0,b_1\}$ is in a triangle, then \ort\ implies that the triangle's third element is in $\{\be,\ga\}$, so $\lambda (\{\al,\be,\ga,d_0\}\cup T_0\cup {  \{a_1,b_1\}}\})\leq 2$; a \cn.  By \cite[Lemma 6.1]{cmoVI}, $a_1$ is not in a triangle of $M\ba c_0,c_1$. If $b_0$ or $b_1$ is in a triangle of $M\ba c_0,c_1$, then \ort\ implies that this triangle contains $\{b_0,b_1\}$; a \cn. Similarly, if $d_0$ or $d_1$ is in a triangle of $M\ba c_0,c_1$, then \ort\ implies that this triangle contains $\{d_0,d_1\}$; a \cn. 
We conclude that  \ref{all4s} holds. 

Next we note that 
\begin{sublemma}
\label{all4s2}
$M\ba c_0,c_1,\dots ,c_n$ is \ffsc.  
\end{sublemma}

This follows immediately from \cite[Lemma~6.5]{cmoVI} when $n > 1$. Moreover, it holds when $n= 1$ by combining   \ref{all4s} with \cite[Lemma~6.1]{cmoVI}. 
 
The matroid $M\ba c_0,c_1,\dots ,c_n$ has $(\beta,\alpha,a_0,d_0)$ or $(\beta,\alpha,a_0,a_1)$ as a $4$-fan. Thus $M\ba c_0,c_1,\ldots,c_n,\beta$ is \thc. Next we observe that $M\ba c_0,c_1,\ldots,c_n,\beta$ is \sfc. To see this, note that if  $M\ba c_0,c_1,\ldots,c_n,\beta$ has a \ns\ \ths\ $(U,V)$, then, as $\{a_0,\alpha\}$ is in a triad, we may assume that this triad is contained in $U$. Thus  $(U \cup \beta,V)$ is a \ns\ \ths\ of $M\ba c_0,c_1,\ldots,c_n$; a \cn. 

Now suppose $M\ba c_0,c_1,\ldots,c_n,\beta$ has a $4$-fan $(w_1,w_2,w_3,w_4)$. Then $M$ has a cocircuit $C^*$ such that $\{w_2,w_3,w_4\} \subsetneqq C^* \subseteq \{w_2,w_3,w_4,\beta,c_0,c_1,\ldots,c_n\}$.

\begin{sublemma}
\label{fanfair} 
 If $(w_1,w_2,w_3,w_4)$ is a $4$-fan of $M\ba c_0,c_1,\ldots,c_n$, then $w_4 = b_n$ and $\{b_n,c_n,w_2,w_3\}$ is a cocircuit of $M$, so $(w_1,w_2,w_3,w_4)$ is a $4$-fan of $M\ba c_n$.
\end{sublemma}

Suppose that this fails.  
If $n>1$, then, by (iii) of~\cite[Lemma~6.5]{cmoVI},  $w_4 = d_0$   and $a_0 \in \{w_2,w_3\}$. Moreover, $(w_1,w_2,w_3,w_4)$ is a $4$-fan of $M\ba c_0$. Thus
$\{w_2,w_3,c_0,d_0\}$ is a $4$-cocircuit of $M$ containing $\{a_0,c_0,d_0\}$. By \ort, this $4$-cocircuit contains $\alpha$ or $\beta$. Hence it is  $\{\alpha,a_0,c_0,d_0\}$. Thus $\{w_1,w_2,w_3\}$ contains $\{\alpha,a_0\}$ and so is $\{\alpha,a_0,\beta\}$; a \cn. We conclude that \ref{fanfair} holds if $n>1$. 

Now let $n=1$.  
By \ref{all4s}, neither $\{b_0,b_1\}$ nor $\{d_0,d_1\}$ is contained in a triangle of $M$. It follows by  \cite[Lemma~6.1]{cmoVI} that $w_4=b_1$.  
Now~\ref{fanfair} holds if $\{w_2,w_3,b_1,c_1\}$ is a cocircuit, so we assume that $\{w_2,w_3,b_1,c_0\}$ or $\{w_2,w_3,b_1,c_0,c_1\}$ is a cocircuit.  
Orthogonality implies that $\{w_2,w_3\}$ meets $\{d_0,a_1\}$; a \cn\ to \ref{all4s}. 
Thus~\ref{fanfair} holds.  

We may now assume that $(w_1,w_2,w_3,w_4)$ is not a $4$-fan of $M\ba c_0,c_1,\ldots,c_n$. Then $\beta \in C^*$. Thus $\{\alpha,a_0\}$ meets $\{w_2,w_3,w_4\}$. 
Next we show that 
\begin{sublemma}
\label{a1no} 
$a_0 \not \in \{w_1,w_2,w_3,w_4\}$.
\end{sublemma}

First we show that $a_0 \not\in \{w_1,w_2,w_3\}$. Assume the contrary.  Let $n =1$. Then, by \ort, $\{w_1,w_2,w_3\}$ meets  $\{\alpha,d_0\}$ or $\{\alpha,a_1\}$. Thus $d_0$ or $a_1$ is in a triangle of $M\ba c_0,c_1$; a \cn\ to \ref{all4s}. 
Hence we may assume that $n \ge 2$. 
Then, by~\cite[Lemma~6.3]{cmoVI}, the triangle $\{w_1,w_2,w_3\}$ of $M\ba c_0,c_1,\ldots,c_n$ meets $\{a_0,b_0,d_0,a_1,b_1,d_1,\ldots,a_n,b_n,d_n\}$ in $\{a_0\}$ or $\{a_0,d_{n-1},d_n\}$. By \ort, $\{a_0,d_{n-1},d_n\}$ is not a triangle. Thus 
 $\{w_1,w_2,w_3\}$ avoids $\{b_0,d_0,a_1,b_1,d_1,\ldots,a_n,b_n,d_n\}$.  
By \ort\ between $\{w_1,w_2,w_3\}$ and   both $\{\be,\ga,a_0,b_0\}$ and   $\{\al,a_0,c_0,d_0\}$, we find that $\{w_1,w_2,w_3\} = \{a_0,\al,\ga\}$. But $\{a_0,\al,\be\}$ is a triangle; a \cn.  
Hence $a_0 \not \in \{w_1,w_2,w_3\}$.

Suppose now that $a_0 = w_4$. 
Orthogonality between $C^*$ and the circuit $\{a_0,b_0,d_0,a_1\}$ implies that $\{w_2,w_3\}$ meets $\{b_0,d_0,a_1\}$. Thus, by \cite[Lemma~6.3]{cmoVI}, $n=1$.  But now we have a \cn\ to \ref{all4s}. 
Thus \ref{a1no} holds. 

We now know that $\alpha \in \{w_2,w_3,w_4\}$. Suppose $\al \in \{w_2,w_3\}$. 
If $\{\al,a_0,c_0,d_0\}$ is a cocircuit, then, by \ort, $\{a_0,d_0\}$ meets $\{w_1,w_2,w_3\}$. By~\cite[Lemma~6.3]{cmoVI} and \ref{a1no}, $n=1$ and $d_0 \in \{w_1,w_2,w_3\}$; a \cn\ to \ref{all4s}. 
We deduce that $\{\al,a_0,c_0,a_1,c_1\}$ is a cocircuit of $M$, so $n = 1$. Then \ort\ implies that $\{w_1,w_2,w_3\}$ meets $\{a_0,{  a_1}\}$.  
Thus, by \ref{a1no} and \ref{all4s}, we have a \cn. 
Hence $\al = w_4$. 
  
Now suppose that $c_i \in C^*$ for some $i$  in $\{0,1,\ldots,n\}$. Then $\{w_2,w_3\}$ meets $\{a_i,b_i\}$.  Thus, by~\cite[Lemma~6.3]{cmoVI},  if $n \ge 2$, then $i=0$ and $a_0\in\{w_2,w_3\}$; a \cn\ to~\ref{a1no}.  Moreover, if $n = 1$, then one of $a_0,b_0,a_1$, or $b_1$ is in $\{w_2,w_3\}$; a \cn\ to  \ref{a1no} or \ref{all4s}. We conclude that $C^*$ avoids $\{c_0,c_1,\ldots,c_n\}$, so $C^* = \{w_2,w_3,{  \al},\beta\}$, and $(w_1,w_2,w_3,{  \al})$ is a $4$-fan of $M\ba \be$. 
\end{proof}

\begin{figure}[htb]
\centering
\includegraphics[scale=1.]{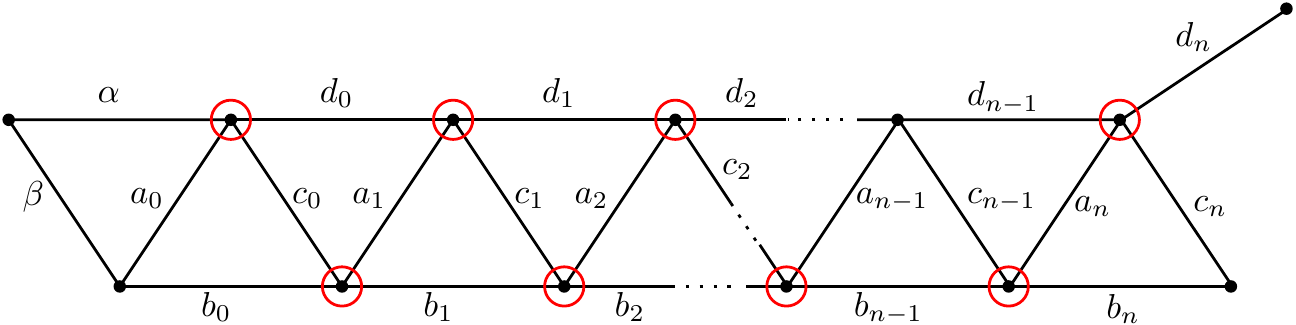}
\caption{Either $\{d_{n-2},a_{n-1},c_{n-1},d_{n-1}\}$ or $\{d_{n-2},a_{n-1},c_{n-1},a_n,c_n\}$ is a cocircuit, {  where $d_{-1}=\al$.}}
\label{drossfigiinouv}
\end{figure}

Beginning with the next lemma and for the rest of the paper, we shall start abbreviating how we refer to the following four outcomes in the main theorem.  
\begin{itemize}
\item[(i)] $M$ has a proper minor $M'$ such that $|E(M)|-|E(M')|\leq 3$ and $M'$ is \ifc\ with an $N$-minor; or 
\item[(ii)] $M$ contains an open rotor chain, a ladder structure, or a ring of bowties that can be trimmed to obtain an \ifc\ matroid with an $N$-minor; 
\item[(iii)] $M$ contains {  an open quartic ladder from which an \ifc\ minor of $M$ with an $N$-minor can be obtained by a mixed ladder move};
\item[(iv)] $M$ contains an enhanced quartic ladder from which an \ifc\ minor of $M$ with an $N$-minor can be obtained by an enhanced-ladder move.
\end{itemize}
When (i) or (iv)  holds, we say, respectively, that $M$ has a {\it quick win} or an {\it enhanced-ladder win}.  
When trimming an open rotor chain, a ladder structure, or a ring of bowties in $M$ produces an \ifc\ matroid with an $N$-minor, we say, respectively, that $M$ has an 
{\it  open-rotor-chain win}, a {\it ladder win}, or a {\it bowtie-ring win}.  
When (iii) holds, we say that $M$ has a {\it mixed ladder win}.  

%


\section{An outline of the proof of the main theorem}
\label{out}

Since the proof of the main theorem is long, we give  an outline of it in this section. By hypothesis, $M$ and $N$ are \ifc\ binary matroids and $M$ has a bowtie $(\{1,2,3\}, \{4,5,6\}, \{2,3,4,5\})$ where $M\ba 4$ is \ffsc\ with an $N$-minor, and $M\ba 1,4$ has no $N$-minor. We may assume that $M\ba 6$ is \ffsc\ otherwise the theorem holds by Theorem~\ref{killcasek}. The one result in this section, Lemma~\ref{ABCDE}, shows that either we get a quick win, or $M$ contains one of  configurations (A), (B), and $C$ in Figure~\ref{ABCDEfig}. In Section~\ref{Cconfig}, we treat the case when $M$ contains configuration (C) noting first that, by Lemma~\ref{airplane}, $M\ba 4/5$ is \ffsc\ with an $N$-minor and with $(a,b,c,6)$ as a $4$-fan. Thus $M\ba 4/5,6$ or $M\ba 4/5\ba a$ has an $N$-minor. These two possibilities are dealt with in Lemmas~\ref{ccrider} and \ref{killtobler2}, respectively. 

In Section~\ref{Aconfigsect}, we deal with  the case when $M$ contains configuration (A). 
{  First we prove a technical lemma detailing the possible structures surrounding a right-maximal bowtie chain in $M$ that is also a right-maximal bowtie chain in $M'$, a minor of $M$.  
Then} 
we show in Lemma~\ref{AconfigBOOM} that we {  obtain our result.}  

The results of Sections~\ref{Cconfig} and \ref{Aconfigsect} mean that we can assume that $M$ contains neither of configurations (C) or (A). It remains to consider when $M$ contains configuration (B) from Figure~\ref{ABCDEfig}. This is done in Section~\ref{Bconfig}. Finally, in Section~\ref{pomt},  the parts already proved are combined to complete the proof of the main theorem.

\begin{figure}[htb]
\center
\includegraphics[width=5in]{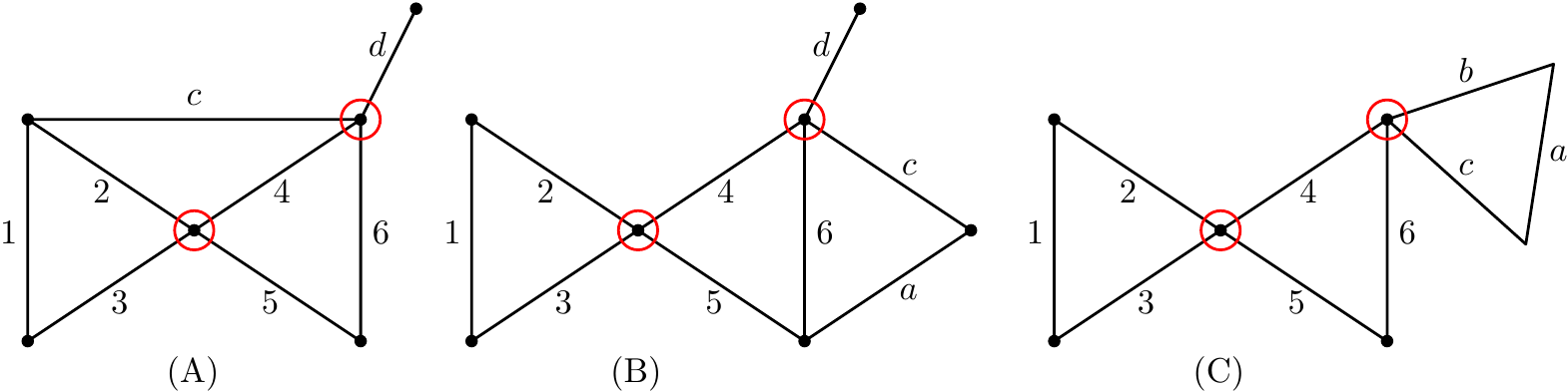}
\caption{In each structure, we view the labels on $2$ and $3$ as being interchangeable.  The elements in each part are distinct except that $a$ may equal $1$ in (B) and (C).}
\label{ABCDEfig}
\end{figure}

We now show that $M$ does indeed contain one of the three structures in Figure~\ref{ABCDEfig}.

\begin{lemma}
\label{ABCDE} 
Let $(\{1,2,3\}, \{4,5,6\}, \{2,3,4,5\})$ be a bowtie in an \ifc\ binary matroid $M$ with $|E(M)| \ge 13$.  Let $N$ be an \ifc\ binary matroid having at least seven elements. Suppose that $M\ba 4$ is \ffsc\ with an $N$-minor and that $N\not \preceq M\ba 1,4$. Then either $M$ has an \ifc\ minor $M'$ with an $N$-minor such that $1 \le |E(M) - E(M')| \le 2$, or, up to switching the labels on the elements $2$ and $3$, the matroid $M$ contains one of the configurations shown in Figure~\ref{ABCDEfig}, the deletion $M\ba 6$ is \ffsc, and $\{4,5,6\}$ is the only triangle in $M$ containing $5$. Moreover, in each of (A), (B), and (C), the elements shown are distinct except that, in (B) and (C), it is possible that $a=1$.
\end{lemma}

\begin{proof}
Since $N\not \preceq M\ba 1,4$, it follows by Lemma~\ref{airplane} that $M\ba 4/5$  is \ffsc\ with an $N$-minor. 
Thus $M\ba 6/5$  is \ffsc\ with an $N$-minor. We may assume that $M\ba 6/5$ is not \ifc\ otherwise the lemma holds. 
If $5$ is in a triangle $T$ other than $\{4,5,6\}$, then $M\ba 4/5$ has $T-5$ as a circuit; a \cn.  
Thus $\{4,5,6\}$ is the only triangle in $M$ containing $5$.  

With a view to using Lemma~\ref{6.3rsv}, we now consider $M\ba 6$.  
First suppose that  $\{4,5,6\}$ is the central triangle of a quasi rotor $(\{1,2,3\},\{4,5,6\},\linebreak \{7,8,9\},\{2,3,4,5\},\{y,6,7,8\},\{x,y,7\})$ for some $x$ in $\{2,3\}$ and some $y\in\{4,5\}$.  
Since $M\ba 4$ is \ffsc, we deduce that $y=4$.  
Then $M\ba 4/5$ has $(x,6,7,8,9)$ as a $5$-fan, a \cn.  
We conclude that $\{4,5,6\}$ is not the central triangle of such a quasi rotor.  
Then Lemma~\ref{6.3rsv} implies that either $M\ba 6$ is \ifc, and the lemma holds; or $M\ba 6$ is \ffsc\ but not \ifc. Moreover, in the latter case, one of the following holds.
\begin{itemize}
\item[(i)] $M$ has $\{a,b,c\}$ as a triangle and $\{b,c,d,6\}$ as a cocircuit for some $a$ in 
$\{2,3\}$ and $b$ in $\{4,5\}$, where $|\{1,2,\dots ,6,c,d\}|=8$; or 
\item[(ii)] $M$ contains the structure in Figure~\ref{ABCDEfig}(C) where the elements are all distinct except that $a$ may be the same as $1$; or
\item[(iii)] $M$ has a triangle $\{7,8,9\}$ and a cocircuit $\{5,6,7,8\}$ where the elements are all distinct except that $1$ may be the same as $9$.  
\end{itemize}

To see this, observe first that (i) above occurs when (iii) of Lemma~\ref{6.3rsv} holds. On the other hand, outcomes (ii) and (iv) of Lemma~\ref{6.3rsv} have been combined into (ii) and (iii) above with the separation between the latter being determined by the relative  placement of the elements $4$ and $5$. 

If (i) holds, then we know that $b=4$, otherwise $\{a,5,c\}$ is a triangle in $M$ that contains $5$ but is not  $\{4,5,6\}$; a \cn.  
Thus, up to switching the labels on $2$ and $3$, if (i) holds, then $M$ contains the structure in Figure~\ref{ABCDEfig}(A) where all of the elements shown are distinct.  

Since (ii) yields (C) in Figure~\ref{ABCDEfig}, we may now assume that (iii) holds. 
Recall that $M\ba 6/5$ is \ffsc\ but not \ifc.  
Thus $M\ba 6/5$ has a $4$-fan $(a,b,c,d)$. 
Lemma~\ref{realclaim1} implies that $4\in\{b,c,d\}$, and $\{b,c,d,6\}$ is a cocircuit of $M$.  
By symmetry, we may assume that $4 = b$ or $4 = d$.  

Assume first that $4 = b$. Then $M$ has $\{4,c,d,6\}$ as a cocircuit, and Lemma~\ref{bowwow} implies that $\{c,d\}$ avoids $\{1,2,3\}$ so $|\{1,2,\dots, 6,c,d\}|=8$, and $M$  has $\{a,4,c\}$ or $\{5,a,4,c\}$ as a circuit. In the first case, \ort\ with $\{2,3,4,5\}$ implies that $a\in\{2,3\}$, so, up to switching the labels of $2$ and $3$, the structure (A) in Figure~\ref{ABCDEfig} occurs, where all of the elements are distinct. If $\{5,a,4,c\}$ is a circuit of $M$, then $M$ also has $\{a,c,6\}$ as a circuit, so $M$ contains the structure in Figure~\ref{ABCDEfig}(B) where all of the elements are distinct except that $a$ may be a repeated element.  
Certainly $a \not\in \{5,6,b,c,d\}$. Hence, 
by \ort, either $a$ is distinct from the other elements in (B), or $a=1$.  

We may now assume that $4 = d$. Then $\{b,c,4,6\}$ is a cocircuit of $M$.  
Lemma~\ref{bowwow} implies that $\{b,c\}$ avoids $\{1,2,3,7,8,9\}$.  
Either $\{a,b,c\}$ or $\{a,b,c,5\}$ is a circuit of $M$.  
In the first case, (C) in Figure~\ref{ABCDEfig} occurs and \ort\ implies that all of the elements are distinct except that $a$ may be $1$. Now suppose that $M$ has  $\{a,b,c,5\}$ as a circuit. 
Orthogonality between this circuit and $\{2,3,4,5\}$ implies that $\{a,b,c\}$ meets $\{2,3\}$.  
By symmetry between $2$ and $3$, we may assume that $3 \in \{a,b,c\}$. As  $\{b,c\}$ avoids $\{1,2,3,7,8,9\}$, we deduce that $3=a$. In (iii), we know that $\{5,6,7,8\}$ is a cocircuit. 
Orthogonality between this cocircuit and the circuit $\{3,b,c,5\}$  implies that $\{3,b,c\}$ meets $\{6,7,8\}$; a \cn.  
\end{proof}

We reiterate here that, in each of configurations (A), (B), and (C) in Figure~\ref{ABCDEfig}, while the labels of $1,4,5$, and $6$ are fixed, we allow the labels on $2$ and $3$ to be interchanged without regarding the resulting structures as being different.


\section{Configuration (C)}
\label{Cconfig}

In this section, we treat  the case when $M$ contains the configuration in Figure~\ref{ABCDEfig}(C). We arrived at this configuration by assuming  that $N \not \preceq M\ba 1,4$. 
Thus, by Lemma~\ref{airplane}, $M\ba 4/5$ is \ffsc\ with an $N$-minor. Since $M\ba 4/5$ has $(a,b,c,6)$ as a $4$-fan, by Lemma~\ref{2.2}, $N\preceq M\ba 4/5/6$ or $N \preceq M\ba 4/5 \ba a$. The next lemma deals with the first of these cases.

\begin{lemma}
\label{ccrider}
Let $M$ and $N$ be binary \ifc\ matroids such that $|E(M)|\geq 13$ and $|E(N)|\geq 7$.  
Suppose that 
$M$ contains structure (C) in Figure~\ref{ABCDEfig}, where $M\ba 4$ is \ffsc\ with an $N$-minor and $N\not \preceq M\ba 1,4$. If $M\ba 4/5,6$ has an $N$-minor, then $M$ has a quick win.  
\end{lemma}

\begin{proof} We observe that $M\ba 4/5,6 \cong M/4/5/6$. Now apply Lemma~\ref{3peaks}. If (i) or (iii) of that lemma holds, then the required result is immediate. If (ii) holds, then, as $M/e$ has an $N$-minor for all $e$ in $\{4,5,6\}$, it follows by 
\cite[Lemma~4.5]{cmoVI} that the lemma holds.  
Finally, suppose that (iv) holds.  
Then either $M$ has a triangle containing $5$ and a member of $\{2,3\}$; or $M$ has a triangle containing $6$ and a member of $\{b,c\}$. In each case, we obtain a \cn\ to the fact that $M\ba 4$ is \ffsc.
\end{proof}

\begin{figure}[htb]
\center
\includegraphics{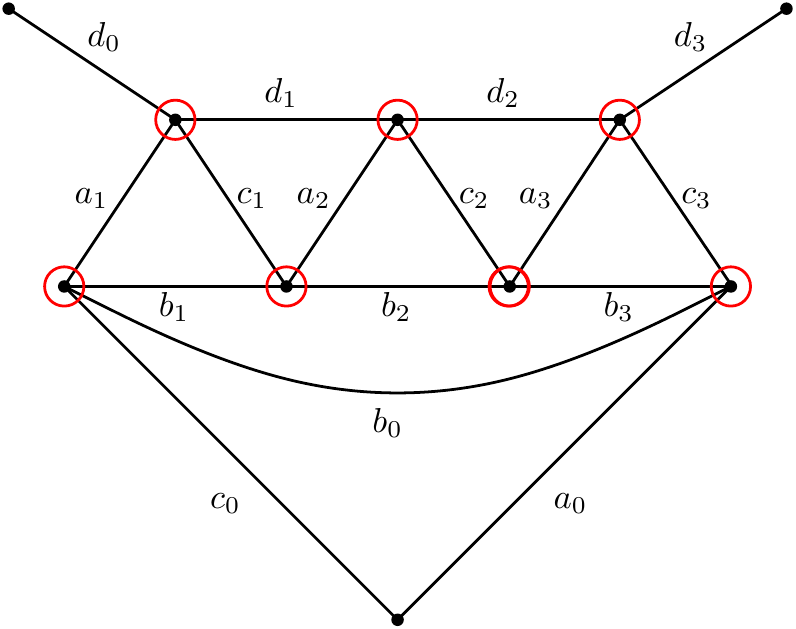}
\caption{All of the elements are distinct and $\{b_0,b_1,b_2,b_3\}$ is a circuit of $M$.}
\label{diamond0}
\end{figure}

The next lemma  concerns    the structure in Figure~\ref{diamond0}, which arises in Lemma~\ref{rainbow}.   

\begin{lemma}
\label{nodiamonds}
Suppose $M$ is an \ifc\ binary matroid. 
If $M$ contains the structure in Figure~\ref{diamond0}, where all of the elements are distinct, then $M$ is the cycle matroid of a $16$-element quartic planar ladder having $\{a_1,c_0,d_0\}$ and $\{a_0,c_3,d_3\}$ as triangles.  
\end{lemma}
\begin{proof}
Clearly $|E(M)| \ge 16$. 
Moreover,  $T_0\cup T_1\cup T_2\cup T_3\cup \{d_1,d_2\}$ has rank at most seven and contains at least five cocircuits, none of which is the symmetric difference of any others.  
Thus this $14$-element set is $3$-separating, so $|E(M)|\leq 17$, and has rank equal to seven.  
Suppose $|E(M)|=17$. Then $M$ has $\{d_0,d_3,e\}$ as a triad  for some element $e$ that is not shown in Figure~\ref{diamond0}.  
Then the symmetric difference of all of the vertex cocircuits and $\{d_0,d_1,e\}$ is $\{e,a_0,c_0\}$, so $M$ is not \ifc, a \cn.  
We conclude that $|E(M)|=16$. As $r(M)=7$, we see that   
 $\{a_0,c_0,a_1,b_1,a_2,b_2,c_3\}$ is a basis $B$ of $M$.  
By \ort\ with the vertex cocircuits in Figure~\ref{diamond0}, we deduce that the fundamental circuit $C(d_0,B)$ for $d_0$ is $\{d_0,a_1,c_0\}$. Similarly, $C(d_3,B)=\{d_3,a_0,c_3\}$.  
The lemma now follows because $M$ is binary and so is determined by its fundamental circuits with respect to $B$. 
\end{proof}

\begin{figure}[htb]
\center
\includegraphics{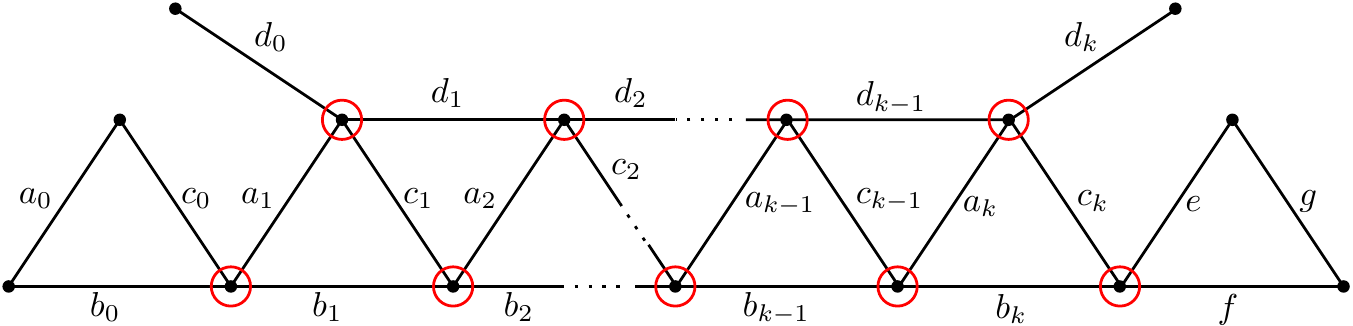}
\caption{All of the elements are distinct except $a_0$ may be $g$, or $T_0$ may be $\{e,f,g\}$. } 
\label{rainbowbrite}
\end{figure}

Next we deal with a structure that  leads to a mixed ladder win.  

\begin{lemma}
\label{rainbow}
Let $M$ be   an \ifc\ binary matroid having at least fifteen elements.  
Suppose that $M$ contains the structure in Figure~\ref{rainbowbrite}, where $k\geq 2$ and all of the elements are distinct except that $a_0$ may be $g$, or $T_0$ may be $\{e,f,g\}$, {  or $d_0$ may be $d_k$}.  
\begin{itemize}
\item[(i)] The set $\{b_0,c_0\} \neq \{e,f\}$ {  and $d_0\neq d_k$}. 
\item[(ii)] If $M\ba c_i$ is \ffsc\ for all $i$  in $\{1,2,\dots ,k\}$, then 
\begin{itemize}
\item[(a)] $\{d_0,a_1\}$ is contained in a triangle; or
\item[(b)] $\{c_k,d_k\}$ is contained in a triangle; or
\item[(c)] $M\ba c_1,c_2,\dots ,c_k/b_k$ is \ifc.  
\end{itemize}
\end{itemize}
\end{lemma}

\begin{proof} Let $X = T_1 \cup T_2 \cup \dots \cup T_k \cup \{d_1,d_2,\dots,d_{k-1}\}$. Then $E(M) - X$ contains $T_0 \cup d_0$. To establish (i), it suffices  to show that $\lambda(X) \le 2$ if $\{b_0,c_0\} = \{e,f\}$ or if $d_0 = d_k$. In the first case, $D_0 \btu  \{b_k,c_k,e,f\} = \{a_1,b_1,a_k,c_k\}$, which is a cocircuit contained in $X$. In the second case, $\{d_0,a_1,c_1,d_1\} \btu \{d_{k-1},a_k,c_k,d_k\}$ is also a cocircuit contained in $X$. In each case, one easily checks that $\lambda(X) \le 2$, so (i) holds.

Next we note that, by \cite[Lemma~5.3]{cmoVI}, $M\ba c_1,c_2,\dots ,c_k$ is \thc\ unless this matroid has a $1$- or $2$-element cocircuit $D^*$ that contains $a_j$ or $b_j$ for some $j$ in $\{3,4,\dots,k\}$. Consider the exceptional case. Then $M$ has a cocircuit 
$C^*$ such that $D^* \subsetneqq C^* \subseteq D^* \cup \{c_1,c_2,\dots,c_k\}$. For $i \neq j$, \ort\ implies that $c_i \in C^*$ if and only if $|D^* \cap \{a_i,b_i\}| = 1$. As $C^*$ meets $T_j$, it must have at least four elements. Hence $C^*= \{x,c_i,y,c_j\}$ for some $x$ in $\{a_i,b_i\}$ and $y$ in $\{a_j,b_j\}$. Let $s$ be the smaller of $i$ and $j$.  Then $\{i,j\} = \{s,s+1\}$ and $a_{s+1}\in C^*$ otherwise 
we get  a \cn\ to \ort\ between $C^*$ and $\{c_s,d_s,a_{s+1}\}$. Thus $C^*$ is a $4$-cocircuit meeting $T_s$ and $T_{s+1}$ and containing $\{a_{s+1},c_{s+1}\}$. This contradiction to Lemma~\ref{bowwow} implies that $M\ba c_1,c_2,\dots ,c_k$ is indeed \thc.

Clearly $M\ba c_1,c_2,\dots ,c_k$ has $b_k$ as the coguts element of a $4$-fan.  
The  proof of the next assertion  occupies most of the rest of the proof of the lemma finishing  just before \ref{kgeq3}. 

\begin{sublemma}
\label{bkcoguts}
If $M\ba c_1,c_2,\dots ,c_k$ is \ffsc\ and has $b_k$ as the coguts element of every $4$-fan, then part (b) of the lemma holds.  
\end{sublemma}

Certainly $M\ba c_1,c_2,\dots ,c_k/b_k$ is \thc.  
Suppose $(U,V)$ is a \ns\ \ths\ of this matroid.  
Without loss of generality, $\{e,f,g\}\subseteq U$, so $(U\cup b_k,V)$ is a \ns\ \ths\ of $M\ba c_1,c_2,\dots ,c_k$; a \cn.  
Thus $M\ba c_1,c_2,\dots ,c_k/b_k$ is \sfc.  

Let $(\al,\be,\ga,\de)$ be a $4$-fan in $M\ba c_1,c_2,\dots ,c_k/b_k$.   
By the hypothesis,  $(\al,\be,\ga,\de)$ is not a $4$-fan in   $M\ba c_1,c_2,\dots ,c_k$, so $\{\al,\be,\ga,b_k\}$ is a circuit, $C$.  Next we observe that 

\begin{sublemma}
\label{seize}
$C$ contains $\{b_{k-1},b_k\}$, avoids $\{a_{k-1},a_k,c_{k-1},c_k\}$, and meets both $\{b_{k-2},c_{k-2}\}$ and $\{e,f\}$.
\end{sublemma}

The last part of \ref{seize} is an immediate consequence of \ort.  
 Suppose $a_k\in C$. Then \ort\ with the cocircuits $\{d_{k-1},a_k,c_k,d_k\}$ and $\{d_{k-2},a_{k-1},c_{k-1},d_{k-1}\}$ implies that the last element in $C$ is $d_k$, so $C$ contains $\{a_k,d_k,b_k\}$.  The symmetric difference of $C$ with $T_k$ is a triangle containing $\{c_k,d_k\}$.  
Thus {  (ii)}(b) holds; a \cn.   
We may now assume that $a_k\notin C$.  
Then $b_{k-1}\in C$ and   
 \ort\ implies that $\{b_{k-2},c_{k-2},a_{k-1}\}$ meets $C$. Moreover, \ort\ with $\{d_{k-2},a_{k-1},c_{k-1},d_{k-1}\}$ implies that $a_{k-1}$ is not in $C$. {  Thus~\ref{seize} holds.}  
 
Suppose that $k\geq 3$. Then,  by \ort, $c_{k-2} \not \in C$. Moreover, without loss of generality, we may assume that $e \in C$. Thus $C = \{e,b_{k-2},b_{k-1},b_k\}$. Then, by \ort\ between $C$ and    $\{b_{k-3},c_{k-3},a_{k-2},b_{k-2}\}$, we deduce that $e\in\{b_{k-3},c_{k-3}\}$. Thus $k=3$ and $T_0=\{e,f,g\}$.  
Without loss of generality, $e=b_0$.  
{  Since (i) holds,}
$c_0=g$, and $M$ contains the structure in Figure~\ref{diamond0} with all the elements in that figure being distinct. Then Lemma~\ref{nodiamonds} implies that {  (ii)}(a) holds.

We may now assume that $k = 2$. Recall that $C = \{\alpha,\beta,\gamma,b_2\}$ where 
$(\alpha,\beta,\gamma,\delta)$ is a $4$-fan of $M\ba c_1,c_2/b_2$. Now $M$ has a cocircuit $C^*$ such that  $\{\beta,\gamma, \delta\} \subseteq C^* \subseteq \{\beta,\gamma, \delta, c_1,c_2\}$. Since $\{\beta,\gamma\}$ meets $\{b_1,b_0,c_0\}$, it follows that $|C^*| \neq 3$. 

Next we show the following.

\begin{sublemma}
\label{efg} 
If $\{e,f,g\} = T_0$, then $C = \{b_0,b_1,b_2,y\}$ for some element $y$ that is not in 
$T_0 \cup T_1 \cup T_2 \cup \{d_0,d_1,d_2\}$. If $\{e,f,g\} \neq T_0$, then, without loss of generality,   $e \in C$ and 
  $C$ is $\{b_0,b_1,b_2,e\}$ or $\{c_0,b_1,b_2,e\}$. 
\end{sublemma}

Suppose first that $\{e,f,g\} = T_0$. By (a), we may assume that $(a_0,b_0,c_0) = (e,f,g)$. If $M$ has a triangle that meets $T_0, T_1$, and $T_2$, then 
$\lambda(T_0 \cup T_1 \cup T_2) \le 2$; a \cn. Thus we may assume that $M$ has no such triangle. Now $C$ contains $\{b_1,b_2\}$ and meets both $\{b_0,c_0\}$ and $\{a_0,b_0\}$. Thus $C$ meets each of $T_0$, $T_1$, and $T_2$. If $C$ contains two elements of one of these triangles, say $T_i$, then $C \btu T_i$ is a triangle that meets each of $T_0, 
T_1$, and $T_2$; a \cn. Thus $C= \{b_0,b_1,b_2,y\}$ for some element $y$ that avoids $T_0 \cup T_1 \cup T_2$. By \ort, $y$ also avoids $\{d_0,d_1,d_2\}$.  Thus the first part of \ref{efg} holds. The second part is an immediate consequence of \ref{seize}. 

\begin{sublemma}
\label{beegee} 
{  If $\{e,f,g\} =T_0$, then} $\{\beta,\gamma\}$ is $\{b_1,y\}$; {  otherwise} $\{\beta,\gamma\}$ is $\{b_1,e\}$.
\end{sublemma}

To see this, note that, by \ref{efg}, $\{\alpha, \beta,\gamma\}$ is one of $\{b_0,b_1,y\}$, $\{b_0,b_1,e\}$, or $\{c_0,b_1,e\}$. To prove \ref{beegee}, it suffices to show that $T_0$ avoids $\{\beta,\gamma\}$. Assume the contrary. Then \ort\ between $C^*$ and $T_0$ implies that $\delta \in T_0$. Suppose $c_1 \in C^*$. Then \ort\ implies that $\{\beta,\gamma\}- T_0 = \{b_1\}$ and $c_2 \not \in C^*$. Thus $C^*$ is a $4$-cocircuit that meets both $T_0$ and $T_1$ but is not $\{b_0,c_0,a_1,b_1\}$; a \cn\ to Lemma~\ref{bowwow}. Thus $c_1 \not \in C^*$. Hence $c_2 \in C^*$ and we contradict \ort\ with $T_2$. We conclude that \ref{beegee} holds.

We now know that $b_1 \in C^*$. As $|C^* \cap T_1|$ is even, either $c_1 \in C^*$, or $c_1 \not \in C^*$ and $C^*$ is $\{b_1,y,a_1,c_2\}$ or $\{b_1,e,a_1,c_2\}$. Thus $c_1 \in C^*$ otherwise \ort\   between $C^*$ and the circuit $\{c_1,b_2,c_2,d_1\}$  gives a \cn. Suppose $c_2 \in C^*$. Then, by \ort, $\delta = a_2$, so $C^*$ is a $5$-cocircuit containing $\{b_1,c_1,a_2\}$. The symmetric difference of this cocircuit with the cocircuit $\{b_1,c_1,a_2,b_2\}$ is a triad that contains $c_2$; a \cn. Hence $c_2 \not\in C^*$. Thus 
$C^*$ is $\{b_1,y,\delta,c_1\}$ or  $\{b_1,e,\delta,c_1\}$. By \ort, with the circuit $\{c_1,c_2,d_1,b_2\}$, we deduce that $\delta = d_1$. If $\{e,f,g\} \neq \{a_0,b_0,c_0\}$, then 
we have a \cn\ to \ort\ between $C^*$ and $\{e,f,g\}$. If $\{e,f,g\} = \{a_0,b_0,c_0\}$, then 
$\lambda(T_0 \cup T_1 \cup T_2 \cup \{d_1,y\}) \le 2$; a \cn\ as $|E(M)| \ge 15$. 
This completes the proof of~\ref{bkcoguts}.  

Now assume that part (ii) of the lemma fails. Next we show the following. 
\begin{sublemma}
\label{kgeq3}
$k \ge 3$ 
\end{sublemma}

Assume that $k = 2$.  Then applying Lemma 6.1 of \cite{cmoVI}, we see that   part (i) of that   lemma does not hold. Moreover, part (v) of that lemma does not hold by \ref{bkcoguts}.  If (ii) of \cite[Lemma 6.1]{cmoVI} holds, that is, $\{d_1,d_2\}$ is in a triangle of $M$, then this triangle together with $c_1$ and $a_2$ forms a $5$-fan in $M\ba c_2$; a \cn.  If (iv) of \cite[Lemma 6.1]{cmoVI} holds, that is, $a_1$ is in a triangle that avoids $\{b_1,c_1,d_1\}$, then \ort\ with $\{d_0,a_1,c_1,d_1\}$ implies that this triangle contains $d_0$, so {  (ii)}(a) of the current lemma holds; a \cn. 
Thus (iii) of  \cite[Lemma 6.1]{cmoVI} holds, that is, $\{b_1,b_2\}$ is in a triangle $T$ of $M$. By \ort, $T$ meets both $\{b_0,c_0\}$ and $\{e,f\}$. Thus $T_0 = \{e,f,g\}$, so $\{b_1,b_2,b_0\}$ or $\{b_1,b_2,c_0\}$ is a triangle. Then $\lambda(T_0 \cup T_1 \cup T_2) \le 2$; a \cn. We conclude that \ref{kgeq3} holds.  

We now apply Lemma 6.5 of \cite{cmoVI} to the configuration induced by $T_1 \cup T_2 \cup \dots \cup T_k \cup \{d_1,d_2,\dots,d_k\}$. Neither (i) nor (ii) of that lemma holds, and if (iv) holds, then $\{d_k,c_k\}$ is in a triangle, that is, {  (ii)}(b) of the current lemma holds; a \cn. We deduce that (iii) of \cite[Lemma~6.5]{cmoVI} holds.  Thus $M\ba c_1,c_2,\dots ,c_k$ is \ffsc\ and every \ftv\ of it is a $4$-fan $(u_1,u_2,u_3,u_4)$ where either $u_4=d_1$ and $a_1$ is in $\{u_2,u_3\}$; or $u_4=b_k$.  
Suppose that $(u_1,u_2,a_1,d_1)$ is a $4$-fan in $M\ba c_1,c_2,\dots ,c_k$.  
By \ort, $\{u_1,u_2\}$ meets $\{d_0,d_1\}$.  
{  Hence $\{d_0,a_1\}$ is contained in a triangle and (ii)(a) holds; a \cn.} 
We conclude that every $4$-fan of $M\ba c_1,c_2,\dots ,c_k$ has $b_k$ as its coguts element.  This contradiction to \ref{bkcoguts} completes the proof of the lemma.  
\end{proof}

We now consider the  case when $M$ contains the configuration in Figure~\ref{ABCDEfig}(C), but $N \not \preceq M\ba 1,4$ and  $N \preceq M\ba 4/5 \ba a$.  

\begin{lemma}
\label{killtobler2}
Let $M$ and $N$ be binary \ifc\ matroids such that $|E(M)|\geq 15$ and $|E(N)|\geq 7$.  
Suppose that Hypothesis {  VII} holds and that $M$ contains structure (C) in Figure~\ref{ABCDEfig}, where $M\ba 6/5\ba a$ has an $N$-minor,  and $M\ba 4$ and $M\ba 6$ are \ffsc. 
If $M\ba 4,1$ has no $N$-minor, then 
\begin{itemize}
\item[(i)] $M$ has a quick win; or 
\item[(ii)] $M$ has an open-rotor-chain win or a ladder win; or 
\item[(iii)] $M$ has a mixed ladder win; or 
\item[(iv)] $M$ has an enhanced-ladder win.  
\end{itemize}
\end{lemma}

\begin{proof}
Suppose that $M$ has no quick win.  
As $M\ba 6/5\ba a\cong M\ba 4/5\ba a$,   each of these matroids has an $N$-minor.  

\begin{sublemma}
\label{josh}
Neither $M\ba 4,a/c$ nor $M\ba 4,a/b$ has an $N$-minor.
\end{sublemma}

Assume that \ref{josh} fails.  
Note that $M\ba 4,a/b\cong M\ba 4,c/b\cong M\ba 4,c/6$ and, by symmetry, $M\ba 4,a/c\cong M\ba 4,b/6$. Since $M\ba 4,a/c$  or $M\ba 4,a/b$ has an $N$-minor, we deduce that $M\ba 4/6$ has an $N$-minor.  
As $M\ba 4/6$ has $(1,2,3,5)$ as a $4$-fan, but $M\ba 4,1$ has no $N$-minor, it follows that $M\ba 4/6/5$ has an $N$-minor. Therefore, by Lemma~\ref{ccrider}, we conclude  that \ref{josh} holds.

Next we relabel letting $(5,6,4)=(a_0,b_0,c_0)$ and $(b,c,a)=(a_1,b_1,c_1)$. Take  $T_0,D_0,T_1,D_1,\dots ,T_n$ to be a right-maximal bowtie string in $M$. 
Now $M\ba c_0,c_1$ has an $N$-minor but none of $M\ba c_0,c_1/b_1,M\ba c_0,c_1/a_1$, and $M\ba c_0,1$ has an $N$-minor.  

Suppose that $\{a_0,b_0,x,c_n\}$ is not a cocircuit for all $x$ in $\{a_n,b_n\}$.  
Now, by {  Lemma~\ref{ABCDE}}, $a_0$ is in a unique triangle of $M$. 
Therefore, by~\cite[Lemma~10.1]{cmoVI},   the lemma holds. 

We may now assume   that $\{a_0,b_0,x,c_n\}$ is a cocircuit for some $x$  in $\{a_n,b_n\}$.  
Then $a_0\neq c_n$, so all of the elements in the bowtie string are distinct.  
Up to relabelling  $a_n$ and $b_n$, we may assume that $x=b_n$.  
{  Lemma~\ref{bowwow} implies that $n> 1$.}  
If $n=2$, then $\lambda (T_0\cup T_1\cup T_2)\leq 2$; a \cn.  
Therefore $n\geq 3$.  
By Lemma~\ref{stringybark} and the observations at the end of the second-last paragraph, we have that  
\begin{sublemma}
\label{dooby}
$M\ba c_0,c_1,\dots ,c_n$ has an $N$-minor, but  $M\ba c_0,c_1,\dots ,c_i/a_i$  has no $N$-minor for all $i$ in $\{1,2,\dots ,n\}$.  
\end{sublemma}

Next we show that 
\begin{sublemma}
\label{123cs}
$\{1,2,3\}$ avoids $\{c_0,c_1,\dots ,c_n\} \cup \{a_n,b_n\}$.
\end{sublemma}

Suppose first that $\{1,2,3\}$ meets $\{c_0,c_1,\dots ,c_n\}$. Then $\{2,3\}$ meets the last set, since $M\ba c_0,1$ has no $N$-minor.  
Up to switching the labels on $2$ and $3$, we may assume that $3=c_i$ for some $i$ in 
$\{0,1,\dots ,n\}$. Then $2$ is in a $1$- or $2$-cocircuit  of $M\ba c_0,c_i$, so $M\ba c_0,c_i/2$ has an $N$-minor. Hence so does $M\ba c_0,/2\ba 1$; a \cn.  We deduce that 
$\{1,2,3\}$ avoids $\{c_0,c_1,\dots ,c_n\}$.

Now suppose that $\{1,2,3\}$ meets $\{a_n,b_n\}$. As $\{1,2,3\} \neq T_n$, \ort\ between $\{1,2,3\}$ and $D_{n-1}$ implies that $b_{n-1} \in \{1,2,3\}$. 
{  If $n=2$, then $\lambda (T_0\cup T_1\cup T_2)\leq 2$; a \cn.  If $n\geq 3$, then} 
\ort\ with $D_{n-2}$ implies, since $\{1,2,3\} \neq T_{n-1}$, that $b_{n-2} \in \{1,2,3\}$. Then \ort\ between $\{1,2,3\}$ 
and $D_{n-3}$ gives a \cn. Thus 
\ref{123cs} holds.

Evidently $(1,2,3,a_0)$ is a $4$-fan in $M\ba c_0,c_1,\dots ,c_n$.  
Since deleting $1$ from the last matroid destroys all $N$-minors, Lemma~\ref{2.2} implies that  $M\ba c_0,c_1,\dots ,c_n/a_0$ has an $N$-minor.  
Now

\begin{align*}
M\ba c_0,c_1,\dots ,c_{n-1},c_n/a_0 &\cong M\ba b_0,c_1,\dots ,c_{n-1},c_n/a_0\\
&\cong M\ba b_0,c_1,\dots ,c_{n-1},c_n/b_n\\
&\cong M\ba b_0,c_1,\dots ,c_{n-1},a_n/b_n\\
&\cong M\ba b_0,c_1,\dots ,c_{n-1},a_n/b_{n-1}\\
& ~~~~~~~~ \vdots ~~~~~~~~\\
&\cong M\ba b_0,c_1,a_2\dots ,a_{n-1},a_n/b_1\\
&\cong M\ba b_0,a_1,\dots ,a_{n-1},a_n/c_0.
\end{align*}
Therefore $M/b_i\ba c_i$ has an $N$-minor for all $i$ in $\{1,2,\dots ,n\}$.  
Hypothesis~VII implies that $M\ba c_1$ is \ffsc\ and, indeed, that $M\ba c_i$ is \ffsc\ for all 
$i$  in $\{1,2,\dots ,n\}$.  

Consider $M/b_n\ba c_n$, which has an $N$-minor.  
Lemma~\ref{claim2} implies that $M/b_n\ba c_n$ is \ffspc\ and either 
\begin{itemize}
\item[(I)] $M$ has a triangle $\{x,y,z\}$   such that $\{y,z,a_n,c_n\}$ is a cocircuit; or 
\item[(II)] $M$ has elements $d_{n-1}$ and $d_n$ such that $\{d_{n-1},d_n\}$ avoids $T_{n-1}\cup T_n\cup T_0$ where $\{d_{n-1},a_n,c_n,d_n\}$ is a cocircuit, and $\{d_{n-1},a_n,s\}$ or $\{d_n,c_n,t\}$ is a triangle for some $s$ in $\{b_{n-1},c_{n-1}\}$ or $t$ in $\{a_0,b_0\}$.  
\end{itemize}

\begin{sublemma}
\label{noI}
Part (I) does not hold.
\end{sublemma}

Suppose  that (I) does hold.    
Since we have a right-maximal bowtie string, we know that $\{x,y,z\}$ meets $T_0\cup T_1\cup \dots \cup T_n$.  
By Lemma~\ref{ABCDE}, $T_0$ is the only triangle containing $a_0$. Thus,  by~\cite[Lemma~5.4]{cmoVI},
$\{x,y,z\}=T_i$ for some $i$ in $\{0,1,\dots ,n-2\}$.  Moreover, by  Lemma~\ref{bowwow}, $i \neq 0$. 
If $c_i\in\{y,z\}$, then $M\ba c_0,c_1,\dots ,c_n$ has $a_n$ in a cocircuit of size at most two, so we can contract $a_n$ from the last matroid keeping an $N$-minor; a \cn\ to~\ref{dooby}.
Therefore $c_i=x$, so $\{a_i,b_i\}=\{y,z\}$.  
Now $D_{i-1}\btu \{y,z,a_n,c_n\}$ is $\{b_{i-1},c_{i-1},a_n,c_n\}$, which must be a cocircuit. 
{  Again $a_n$ is}
in a cocircuit in $M\ba c_0,c_1,\dots ,c_n$ of size at most two, so contracting {  $a_n$} 
from the last matroid retains an $N$-minor; a \cn\ to~\ref{dooby}.  
We conclude that \ref{noI} holds.  

We may now assume that (II) holds.  Next we show the following.
\begin{sublemma}
\label{nowayjose}
{   $M$ has no triangle containing $\{d_n,c_n,t\}$.}  
\end{sublemma}

{   Suppose $M$ has a triangle $T$ containing $\{d_n,c_n\}$. By \ort\ with the cocircuit $\{b_n,c_n,a_0,b_0\}$, we deduce that $a_0$ or $b_0$ is in $T$.} As $T_0$ is the only triangle containing $a_0$, it follows that
 $T = \{d_n,c_n,b_0\}$.  
Orthogonality implies that $d_n\in\{c_0,a_1,b_1\}$  and hence that $\{d_{n-1},d_n\}\subseteq T_1$.  
{  Then} (I) holds so we have a \cn\ to \ref{noI} that completes the proof of  
  \ref{nowayjose}.

We now know that $\{d_{n-1},a_n,s\}$ is a triangle for some $s$ in $\{b_{n-1},c_{n-1}\}$.  
If $s=b_{n-1}$, then \ort\ implies that $d_{n-1}\in\{b_{n-2},c_{n-2}\}$. Hence \ort\ implies that $\{d_{n-1},d_n\}\subseteq T_{n-2}$, and $\lambda (T_{n-2}\cup T_{n-1}\cup T_n)\leq 2$; a \cn.  Thus $s=c_{n-1}$.  By assumption, $\{d_{n-1},d_n\}$ avoids $T_{n-1} \cup T_n \cup T_0$. If $\{d_{n-1},d_n\}$ meets $T_{0} \cup T_1 \cup \dots \cup T_n$, then, by \ort\ between $\{d_{n-1},a_n,b_n,d_n\}$ and each $T_i$, we see that  $\{d_{n-1},d_n\} \subseteq T_i$ for some $i \not\in \{n-1,n,0\}$, {  and (I) holds; a \cn.} 
We deduce that the elements of $T_{0} \cup T_1 \cup \dots \cup T_n \cup \{d_{n-1},d_n\}$ are distinct. 

\begin{figure}[htb]
\center
\includegraphics[scale = 0.8]{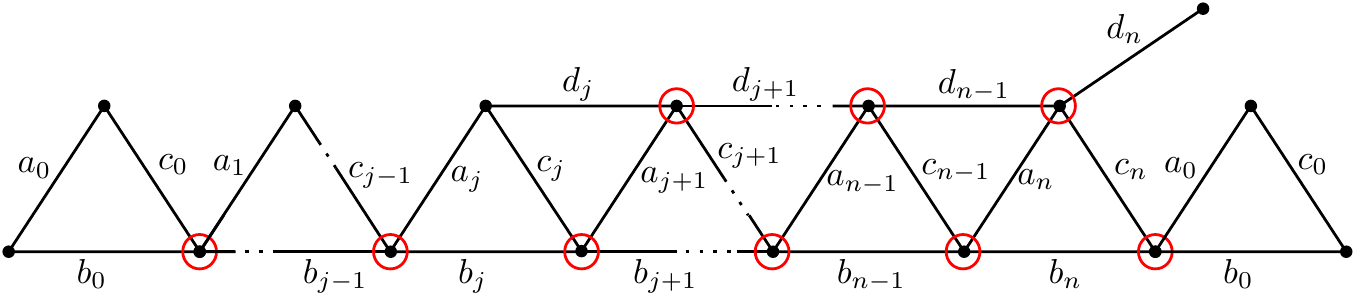}
\caption{All of the elements are distinct except those with the same label.} 
\label{newton}
\end{figure}

By taking $j = n-1$, we see that $M$ contains the structure in Figure~\ref{newton} where all of the elements shown are distinct except those with the same labels. 
Next we show the following.

\begin{sublemma}
\label{calculusbaby}
Suppose $M$ contains the structure in Figure~\ref{newton} for some $j$ with $1 \le j \le n-1$ 
where all of the elements are distinct except those with the same label. Then either $M$ has a mixed ladder win, or there is an element $d_{j-1}$ that is not in  $T_{0} \cup T_1 \cup \dots \cup T_n \cup \{d_j, d_{j+1},\dots,d_n\}$ such that $\{c_{j-1},d_{j-1},a_j\}$ is a triangle and $\{d_{j-1},a_j,c_j,d_j\}$ is a cocircuit.
\end{sublemma}

We apply Lemma~\ref{claim2} to the bowtie string $T_{j-1},D_{j-1},T_{j},D_{j},T_{j+1}$, 
noting that $M/b_{j}\ba c_{j}$ has an $N$-minor. As $M$ has no quick win,  outcome (i) of that lemma does not hold. 
{  Thus (ii) or (iii) of Lemma~\ref{claim2} holds, so $\{a_j,c_j\}$ is contained in a $4$-cocircuit $D^*$.  
Lemma~\ref{bowwow} implies that $D^*$ avoids $T_{j+1}$.  
By \ort\ with the circuit $\{c_j,d_j,a_{j+1}\}$, we see that $D^*$ contains $d_j$.  
Let $d_{j-1}$ be the fourth element of $D^*$.}  
%
The structure induced on 
 $T_{j-1}\cup T_{j}\cup \dots \cup T_n \cup T_0 \cup \{d_{j-1},d_{j},\dots,d_n\}$ has the form of the one shown in 
 Figure~\ref{rainbowbrite}. 
 
By \ort\ using the cocircuit $\{d_{j-1},a_j,c_j,d_j\}$ and the triangles in Figure~\ref{newton}, we see that $d_{j-1}$ avoids $T_{0} \cup T_1 \cup \dots \cup T_n$, and  $d_{j-1}$ avoids $\{d_{j},d_{j+1},\dots,d_{n-1}\}$. 
We now apply Lemma~\ref{rainbow} to the structure 
on $T_{j-1} \cup T_j \cup \dots \cup T_n \cup T_0 \cup \{d_{j-1}, d_{j},\dots,d_n\}$. 
If {  (ii)}(c) of that lemma holds, then $M$ has a mixed ladder win. Part {  (ii)}(b) does not hold by \ref{nowayjose}, so part {  (ii)}(a) holds; that is, $\{d_{j-1},a_j\}$ is contained in a triangle $T$. By \ort\ between $T$ and the cocircuits $\{b_{j-1},c_{j-1},a_j,b_j\}$ and $\{b_{j-2},c_{j-2},a_{j-1},b_{j-1}\}$, we deduce that  $T =  \{d_{j-1},a_j, c_{j-1}\}$. We conclude that \ref{calculusbaby} holds.

By repeatedly applying \ref{calculusbaby}, we find that either $M$ has a mixed ladder win or 
$M$ has $\{c_0,d_0,a_1\}$ as a triangle. Hence we may assume the latter. 
But recall that we began with a triangle $\{1,2,3\}$ and a cocircuit that, after relabelling, became $\{2,3,{  a_0,c_0}\}$. Now the elements in $\{1,2,3,a_0,b_0,c_0,a_1,b_1,c_1\}$ are distinct except that $1$ and $c_1$ may be equal. 
By \ort\ between $\{c_0,d_0,a_1\}$ and $\{2,3,{  a_0,c_0}\}$, we see that $\{d_0,a_1\}$ meets 
$\{2,3,b_0\}$ so $d_0\in\{2,3\}$. 
Suppose $1 = c_1$. Then the circuit $\{1,2,3\}$ is $\{d_0,c_1,a_2\}$ or $\{d_0,c_1,b_2\}$. The first possibility contradicts the fact that $\{d_1,c_1,a_2\}$ is a circuit; the second violates \ort. We deduce that $1 \neq c_1$, so $\{1,2,3\}$ avoids $\{a_0,b_0,c_0,a_1,b_1,c_1\}$. 
Now \ort\ between $\{1,2,3\}$ and $\{d_0,a_1,c_1,d_1\}$ implies that  $\{d_0,d_1\}\subseteq \{1,2,3\}$.  
 {  Then $\lambda (\{1,2,3\}\cup T_0\cup T_1)\leq 2$; a \cn.}  
\end{proof}

The results in this section enable us to conclude that $M$ does not contain the structure in Figure~\ref{ABCDEfig}(C).



\section{Configuration (A).}
\label{Aconfigsect}

In this section, we deal with the case when $M$ contains configuration (A) from Figure~\ref{ABCDEfig}, where $M\ba 4/5$ has an $N$-minor and $M\ba 4,1$ has no $N$-minor.  We begin with a straightforward lemma that will aid our efforts in this case.

\begin{lemma}
\label{con45} Suppose that $(\{1,2,3\},\{4,5,6\}), \{2,3,4,5\})$ is a bowtie in an \ifc\ binary matroid $M$ and that $M$ has $\{2,4,7\}$ as a triangle. Let $N$ be an \ifc\ matroid with at least seven elements such that $N$  is a minor of $M/\{x,y\}$ for some pair $\{x,y\}$ of elements of $\{4,5,6\}$. Then $N \preceq M\ba 1,4$.
\end{lemma}

\begin{proof}
As $\{4,5,6\}$ is a triangle of $M$, clearly $N \preceq M/\{4,5,6\}$, so $N\preceq M/5,6\ba 4$. Since $M/5,6\ba 4$ has $\{2,7\}$ as a circuit, it follows that $N\preceq M\ba 4,2$, so $N\preceq M\ba 4,2/3$. Hence $N \preceq M\ba 1,4$.
\end{proof}

{  

\begin{figure}[htb]
\center
\includegraphics{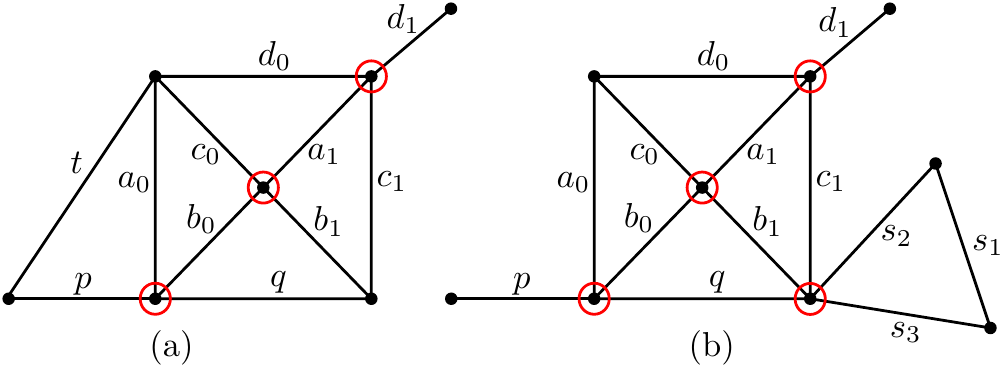}
\caption{The elements in both structures are all distinct, and we view the labels on $a_1$ and $b_1$ as being interchangeable.}
\label{fign=2}
\end{figure}

\begin{figure}[htb]
\center
\includegraphics{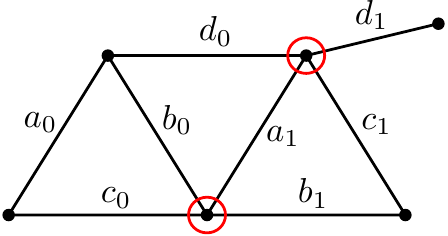}
\caption{The elements in this structure are all distinct, and we view the labels on $a_1$ and $b_1$ as being interchangeable.  Furthermore, $d_1\in Y$, and either $d_0\in X$ or $M'\ba c_0,c_1,d_0$ has an $N$-minor.}
\label{goofy}
\end{figure}


In the following lemma, we consider a matroid that is not necessarily \ifc, or even \thc.  
In the case that a binary matroid has $\{a,b,c\}$ as a disjoint union of circuits and $\{b,c,d\}$ as a disjoint union of cocircuits, we say that $(a,b,c,d)$ is a {\it loose $4$-fan}.  
If $(a,b,c,d)$ and $(e,d,c,b)$ are loose $4$-fans, then we say that $(a,b,c,d,e)$ is a {\it loose $5$-fan} in $M$ and a {\it loose $5$-cofan} in $M^*$.  
It is easy to see, by modifying the proof of Lemma~\ref{2.2}, that if $M$ has a loose $4$-fan and has an \ifc\ minor $N$ having at least seven elements, then either deleting the first element or contracting the last element of the loose $4$-fan retains the $N$-minor.  
We will use this fact within the proof of the next lemma.  

\begin{lemma}
\label{beachtheend}
Let $M$ and $N$ be \ifc\ binary matroids such that $|E(M)|\geq 16$ and $|E(N)|\geq 7$.  
Suppose that Hypothesis VII holds and that $M$ is not the cycle matroid of a terrahawk or a quartic M\"{o}bius ladder {  and is not the dual of a triadic M\"{o}bius matroid}.    
Let $M' = M\ba X/Y$ and let $M$ have $T_0,D_0,T_1,D_1,\dots ,T_n$ as a right-maximal bowtie string that is also  a bowtie string in $M'$.  Suppose that $M'\ba c_0,c_1,\dots ,c_n$ has an $N$-minor.  
Then one of the following holds.  
\begin{itemize}
\item[(i)] $M$ has a quick win; or
\item[(ii)] $M$ has an open-rotor-chain win, a ladder win, or an enhanced-ladder win; or 
\item[(iii)] $M$ has $\{a_0,b_0,z,c_n\}$ as a $4$-cocircuit for some $z$ in $\{a_n,b_n\}$; or
\item[(iv)]  the structure in Figure~\ref{drossfigiinouv} is contained in $M$, up to switching the labels on $a_n$ and $b_n$, and either $\{d_{n-2},a_{n-1},c_{n-1},d_{n-1}\}$ or $\{d_{n-2},a_{n-1},c_{n-1},a_n,c_n\}$ is a cocircuit, where $d_{-1}=\al$; or 
\item[(v)] $M'\ba c_0,c_1/b_1$ has an $N$-minor, or $n=1$ and $M'\ba c_0,c_1/a_1$ has an $N$-minor; or
\item[(vi)] $n=1$ and $M$ contains one of the structures in Figure~\ref{fign=2} or Figure~\ref{goofy}, where all of the elements are distinct and  the labels $a_1$ and $b_1$ are viewed as being interchangeable. Moreover,  if $M$ contains  the structure in Figure~\ref{goofy}, then $d_1\in Y$, and either $d_0\in X$ or $M'\ba c_0,c_1,d_0$ has an $N$-minor; or
\item[(vii)] deleting the central cocircuit of some  augmented $4$-wheel in $M$ gives an \ifc\ matroid with an $N$-minor.    
\end{itemize}
Furthermore, if neither (iii) nor (v) holds, then $M$ has no triangle $T_{n+1}$ such that $\{x,c_n,a_{n+1},b_{n+1}\}$ is a $4$-cocircuit for any $x$ in $\{a_n,b_n\}$.  
\end{lemma}

\begin{proof}
Suppose that neither (iii) nor (v) holds.  
By Lemma~\ref{stringswitch}, we get the following.
\begin{sublemma}
\label{conning}
{   If $i\in\{1,2,\dots ,n\}$, then $M'\ba c_0,c_1,\dots ,c_i/b_i$ has no $N$-minor.  If $j\in\{1,2,\dots ,n-1\}$, then $M'\ba c_0,c_1,\dots ,c_j/a_{j+1}$ has no $N$-minor. }
\end{sublemma}

{ 

\begin{sublemma}
\label{newsubber}
If $M$ has a triangle $T_{n+1}$ where $\{x,c_n,a_{n+1},b_{n+1}\}$ is a cocircuit for some $x$ in $\{a_n,b_n\}$, then $a_0\neq c_n$.  
\end{sublemma}

To show this, suppose that  $a_0=c_n$.  Then $n\geq 2$.  Without loss of generality, we assume that $x=b_n$.  By \ort\ with $T_0$, the cocircuit $D_n$ meets $\{b_0,c_0\}$.  Up to switching the labels on $a_{n+1}$ and $b_{n+1}$, we may assume that $b_{n+1}\in\{b_0,c_0\}$.  Then \ort\ with $D_0$ implies that $T_{n+1}$ meets $\{a_1,b_1\}$.  As $a_0\in D_n$, Lemma~\ref{bowwow} implies that $a_{n+1}\notin T_1$.  Hence $c_{n+1}\in\{a_1,b_1\}$.

Suppose $c_{n+1}=b_1$. Then \ort\ between $T_{n+1}$ and $D_1$ implies that $a_{n+1}\in\{a_2,b_2\}$.  Moreover,  \ort\ with $D_n$ implies that $T_2$ meets $\{b_n,a_0,b_{n+1}\}$. It follows that  $n=2$, so $c_2=a_0$ and $\lambda (T_0\cup T_1\cup T_2)\leq 2$; a \cn.  Thus $c_{n+1}=a_1$.

As the next step towards \ref{newsubber}, we now show that 

\begin{sublemma}
\label{an+1}  $\{b_{n+1},c_{n+1}\} \subseteq E(M')$ and    $a_{n+1} \in Y$.
\end{sublemma}

Since $b_{n+1}\in\{b_0,c_0\}$ and $c_{n+1}=a_1$, we know that $\{b_{n+1},c_{n+1}\}\subseteq E(M')$.  If $a_{n+1}\in X$, then $M'\ba c_0,c_1,\dots ,c_n$ has $b_n$ in a $1$- or $2$-cocircuit; a \cn\ to~\ref{conning}.  Suppose $a_{n+1}\in E(M')$.  If $b_{n+1}=c_0$, then $M'\ba c_0,c_1,\dots ,c_n$ has $\{b_n,a_{n+1}\}$ as a disjoint union of cocircuits; a \cn\ to~\ref{conning}.  Thus $b_{n+1}=b_0$ and $(a_1,b_0,a_{n+1},b_n)$ is a loose $4$-fan in $M'\ba c_0,c_1,\dots ,c_n$.  By~\ref{conning}, we see that $M'\ba c_0,c_1,\dots ,c_n,a_1$ has an $N$-minor.  The last matroid has $\{b_0,b_1\}$ as a disjoint union of cocircuits; a \cn\ to~\ref{conning}.  Hence ~\ref{an+1} holds. 

\begin{figure}[htb]
\center
\includegraphics{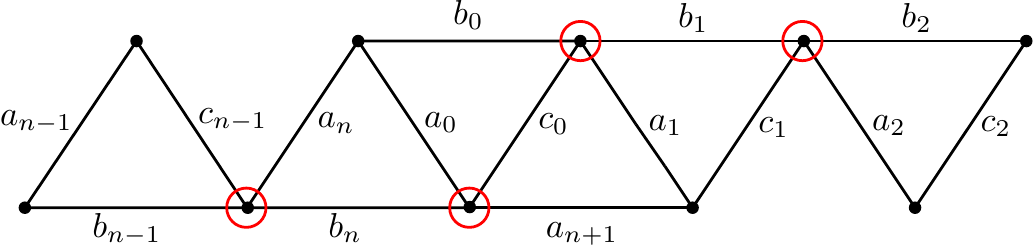}
\caption{}
\label{buildingqml}
\end{figure}

Still aiming to show \ref{newsubber}, we show next that 

\begin{sublemma}
\label{bn+1}  $b_{n+1} = c_0$.
\end{sublemma}

Suppose that  $b_{n+1} \neq c_0$. Then $M'\ba c_0,c_1,\dots ,c_n$ has $\{b_{n+1},a_1\}$ as a disjoint union of circuits, so $M'\ba c_0,c_1,\dots ,c_n,a_1$ has an $N$-minor.  But the last  matroid has $\{b_0,b_1\}$ as a disjoint union of cocircuits; a \cn\ to~\ref{conning}.  Thus 
\ref{bn+1} holds. 

We now know that  $M$ contains the structure in Figure~\ref{buildingqml} and that $M\ba c_0,c_1,\dots ,c_{n-1},c_n/b_n$ has an $N$-minor.  By Lemma~\ref{stringswitch}, $M\ba c_i/b_i$ has an $N$-minor for all $i$ in $\{0,1,\dots ,n\}$.  

As the next step towards \ref{newsubber}, we now show that 

\begin{sublemma}
\label{acco}  $M$ has a $4$-cocircuit containing $\{a_i,c_i\}$ for all $i$ in $\{1,2,\ldots,n\}$.
\end{sublemma}

Since $M$ has no quick win, this follows by Lemma~\ref{realclaim1} and Hypothesis VII unless $M\ba c_1$ is not \ffsc. Consider the exceptional case. By Lemma~\ref{6.3rsv}, $M$ has a quasi rotor $(T_0,T_1,\{7,8,9\},D_0,\{v,c_1,7,8\},\{u,v,7\})$ for some 
$u$ in $\{b_0,c_0\}$ and $v$ in $\{a_1,b_1\}$.
Since $M/b_1$ has an $N$-minor, Lemma~\ref{rotorwin} implies that
$b_1\neq v$.
Thus $v=a_1$.
By \ort\ between   $\{c_0,a_1,a_{n+1}\}$ and the cocircuit $\{a_1,c_1,7,8\}$, it follows that $\{7,8\}$
meets $\{c_0,a_{n+1}\}$.
Since the triangles $T_0,T_1$, and $\{7,8,9\}$ are disjoint, $c_0\notin \{a_1,c_1,7,8\}$.
Hence $a_{n+1}\in\{7,8\}$.
By \ort\ between $\{7,8,9\}$ and  $\{a_0,c_0,b_n,a_{n+1}\}$, it follows that
$b_n\in\{7,8,9\}$.
Then \ort\ between $\{7,8,9\}$ and $D_{n-1}$ implies that $\{7,8,9\}$ meets 
$\{b_{n-1},c_{n-1}\}$.
Then $M$ has
$(T_{n-1},T_n,\{c_0,a_1,a_{n+1}\},D_{n-1},\{a_0,c_0,b_n,a_{n+1}\},\{7,8,9\})$
as a quasi rotor in which $b_n$ is in two triangles. 
As $M/b_n$ has an $N$-minor, Lemma~\ref{rotorwin} gives a \cn.
Hence \ref{acco} holds.}

\begin{figure}[htb]
\center
\includegraphics{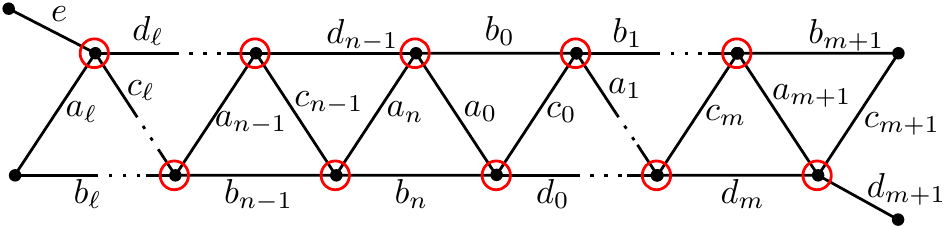}
\caption{The elements depicted are all distinct, $\ell\leq n$, where we let $d_n=b_0$, and $m\geq 0$.  Furthermore, $m\leq \ell -2$.}
\label{buildingdist}
\end{figure}

{  
We continue the proof of~\ref{newsubber} by showing that $M$ contains the structure in Figure~\ref{buildingdist}.  
We will construct the left side of the figure first.  Take $d_n = b_0$. 
Let $\ell=n$ if $\{c_{n-1},a_n\}$ is not contained in a triangle of $M$.  
If $\{c_{n-1},a_n\}$ is contained in a triangle, then take $\ell$ minimal in $\{1,2,\dots ,n-1\}$ such that $\{c_i, a_{i+1},d_i\}$ is a triangle of $M$ for all $i$ in $\{\ell,\ell +1,\dots ,n-1\}$, where $d_i$ is an element in $E(M)$.  
By~\ref{acco}, $M$ has a $4$-cocircuit containing $\{a_n,a_0\}$.  
By \ort\ with $T_0$ and Lemma~\ref{bowwow} applied to $(T_n,\{c_0,a_1,a_{n+1}\},\{a_0,c_0,b_n,a_{n+1}\})$, this cocircuit contains $b_0$.  
If $\ell < n$, then \ort\ and Lemma~\ref{bowwow} imply that the fourth element of this cocircuit is $d_{n-1}$. By repeated application of this argument, we see that $M$ has $\{d_{i-1},a_i,c_i,d_i\}$ as a cocircuit for all $i$ in $\{\ell+1,\ell+2,\ldots,n-1\}$. Moreover, $M$ has $\{e,a_{\ell},c_{\ell},d_{\ell}\}$ as a cocircuit for some element $e$ where we note that this includes the case when $\ell = n$. 

We continue to construct the structure in Figure~\ref{buildingdist} by showing that 
\begin{sublemma}
\label{ellnot1}
$\ell\geq 2$.  
\end{sublemma}

Suppose instead that $\ell=1$.  
Then by  \ort\ between $\{e,a_1,c_1,d_1\}$ and the triangle $\{c_0,a_1,a_{n+1}\}$ and  Lemma~\ref{bowwow}, we deduce that  $e=a_{n+1}$.  
By~\cite[Lemma~6.4]{cmoVI}, the elements in $T_1 \cup T_2\cup\dots\cup T_n\cup\{c_0,a_{n+1}\}\cup\{d_1 ,d_2,\dots,d_{n-1},b_0\}$ are distinct.  
Now it is easy to see that $M$ is the cycle matroid of a quartic M\"obius ladder or is the dual of a triadic M\"{o}bius ladder; a \cn.  
Thus~\ref{ellnot1} holds.  

We now consider the right side of Figure~\ref{buildingdist}. It is convenient to let $a_{n+1} = d_0$. Take $m$ maximal in $\{0,1,\ldots,\ell - 1\}$ such that, for all $j$ in $\{0,1,\ldots,m\}$, there is an element $d_j$ such that $\{c_j,d_j,a_{j+1}\}$ is a triangle. By the definition of $\ell$, we know that $\{c_{\ell - 1},a_{\ell}\}$ is not contained in a triangle of $M$. Hence

\begin{sublemma}
\label{mell}
$m\leq \ell - 2$.  
\end{sublemma}

As before, by \ort\ and Lemma~\ref{bowwow}, we know that $\{d_{j-1},a_j,c_j,d_j\}$ is a cocircuit for all $j$ in $\{1,2,\ldots,m\}$. Moreover, $\{d_m,a_m,c_m,d_{m+1}\}$ is a cocircuit for some element $d_{m+1}$. We deduce that $M$ contains the structure in Figure~\ref{buildingdist} where $\ell \le n$, and $m \ge 0$.

Next we show  that 

\begin{sublemma}
\label{distinctor}
the elements in Figure~\ref{buildingdist} are  distinct.  
\end{sublemma}

By~\cite[Lemma~6.4]{cmoVI}, since $M$ is neither the dual of the triadic M\"{o}bius matroid nor the cycle matroid of a quartic M\"{o}bius ladder, 
the elements in Figure~\ref{buildingdist} other than $\{e,c_{m+1},d_{m+1}\}$ are  distinct  except that  $(b_m,b_{m+1})$ may equal $(a_\ell ,b_\ell)$.  
As $\ell \neq m+1$, we deduce that 
the elements in Figure~\ref{buildingdist} are distinct unless $e, c_{m+1}$, or $d_{m+1}$ is equal to one of the other elements.

Suppose $c_{m+1}$ is equal to another element in Figure~\ref{buildingdist}. Then $c_{m+1}\in\{e,d_0,d_1,\dots ,d_m,d_\ell ,d_{\ell +1},\dots, d_{n-1}\}$.  
By \ort\ between $T_{m+1}$ and the cocircuits in Figure~\ref{buildingdist}, it follows  that $e\in\{a_{m+1},b_{m+1}\}$ and $c_{m+1}=d_\ell$.  
Then \ort\ with the cocircuits in Figure~\ref{buildingdist} implies that $(\ell,m) =(n,0),$  and $T_1$ is $\{e,b_0,b_1\}$ or $\{e,b_0,a_1\}$.  
In either case, $\lambda (T_n\cup\{c_0,d_0,a_1,e,b_0,b_1\})\leq 2$; a \cn.  
Thus $c_{m+1}$ is not equal to any other element in Figure~\ref{buildingdist}.

By \ort\ between the triangles in Figure~\ref{buildingdist} and the cocircuit $\{d_m,a_{m+1},c_{m+1},d_{m+1}\}$, we know that $d_{m+1}$ avoids all of the triangles in that figure.  Similarly,  
 \ort\ between these triangles and the cocircuit $\{e,a_\ell,c_\ell,d_\ell\}$ implies  that $e$ avoids all these triangles. 
We deduce that~\ref{distinctor} holds unless $d_{m+1}=e$.  
In the exceptional case, the set $\{e,a_\ell,c_\ell,d_\ell\}\btu\{d_m,a_{m+1},c_{m+1},d_{m+1}\}$ contains a cocircuit, and the set $Z$ of elements in Figure~\ref{buildingdist} other than $\{e,d_{m+1}\}$ is $3$-separating.  Then $Z$ meets 
 $T_{\ell -1}\cup e$  otherwise we contradict  the fact that $M$ is \ifc.  But now we get a \cn\ to \ort\ between $T_{\ell - 1}$ and the cocircuits in the figure. 
Thus~\ref{distinctor} holds.  

Continuing the proof of~\ref{newsubber}, we show next that   

\begin{sublemma}
\label{dm+1hasn}
$M/d_{m+1}$ has an $N$-minor.  
\end{sublemma}

Recall that $M\ba c_0,c_1,\dots ,c_{n-1},c_n/b_n$ has an $N$-minor.  
By Lemma~\ref{stringswitch}, $M\ba a_1,a_2,\dots a_{m+1},c_{m+1},c_{m+2},\dots ,c_{n-1},c_n/d_m$ has an $N$-minor.  
The last matroid is isomorphic to $M\ba a_1,a_2,\dots a_{m+1},c_{m+1},c_{m+2},\dots ,c_{n-1},a_0/d_{m+1}$.  
Hence~\ref{dm+1hasn} holds.  

Next we show that

\begin{sublemma}
\label{dm+1notintri}
$d_{m+1}$ is in no triangle of $M$.  
\end{sublemma}

Suppose $M$ has $d_{m+1}$ in a triangle, $T$.  
By \ort\ between $T$  and the $4$-cocircuits in Figure~\ref{buildingdist}, we see that $\{c_{m+1},d_{m+1}\}\subseteq T$.    
By the maximality of $m$, we know that $a_{m+2}\notin T$.  
Thus \ort\ with $D_{m+1}$ implies that $b_{m+2}\in T$; a \cn\ to \ort\ with $D_{m+2}$.  
Thus~\ref{dm+1notintri} holds.  

As $M$ has no quick win, we know that $M/d_{m+1}$ is not \ifc.  
We complete the proof of~\ref{newsubber} by giving a \cn\ to this fact.  

Suppose $(U,V)$ is a \ns\ $2$- or \ths\ of $M/d_{m+1}$.  
By~\cite[Lemma~3.3]{cmoV}, we may assume that $T_m\cup T_{m+1}\cup d_m\subseteq U$.  
Then $(U\cup d_{m+1},V)$ is a \ns\ $2$- or \ths\ of $M$; a \cn.  
By \ref{dm+1notintri}, $M/d_{m+1}$ has no sequential $2$-separation.  
It follows that $M/d_{m+1}$ is $3$-connected having  a $4$-fan, $(u,v,w,x)$.  
Hence $M$ has $\{v,w,x\}$ as a cocircuit and $\{d_{m+1},u,v,w\}$ as a circuit.  
By \ort\ with the cocircuit $\{d_m,a_{m+1},c_{m+1},d_{m+1}\}$, the set $\{u,v,w\}$ meets $\{d_m,a_{m+1},c_{m+1}\}$.  
If $\{u,v,w\}$ meets $D_m$ or $D_{m+1}$, then \ort\ implies that two elements of $\{u,v,w\}$ is in one of these cocircuits.  
Then $\{v,w,x\}$ meets $T_m,T_{m+1}$, or $T_{m+2}$; a \cn.  
Thus $\{u,v,w\}$ avoids $\{a_{m+1},c_{m+1}\}$,   
so $d_m\in\{u,v,w\}$.  
Either $\{d_{m-1},a_m,c_m,d_m\}$ is a cocircuit, or $m=0$ and $\{b_n,a_0,c_0,d_0\}$ is a cocircuit.  
Hence $\{u,v,w\}$ contains two elements from one of these cocircuits, so the triad $\{v,w,x\}$ meets a triangle in Figure~\ref{buildingdist}; a \cn.  
Thus $M/d_{m+1}$ has no $4$-fan; a \cn.  
We conclude that $a_0\neq c_n$.  
Thus~\ref{newsubber} holds.  }

Next we show that

\begin{sublemma}
\label{non+1}
$M$ does not have a triangle $T_{n+1}$ such that $(T_n,T_{n+1},\{x,c_n,a_{n+1},b_{n+1}\})$ is a bowtie for some $x$ in $\{a_n,b_n\}$.  
\end{sublemma}

Suppose $M$ has such a bowtie.  
Then~\cite[Lemma~5.4]{cmoVI} implies that $T_{n+1}=T_j$ for some $j$ in $\{0,1,\dots, n-2\}$, so $n\geq 2$.  
If $c_j\in\{a_{n+1},b_{n+1}\}$, then $M'\ba c_0,c_1,\dots ,c_n$ has $x$ in a $1$- or $2$-cocircuit, so $M'\ba c_0,c_1,\dots ,c_n/x$ has an $N$-minor; a \cn\ to~\ref{conning}.  
Thus $c_j=c_{n+1}$ {   so $\{a_j,b_j\} = \{a_{n+1},b_{n+1}\}$.   
If $j=0$, then  (iii) holds; a \cn.  
Thus $j\geq 1$, and $D_{j-1}\btu \{x,c_n,a_{n+1},b_{n+1}\}$}  is $\{b_{j-1},c_{j-1},x,c_n\}$, a disjoint union of cocircuits in $M'$.  
Again $M'\ba c_0,c_1,\dots ,c_n$ has $x$ in a $1$- or $2$-cocircuit; a \cn.  
Thus~\ref{non+1} holds, as does the last assertion of the lemma.  

We now assume that  the lemma fails.  
We show next that 
\begin{sublemma}
\label{cnffsc}
$M\ba c_n$ is \ffsc\ and every $4$-fan in $M\ba c_n$ has {   the form $(u,v,d_{n-1},d_n)$ for some $u$ and $v$ in $\{b_{n-1},c_{n-1}\}$ and $\{a_n,b_n\}$, respectively, where 
$|T_{n-1} \cup T_n \cup \{d_{n-1},d_n\}| = 8$. }
 \end{sublemma}
 
As (i) does not hold, $M\ba c_n$ is not \ifc.  
Observe that, if $T_n$ is the central triangle of a quasi rotor with $a_n$ or $b_n$ as its central element, then we have a \cn\ to~\ref{non+1}.  
{   Thus, applying Lemma~\ref{6.3rsv} to the bowtie $(T_{n-1},T_n,D_{n-1})$, we deduce using \ref{non+1}   that outcome (iii) of that lemma holds; that is,  \ref{cnffsc} holds.  }

{  Without loss of generality, we may now assume that $M\ba c_n$ has a $4$-fan of the form $(u,a_n,d_{n-1},d_n)$  for some $u$  in $\{b_{n-1},c_{n-1}\}$.  }

\begin{sublemma}
\label{ubn}
{   $u \neq b_{n-1}$.}
 \end{sublemma}

{   To show this, we assume the contrary. 
Suppose that  $n>1$. Then \ort\ between $\{b_{n-1},a_n,d_{n-1}\}$ and  $D_{n-2}$ implies that $d_{n-1}\in\{b_{n-2},c_{n-2}\}$.  
Then \ort\ between $\{a_n,d_{n-1},d_n,c_n\}$ and $T_{n-2}$ implies that $\{a_n,c_n,d_n\}$ meets $T_{n-2}$.  
If $a_{n-2} = c_n$, then 
$n=2$ and, since $d_1\in\{b_0,c_0\}$, the triangle $T_0$ is in the closure of $T_1\cup T_2$, so $\lambda (T_0\cup T_1\cup T_2)\leq 2$; a \cn.  
Thus $a_{n-2}\neq c_n$, so $d_n\in T_{n-2}$, a \cn\ to~\ref{non+1}.  
We deduce that  $n=1$ and $M$ contains the structure in Figure~\ref{goofy}.  }

Suppose $\{d_0,d_1\}$ avoids $Y$.  
If $\{d_0,d_1\}$ meets $X$, then $M'\ba c_0,c_1$ has $a_1$ in a $1$- or $2$-cocircuit, so (v) holds; a \cn.  
Thus $\{d_0,d_1\}$ avoids $X$, and $(b_1,b_0,a_1,d_0,d_1)$ is a loose $5$-cofan of $M'\ba c_0,c_1$, so (v) holds; a \cn.  
We deduce that $\{d_0,d_1\}$ meets $Y$.  
If $d_0\in Y$, then $\{b_0,a_1\}$ is a disjoint union of circuits in $M'\ba c_0,c_1$, so $M'\ba c_0,c_1,b_0$ has an $N$-minor.  
As this matroid has $\{a_1,b_1\}$ as a disjoint union of cocircuits, (v) holds; a \cn.  
Thus $d_0\notin Y$, so $d_1\in Y$.  
If $d_0\in X$, then (vi) holds; a \cn.  
Thus $d_0 \not\in X$, so $(d_0,b_0,a_1,b_1)$ is a loose $4$-fan in $M'\ba c_0,c_1$, and (v) or (vi) holds; a \cn.  
We conclude that \ref{ubn} holds.  

\begin{figure}[htb]
\center
\includegraphics{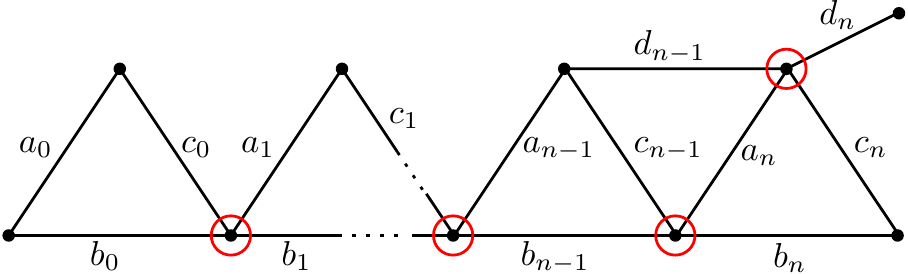}
\caption{We view the labels $a_n$ and $b_n$ as being interchangeable.  These elements are all distinct with the possible exception that $(a_0,b_0)$ may be $(c_n,d_n)$ in the case that $n>1$.  Furthermore, if $n>1$, then $d_{n-1}\notin Y$.  }
\label{figton}
\end{figure}

{   We  show next that
\begin{sublemma}
\label{newfigton}
$M$ contains the configuration in Figure~\ref{figton} where the elements $a_n$ and $b_n$ are viewed as being interchangeable. Moreover, the elements in this figure are distinct with the possible exception that $(a_0,b_0)$ may equal $(c_n,d_n)$ when $n > 1$. Also if $n > 1$, then $d_{n-1} \not\in Y$.
\end{sublemma}

By \ref{ubn},  $u=c_{n-1}$.  Thus, $M$ contains the configuration in Figure~\ref{figton}. Moreover, by \ref{cnffsc}, when $n = 1$, the elements in the figure are all distinct. 
Hence we may suppose that $n\geq 2$.  
By~\ref{non+1}, the pair $\{d_{n-1},d_n\}$ is not contained in a triangle of $M$.  
Hence, by \ort,  $\{d_{n-1},a_n,c_n,d_n\}$ avoids $T_1\cup T_2\cup\dots\cup T_{n-2}$.  
Now $d_{n-1} \not \in Y$ otherwise 
  $(b_{n-2},a_{n-1},b_{n-1},a_n,b_n)$ is a loose $5$-cofan in $M'\ba c_0,c_1,\dots ,c_n$, so $M'\ba c_0,c_1,\dots ,c_n/b_n$ has an $N$-minor; a \cn\ to~\ref{conning}.  

Suppose $T_0$ meets $\{d_{n-1},d_n\}$.  
By \ort\ with $\{d_{n-1},a_n,c_n,d_n\}$,  it follows that $a_0=c_n$,  and $\{b_0,c_0\}$ meets $\{d_{n-1},d_n\}$.  
Orthogonality between  $D_0$ and  $\{c_{n-1},d_{n-1},a_n\}$ implies that  $d_{n-1} \not\in \{b_0,c_0\}$. Hence $d_n\in\{b_0,c_0\}$.  
Suppose $d_n=c_0$.  
As $d_{n-1}\in E(M')\cup X$, we see that $a_n$ is in a $1$- or $2$-cocircuit of $M'\ba c_0,c_1,\dots ,c_n$; a \cn\ to~\ref{conning}.  
We conclude that \ref{newfigton} holds.}

\begin{sublemma}
\label{nge2}
$n\geq 2$.  
\end{sublemma}

{   As a first step towards proving this, we now show the following.
\begin{sublemma}
\label{6.1conseq}
If $n = 1$, then $M$ has $\{b_0,b_1,q\}$ as a triangle for some element $q$ that is not in $T_0 \cup T_1 \cup \{d_0,d_1\}$.
\end{sublemma} 

As $n = 1$, \ref{newfigton} implies that the elements in Figure~\ref{figton} are distinct. We now apply \cite[Lemma 6.1]{cmoVI} assuming first that (v) of that lemma does not hold, that is, that $M$ has no triangle containing $\{b_0,b_1\}$.  As the current lemma fails, it follows using \ref{non+1} that none of (i), (ii), or (iv) of \cite[Lemma 6.1]{cmoVI} holds. Thus (iii) of that lemma holds, so $M\ba c_0,c_1$ has a $4$-fan of the form  $(1,2,3,b_1)$. By \ref{non+1}, $M$ does not have $\{2,3,b_1,c_1\}$ as a cocircuit. 
Thus $M$ has $\{2,3,b_1,c_0,c_1\}$ or $\{2,3,b_1,c_0\}$ as a cocircuit.  
Since $T_0\cup T_1\cup d_0$ is not $3$-separating, this set contains no cocircuit of $M$ other than $D_0$.  
By \ort, $\{2,3\}$ meets both $\{a_0,b_0\}$ and $\{d_0,a_1\}$. Hence  $\{2,3\}=\{b_0,a_1\}$, so $a_1$ is in a triangle of $M\ba c_0,c_1$; a \cn\ to~\cite[Lemma~6.1]{cmoVI}.  
We conclude that (v) of \cite[Lemma~6.1]{cmoVI} holds, that is, $M$ has a triangle of the form $\{b_0,b_1,q\}$. Clearly $q \neq a_0$, and \ort\ with the cocircuits $D_0$ and 
 $\{d_0,a_1,c_1,d_1\}$ implies that $q \not\in T_0 \cup T_1 \cup \{d_0,d_1\}$. Thus \ref{6.1conseq} holds.  }
 
We continue our proof of \ref{nge2} by showing the following.

 \begin{sublemma}
\label{qxm}
Either $q \in X$, or $M'\ba c_0,c_1,q$ has an $N$-minor. 
\end{sublemma}

To prove this, suppose first that $q\in Y$. Then $\{b_0,b_1\}$ is a disjoint union of circuits in $M'\ba c_0,c_1$, so $M'\ba c_0,c_1,b_0$ has an $N$-minor.  
As this matroid has $\{a_1,b_1\}$ as a disjoint union of cocircuits, (v) holds; a \cn.  
Suppose $q\in E(M')$.  
Then $(q,b_0,b_1,a_1)$ is a loose $4$-fan in $M'\ba c_0,c_1$, and, as (v) does not hold, $M'\ba c_0,c_1,q$ has an $N$-minor.  
Thus \ref{qxm} holds.  

Consider $M\ba q$.  
Since this matroid is not \ifc, {  by applying Lemma~\ref{6.3rsv} to the bowtie $(\{d_0,c_0,a_1\}, \{b_0,b_1,q\}, \{a_1,c_0,b_0,b_1\})$,} we deduce that $\{q,b_0\}$ or $\{q,b_1\}$ is contained in a $4$-cocircuit of $M$.  
By \ort\ with the circuits $T_0$ and $T_1$, it follows using Lemma~\ref{bowwow} that $\{q,a_0,b_0\}$ or $\{q,b_1,c_1\}$ is contained in a $4$-cocircuit of $M$.  

\begin{figure}[htb]
\center
\includegraphics{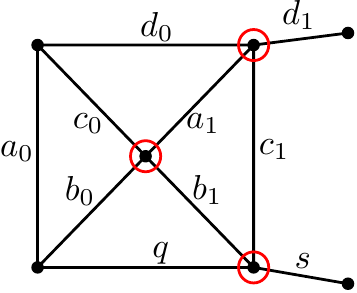}
\caption{We view the labels $a_1$ and $b_1$ as being interchangeable.  The elements shown are distinct.}
\label{twohorns}
\end{figure}

{  Next we show that 

\begin{sublemma}
\label{m4co}
$M$ has no $4$-cocircuit containing $\{q,b_1,c_1\}$. 
\end{sublemma} 
}

Suppose that $\{q,b_1,c_1,s\}$ is a  $4$-cocircuit in $M$, for some element $s$.  
By \ort\ between this cocircuit and  $T_0,T_1$, and $\{c_0,d_0,a_1\}$,  we see that either $s$ is a new element, or $s=d_1$.  
The latter {   gives the \cn\ that $\lambda (T_0\cup T_1\cup\{d_0,q,s\})\leq 2$.}  
Thus $s$ is a new element and $M$ contains the structure in Figure~\ref{twohorns}.  
Since $M\ba c_0,c_1,q$ has $\{b_1,s\}$ as a disjoint union of cocircuits,  $M\ba c_0,c_1,q/b_1$ has an $N$-minor.  
This matroid is isomorphic to $M\ba a_0,a_1,q/b_0$ by Lemma~\ref{stringswitch}.  
By Lemma~\ref{6.3rsv}, as $M$ has no quick win, $a_0$ is in a $4$-cocircuit $D$ of $M$.  
By \ort, $D$ meets $\{b_0,c_0\}$.  
Lemma~\ref{bowwow} implies that $D$ avoids $T_1$, so \ort\ with $\{c_0,d_0,a_1\}$ and $\{b_0,b_1,q\}$ implies that $D$ contains $\{a_0,c_0,d_0\}$ or $\{q,a_0,b_0\}$.  
If the fourth element of $D$ is $d_1$ or $s$, then the symmetric difference of $D$ with the $4$-cocircuit in Figure~\ref{twohorns} containing $d_1$ or $s$, respectively, is a new cocircuit contained in $T_0\cup T_1\cup\{d_0,q\}$.  
Thus the last  set is $3$-separating; a \cn.  
{   Using $D$ together with the structure in Figure~\ref{twohorns}, we see that  $M$ 
contains}  a good augmented $4$-wheel.  
Since $M\ba c_0,c_1,q/b_1$ has an $N$-minor and is isomorphic to $M\ba c_0,a_1,b_0/b_1$ and hence to $M\ba c_0,a_1,b_0,b_1$, we see that $M$ contains  an augmented $4$-wheel  such that removing the central cocircuit preserves an $N$-minor.  
Then~\cite[Theorem~4.1]{cmoV} implies that (i) or (vii) holds; a \cn.  
Thus \ref{m4co} holds.

We now  continue the proof of \ref{nge2}   knowing that  $\{a_0,b_0,q,p\}$ is a cocircuit in $M$, for some element $p$.  
By \ort\ between this cocircuit and the circuits $T_0,T_1$, and $\{c_0,d_0,a_1\}$, we see 
that either $p$ is a new element, or  $p=d_1$.  
If $p=d_1$, then 
$\lambda (T_0\cup T_1\cup \{d_0,q,{   d_1}\})\leq 2$; a \cn.  
We now apply~\cite[Lemma~6.2]{cmoVI}.  
{   Clearly (i) of that lemma does not hold. Also \ref{non+1} and \ref{m4co} imply that neither (ii) nor (v) holds. If (iv) holds, then $M$ has a triangle $\{a_0,p,t\}$ for some element $t$ which is new unless it equals $s$. Orthogonality excludes the exceptional case.} 
Thus $M$ contains structure (a) in Figure~\ref{fign=2}, and (vi) of the current lemma  holds; a \cn.  
{   We deduce that (iii) of \cite[Lemma~6.2]{cmoVI} holds, that is, } $M$ has a triangle $\{s_1,s_2,s_3\}$ and a cocircuit $\{q,c_1,b_1,s_2,s_3\}$ where $\{s_1,s_2,s_3\}$ avoids $\{b_0,c_0,q,a_1,b_1,c_1\}$.  
Orthogonality and~\ref{non+1} imply that $\{s_1,s_2,s_3\}$ avoids $\{d_0,d_1\}$.  
If $\{s_1,s_2,s_3\}$ meets $\{p,a_0\}$, then \ort\ implies that $\{p,a_0\}\subseteq \{s_1,s_2,s_3\}$, and $\lambda (\{b_0,c_0\}\cup T_1\cup \{{   s_1,s_2,s_3},q\})\leq 2$; a \cn.  
{   We conclude that $\{s_1,s_2,s_3\}$ avoids $p$ as well as all of the elements in Figure~\ref{twohorns}.} 
Thus $M$ contains structure (a) in Figure~\ref{fign=2}, and (vi) holds; a \cn.  
Therefore~\ref{nge2} holds.  

Take $m$ minimal such that, for all $i$ in $\{m,m+1,\dots ,n-1\}$, the matroid $M$ has 
$\{c_i,d_i,a_{i+1}\}$ as a triangle and has $\{d_i,a_{i+1},c_{i+1},d_{i+1}\}$ or $\{d_i,a_{i+1},c_{i+1},a_{i+2},c_{i+2}\}$ as a cocircuit, where the $5$-cocircuit option is only possible when  $i=n-2$.  

\begin{sublemma}
\label{distds}
The elements in $T_m\cup T_{m+1}\cup \dots \cup T_n\cup\{d_m,d_{m+1},\dots ,d_n\}$ are all distinct and $d_j\notin Y$ if ${  \max\{1,m\} }\leq j\leq n-1$.  
\end{sublemma}

To see this, suppose first that the elements in this set are not all distinct.  
By~\ref{non+1}, $\{d_{n-1},d_n\}$ is not contained in a triangle.  
By~\cite[Lemma~6.4]{cmoVI}, we deduce that {   $(a_m,b_m) = (c_n,d_n)$, and $M$ has } 
$\{b_n,c_n,c_m,d_m\}$ as a cocircuit.  {   Then $m \le n-2$.} 
But $\{c_m,d_m,a_{m+1}\}$ is a triangle of $M$, so we get a \cn\ to~\ref{non+1}.  
Thus the elements in $T_m\cup T_{m+1}\cup\dots\cup T_n\cup\{d_m,d_{m+1},\dots ,d_n\}$ are all distinct.  

Suppose  that $d_j\in Y$ for some $j$ with {   $\max\{1,m\}\leq j\leq n-1$. }
Then $(b_{j-1},a_j,b_j,a_{j+1},b_{j+1})$ is a loose $5$-cofan in $M'\ba c_0,c_1,\dots ,c_n$.  
Then $M'\ba c_0,c_1,\dots ,c_n/ b_{j+1}$ has an $N$-minor; a \cn\ to~\ref{conning}.  
Thus~\ref{distds} holds.  

Next we show that
\begin{sublemma}
\label{notdross}
$m=n-1$.  
\end{sublemma}

Suppose $m<n-1$.  
We consider $M\ba c_m,c_{m+1},\dots ,c_n$.  
By~\cite[Lemma~6.5]{cmoVI} and \ref{non+1}, since $M$ has no ladder win, $M\ba c_m,c_{m+1},\dots ,c_n$ is \ffsc\ and every $4$-fan of this matroid is either a $4$-fan of $M\ba c_n$ with $b_n$ as its coguts element   or a $4$-fan of  $M\ba c_m$ with $d_m$ as its coguts element    and with $a_m$ as an interior element.  
By~\ref{non+1},  $M\ba c_n$ does not have a $4$-fan with $b_n$ as its coguts element.  
Therefore $M\ba c_m$ has a $4$-fan of the form $(\be,\al,a_m,d_m)$, so $\{\al,a_m,d_m,c_m\}$ is a cocircuit of $M$.  
As (iv) does not occur, we know that $m\geq 1$.  
Orthogonality implies that $\{\al,\be\}$ meets $\{b_{m-1},c_{m-1},b_m\}$.  
As $\{\al,\be,a_m\}\neq T_m$, we know that $b_m\notin\{\al,\be\}$.  
Furthermore, Lemma~\ref{bowwow} implies that $\al\notin T_{m-1}$.  
Hence $\be\in\{b_{m-1},c_{m-1}\}$.  
By the minimality of $m$, we deduce that $\be=b_{m-1}$.  

Suppose  $m>1$. Then \ort\ between $\{b_{m-1},\al,a_m\}$ and  $D_{m-2}$, implies that $\al\in\{b_{m-2},c_{m-2},a_{m-1}\}$.  
As $\al\notin T_{m-1}$, the cocircuit  $\{\al,a_m,c_m,d_m\}$ meets $T_{m-2}$.  
Orthogonality implies that this $4$-cocircuit is contained in $T_{m-2}\cup T_m$.  Hence $\lambda (T_{m-2}\cup T_{m-1}\cup T_m)\leq 2$; a \cn.  
We conclude that  $m=1$.  

By~\ref{distds},  $d_1\notin Y$. Thus $d_1\in X\cup E(M')$.  
If $d_1\in X$, then $a_2$ is in a $1$- or $2$-cocircuit of $M'\ba c_0,c_1,\dots ,c_n$, so $M'\ba c_0,c_1/a_2$ has an $N$-minor; a \cn\ to~\ref{conning}.  
Hence $d_1\in E(M')$.  
If $\al\in E(M')$, then $(b_1,b_0,a_1,\al,d_1)$ is a loose $5$-cofan in $M'\ba c_0,c_1$, so $M'\ba c_0,c_1/b_1$ has an $N$-minor; a \cn\ to~\ref{conning}.  
Thus $\al\notin E(M')$.  
Then $\al\in Y$ or $\al \in X$ so, in $M'\ba c_0,c_1,\dots ,c_n$, either $\{b_0,a_1\}$ is a  disjoint union of circuits, or $\{a_1,d_1\}$ is a disjoint union of cocircuits.  
The former implies that $M'\ba c_0,c_1,\dots ,c_n,a_1$ has an $N$-minor.  
As the last  matroid also has $\{b_0,b_1\}$ as a disjoint union of cocircuits, we get a \cn\ to~\ref{conning}.  
Therefore $\{a_1,d_1\}$ is a disjoint union of cocircuits in $M'\ba c_0,c_1,\dots ,c_n$, so   
 $M'\ba c_0,c_1/d_1$ has an $N$-minor.  
As the last  matroid  has $(b_0,a_1,b_1,a_2,b_2)$ as a loose $5$-cofan, $M'\ba c_0,c_1/b_2$ has an $N$-minor. This \cn\ to~\ref{conning}   
  completes the proof of~\ref{notdross}.  

We now apply~\cite[Lemma~6.1]{cmoVI} to $M\ba c_{n-1},c_n$.  
By assumption and~\ref{non+1}, neither (i) nor (ii) of that lemma holds. 

{  
We show next that
\begin{sublemma}
\label{no4fan}
$b_n$ is not the coguts element of any $4$-fan of $M\ba c_{n-1},c_n$.   
\end{sublemma}
}
 
Suppose $(\al,\be,\ga,b_n)$ is a $4$-fan in $M\ba c_{n-1},c_n$.  
By~\ref{non+1}, we know that $\{\be,\ga,b_n,c_n\}$ is not a cocircuit of $M$. Hence either $\{\be,\ga,b_n,c_{n-1},c_n\}$ or $\{\be ,\ga ,b_n,c_{n-1}\}$ is a cocircuit $D$ of $M$.  
By \ort\ between $D$ and the triangles $T_{n-1}$ and $\{c_{n-1},d_{n-1},a_n\}$, we see that $\{\be,\ga\}\subseteq \{a_{n-1},b_{n-1},d_{n-1},a_n\}$.  
Now $T_{n-1}\cup T_n\cup d_{n-1}$ is not $3$-separating in $M$, and so $D =  D_{n-1}$.  
Thus $\{\be,\ga\}=\{b_{n-1},a_n\}$, so $a_n$ is in a triangle of $M\ba c_{n-1},c_n$; a \cn\ to~\cite[Lemma~6.1]{cmoVI}.  
Hence \ref{no4fan} holds.

By \ref{no4fan}, part (v) of \cite[Lemma~6.1]{cmoVI} does not hold. Now suppose that   (iv) of~\cite[Lemma~6.1]{cmoVI} holds. Then $M$ has a triangle $\{\al,\be,a_{n-1}\}$ that differs from $T_{n-1}$, and $M$ has  $\{\al,a_{n-1},c_{n-1},d_{n-1}\}$ or 
$\{\al,a_{n-1},c_{n-1},a_n,c_n\}$ as a cocircuit.  
By the minimality of $m$, we deduce that $\be\neq c_{n-2}$.  
By~\ref{nge2},  $n\geq 2$.  
Orthogonality between $\{\al,\be,a_{n-1}\}$ and  $D_{n-2}$ implies that $\{\al,\be\}$ meets $\{b_{n-2},c_{n-2},b_{n-1}\}$, so $b_{n-2}=\be$, or $\al\in\{b_{n-2},c_{n-2}\}$. 

 \begin{sublemma}
\label{bbeta}
{   $b_{n-2} \neq \beta$.}
\end{sublemma} 

{   Suppose that $b_{n-2} = \be$. We begin by locating the element $\al$. Suppose first that $\al \in Y$. Then $\{b_{n-2},a_{n-1}\}$ is a disjoint union of circuits in $M'\ba c_0,c_1,\dots,c_n$, so $M'\ba c_0,c_1,\dots,c_n,a_{n-1}$ has an $N$-minor. The last matroid has $\{b_{n-2},b_{n-1}\}$ as a disjoint union of cocircuits, so we get a \cn\ to \ref{conning}. Thus $\al \not\in Y$. 

Suppose next that $\al \in X$. Then, since $\{\al,a_{n-1},c_{n-1},d_{n-1}\}$ or $\{\al,a_{n-1},c_{n-1},a_n,c_n\}$ is a cocircuit in $M$, either  $M' \ba c_0,c_1,\dots ,c_n/d_{n-1}$ or $M'\ba c_0,c_1,\dots ,c_n/a_n$ has an $N$-minor.  The   second option contradicts \ref{conning}. If $M' \ba c_0,c_1,\dots ,c_n/d_{n-1}$ has an $N$-minor, then, as the last matroid has $(b_{n-2},a_{n-1},b_{n-1},a_n,b_n)$ as a loose $5$-cofan, $M'\ba c_0,c_1,\dots ,c_n/d_{n-1}/b_n$ has an $N$-minor. This \cn\ to~\ref{conning} establishes that $\al \not\in X$.  

We now know that $\al\in E(M')$.  
Then $(b_{n-1},b_{n-2},a_{n-1},\al,x)$ is a loose $5$-cofan in $M'\ba c_0,c_1,\dots ,c_n$ for some $x$ in $\{d_{n-1},a_n\}$, so $M'\ba c_0,c_1,\dots ,c_n/b_{n-1}$ has an $N$-minor. This \cn\ to \ref{conning} completes the proof of \ref{bbeta}.   
}

{   From the above,  if (iv) of~\cite[Lemma~6.1]{cmoVI} holds, then $\al\in\{b_{n-2},c_{n-2}\}$.  
Thus, by Lemma~\ref{bowwow}, 
 $\{\al,a_{n-1},c_{n-1},d_{n-1}\}$ is not a cocircuit of $M$, so  $\{\al,a_{n-1},c_{n-1},a_n,c_n\}$ is a cocircuit. Orthogonality with $T_{n-2}$ implies that $n=2$ and $a_0=c_2$.  
By \ort\ with $\{d_1,a_2,c_2,d_2\}$, the triangle $T_0$ meets $\{d_1,d_2\}$.  
Thus   $T_0$ is $\{\al,c_2,d_{1}\}$ or $\{\al,c_2,d_2\}$, so $\lambda (T_1\cup T_2\cup\{\al,d_{1},d_2\})\leq 2$; a \cn.   We conclude that 
 (iv) of~\cite[Lemma~6.1]{cmoVI} does not hold.  }

It now follows that (iii) of \cite[Lemma~6.1]{cmoVI} holds, that is,  $\{b_{n-1},b_n\}$ is contained in a triangle.  
Recall that $n\geq 2$ by~\ref{nge2}.  
Orthogonality with $D_{n-2}$ implies that the third element of this triangle is in $\{b_{n-2},c_{n-2}\}$.  
If $\{b_{n-2},b_{n-1},b_n\}$ is a triangle, then $M'\ba c_0,c_1,\dots ,c_n$ has $(a_n,b_n,b_{n-1},b_{n-2},a_{n-1})$ as a loose $5$-cofan, so $M'\ba c_0,c_1,\dots ,c_n/a_n$ has an $N$-minor; a \cn\ to~\ref{conning}.  We conclude that

\begin{sublemma}
\label{newtriangle}
{   $\{c_{n-2},b_{n-1},b_n\}$ is a triangle of $M$.} 
\end{sublemma}

Next we show  that 
\begin{sublemma}
\label{cndnno}
$\{c_n,d_n\}$ is not contained in a triangle.  
\end{sublemma}

Suppose $\{c_n,d_n\}$ is contained in a triangle $T$. Then $(T,\{a_n,d_{n-1},c_{n-1}\},\{d_n,c_n,a_n,d_{n-1}\})$ is a bowtie of $M$.   
{   By~\ref{cnffsc}, $M\ba c_n$ is \ffsc\  and $M\ba c_n,c_{n-1}$ has an $N$-minor. Thus, by applying Hypothesis VII to the last bowtie, we get that 
 $M\ba c_{n-1}$ is  \ffsc. This is a \cn\ as the last matroid has a $5$-fan. Hence \ref{cndnno} holds. } 

\begin{figure}[htb]
\center
\includegraphics{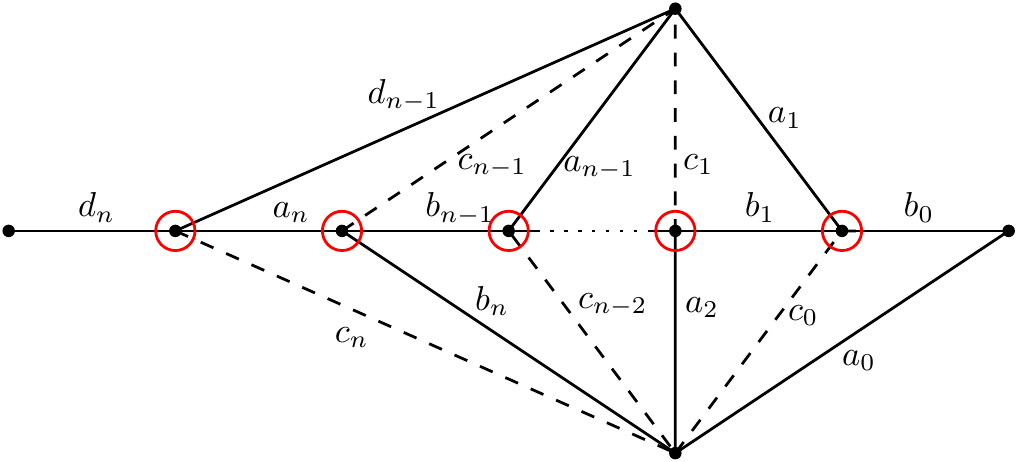}
\caption{These elements are all distinct.}
\label{figrot}
\end{figure}


{   Consider the rotor chain
$((c_n,a_n,b_n),(c_{n-1},b_{n-1},a_{n-1}),(c_{n-2},b_{n-2},a_{n-2}))$. If this rotor chain can be extended to the right by adjoining $(x,y,z)$, then $\{b_{n-2},a_{n-2},x,y\}$ is a cocircuit and $\{a_{n-1},b_{n-2},x\}$ is a circuit. By \ort, $x \in \{b_{n-3},c_{n-3}\}$. Indeed, by \ort\ with $D_{n-4}$, it follows that $x = c_{n-3}$ unless $n = 3$. Thus $\{x,y\} = \{b_{n-3},c_{n-3}\}$, so $z= a_{n-3}$. Continuing in this way, we see that a right-maximal bowtie chain that begins as above has one of the following forms:
\begin{enumerate}
\item[(a)] $((c_n,a_n,b_n),(c_{n-1},b_{n-1},a_{n-1}),\ldots,(c_{\ell},b_{\ell},a_{\ell}))$ for some $\ell$, which may be negative; or
\item[(b)] $((c_n,a_n,b_n),(c_{n-1},b_{n-1},a_{n-1}),\ldots,(c_1,b_1,a_1), (b_0,c_0,a_0), (c_{-1},b_{-1},a_{-1}), \linebreak 
\ldots, (c_{\ell},b_{\ell},a_{\ell}))$ for some $\ell \le 0$.
\end{enumerate}

To eliminate the second possibility, assume that it arises. Recall that $M\ba c_0,c_1,\ldots,c_n$ has an $N$-minor. The last matroid has $(b_0,b_1,a_2,b_2)$ as a $4$-fan. Then $M\ba c_0,c_1,\ldots,c_n\ba b_0$ or $M\ba c_0,c_1,\ldots,c_n/b_2$ has an $N$-minor. In both cases, $M/b_1$ has an $N$-minor. To see this, observe in the second case that
$M\ba c_0,c_1,\ldots,c_n/b_2
\cong M\ba c_0,c_1,c_3,c_4\ldots,c_n\ba a_2/b_2
\cong M\ba c_0,c_1,c_3,c_4\ldots,c_n\ba a_2/b_1$; in the first case, note that $M\ba c_0,c_1,\ldots,c_n\ba b_0$ has $b_1$ in a cocircuit of size at most two. Thus $M/b_1$ does indeed have an $N$-minor, and it follows by Lemma~\ref{rotorwin} that $M$ has a quick win; a \cn. We conclude that (a) holds.

If $\ell > 0$, then $M\ba c_0,c_1,\ldots,c_{\ell}$ has an $N$-minor. Now suppose that $\ell \le 0$. Let $k$ be a non-negative integer such that $-k \ge \ell$. We argue by induction on $k$ that
$M\ba c_0,c_1,\ldots,c_0,c_{-1},\ldots,c_{-k}$ has an $N$-minor. This is certainly true if $k = 0$. Hence we may assume that $k \ge 1$. Now
$M\ba c_n,c_{n-1},\ldots,c_0,c_{-1},\ldots,c_{-k+1}$ has
$(c_{-k},b_{-k+1},a_{-k+2},b_{-k+2})$ as a loose $4$-fan. By Lemma~\ref{rotorwin}, $M/b_{-k+2}$ does not have an $N$-minor. Thus
$M\ba c_0,c_1,\ldots,c_{-k}$ has an $N$-minor.

When $\ell \neq 0$, we will now adjust the notation to make $\ell = 0$ irrespective of whether it was originally positive or negative. This will change  $n$ but we will continue to use the same symbol for this quantity noting that we still know that its value is at least two and that
$M\ba c_n,c_{n-1},\ldots,c_{0}$ has an $N$-minor.

We show next that the elements in the rotor chain
 $((c_n,a_n,b_n),(c_{n-1},b_{n-1},a_{n-1}),\ldots,(c_{0},b_{0},a_{0}))$ are distinct. Suppose not. Then $a_0 = c_n$. Orthogonality between $T_0$ and $\{d_{n-1},a_n,c_n,d_n\}$ implies that $T_0$ meets $\{d_{n-1},d_n\}$. By \ref{cndnno}, we see that $d_{n-1} \in T_0$ but $d_{n-1} \neq a_0$. Then $\{c_{n-1},d_{n-1},a_n\}$ meets $D_0$ in a single element; a \cn\ to \ort. We conclude that the elements in the rotor chain are distinct. }

Suppose $\{d_{n-1},d_n\}$ meets a triangle in this rotor chain.  
Then \ort\ with the cocircuit $\{d_{n-1},a_n,c_n,d_n\}$ implies that $\{d_{n-1},d_n\}$ is contained in this triangle; a \cn\ to~\ref{non+1}.  
Thus $M$ contains the structure in Figure~\ref{figrot} where all of the elements shown are distinct.

Since  $N \preceq M\ba c_1,c_0$ and $M\ba c_1$ is not \ffsc, it follows by Hypothesis VII  
that  $M\ba c_0$ is not \ffsc. 
{   Now $M$ has a bowtie of the form $(T, \{z,b_1,c_0\},D_1)$ where $z$ is $b_2$ when $n = 2$, and $z$ is $c_2$ otherwise. Applying  
Lemma~\ref{6.3rsv} to this bowtie gives that $M$ has a quasi rotor having $\{z,b_1,c_0\}$ as its central triangle. Moreover, $M\ba c_0$ has a $5$-fan $(p,q,s,t,u)$ whose elements are contained in this quasi rotor.  
Then $\{q,s,t,c_0\}$ is a cocircuit $D$ of $M$.  
By \ort\ with $\{z,b_1,c_0\}$, the cocircuit $D$ meets either $T_1$ or $T_2$.  
If $D$ meets $T_2$, then \ort\ implies that it is contained in $T_0\cup T_2$, and $\lambda (T_0\cup T_1\cup T_2)\leq 2$; a \cn.  
Thus $D$ meets $T_1$.  
Lemma~\ref{bowwow} implies that $D = D_0$. 
Since $D_0 \cap D_1 = \{b_1\}$, it follows that $b_1$ is the central element of the quasi rotor. 
Thus $\{b_0,a_1,g\}$ is a triangle for some element $g$.  
By \ort\ with the cocircuits shown in Figure~\ref{figrot}, we know that $g$ is a new element. } 

{   Now $M\ba c_0,c_1,\dots ,c_n$ has an $N$-minor and has  
  $(g,b_0,a_1,b_1)$ as a loose $4$-fan. Since,  by Lemma~\ref{rotorwin}, $M/b_1$ has no $N$-minor, we deduce that }
  
\begin{sublemma}
\label{mgn}
{   $M\ba c_0,c_1,\dots ,c_n,g$ has an $N$-minor. } 
\end{sublemma}

{   Now $M$ has $(\{z,b_1,c_0\}, \{b_0,a_1,g\}, D_0)$ as a bowtie where $z$ is $b_2$ when $n = 2$ and is $a_2$ otherwise. } 
 As $M\ba g$ is not \ifc, Lemma~\ref{6.3rsv} implies that $M$ has a $4$-cocircuit of the form $\{v,w,x,g\}$.  
By \ort\ with $\{g,b_0,a_1\},T_0$, and $T_1$, we know that $\{v,w,x\}$ contains two elements of $T_0$ or two elements of $T_1$.  
By Lemma~\ref{bowwow},  $\{v,w,x\}$ avoids $\{z,b_1,c_0\}$, so $\{v,w,x,g\}$ contains $\{a_0,b_0,g\}$ or $\{a_1,c_1,g\}$.  
If the latter holds, then, since $M\ba c_0,c_1,g$ has an $N$-minor, we deduce that $M\ba c_0,c_1,g/a_1$ has an $N$-minor.  This gives a \cn\ to Lemma~\ref{rotorwin} because 
  $a_1$ is in two triangles of a   quasi rotor of  
$M$.  
Thus $\{v,w,x,g\}=\{a_0,b_0,f,g\}$  for some element $f$.   
By \ort\ with the triangles in Figure~\ref{figrot}, we deduce  that either  $f=d_n$, or $f$ differs from all the elements in Figure~\ref{figrot}.  

We show next that 
\begin{sublemma}
\label{efgnope}
$\{f,g\}$ is not contained in a triangle of $M$.  
\end{sublemma}

Suppose $\{e,f,g\}$ is a triangle of $M$.  
Since we constructed a right-maximal rotor chain, we know that $\{e,f,g\}$ meets the set of elements in Figure~\ref{figrot}.  {   But $g$ avoids this set of elements, and so does $f$ unless $f = d_n$.  
By \ort\ with the cocircuits shown in Figure~\ref{figrot}, it follows  that $f=d_n$ and $e\in\{d_{n-1},a_n,c_n\}$.  
Furthermore, \ort\ implies that $e\notin D_{n-1}$,   
so $e\neq a_n$. } 
By~\ref{non+1}, we deduce that $e=c_n$, so $\{c_n,d_n\}$ is contained in a triangle; a \cn\ to~\ref{cndnno}.  
Thus~\ref{efgnope} holds.  

\begin{figure}[htb]
\center
\includegraphics{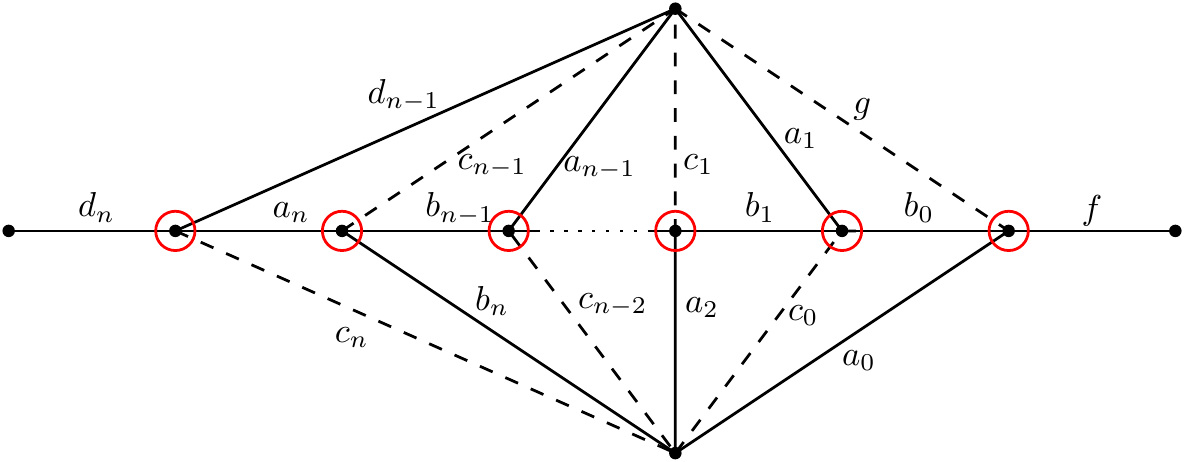}
\caption{These elements are all distinct except that $f$ may be $d_n$.  No set in $\{\{f,g\},\{c_n,d_n\},\{d_{n-1},d_n\}\}$ is contained in a triangle.  Deleting the dashed elements preserves an $N$-minor.}
\label{3not}
\end{figure}

Now $M$ contains the structure in Figure~\ref{3not}.  Moreover, by 
 ~\ref{non+1},  ~\ref{cndnno}, and~\ref{efg},   none of $\{d_{n-1},d_n\},\{c_n,d_n\}$, and $\{f,g\}$ is contained in a triangle.  
Next we show that 

\begin{sublemma}
\label{badthings}
neither $M/a_1$ nor $M/b_0$ has an $N$-minor, and $\{f,g,a_0,b_0\}$ is the only $4$-cocircuit in $M$ containing $a_0$.  
\end{sublemma}

Suppose $M/b_0$ has an $N$-minor. Then, since $M/b_0$ has $\{a_0,c_0\}$ and $\{g,a_1\}$ as circuits, $M/b_0\ba a_1,c_0$ has an $N$-minor.  
This matroid is isomorphic to $M/b_1\ba a_1,c_0$, so $M/b_1$ has an $N$-minor; a \cn\ to Lemma~\ref{rotorwin}.  
{   Thus $M/b_0$ has no $N$-minor. Moreover, an immediate consequence of Lemma~\ref{rotorwin} is that $M/a_1$ has no $N$-minor. }

Now suppose that $a_0$ is in a $4$-cocircuit $D$ other than $\{f,g,a_0,b_0\}$.  
Lemma~\ref{bowwow} implies that $D$ avoids $T_1$.  
Orthogonality with $T_0$ implies that $D$ meets $\{b_0,c_0\}$, so \ort\ with $\{g,b_0,a_1\}$ and either $\{a_2,b_1,c_0\}$ or $\{b_2,b_1,c_0\}$ implies that $D$ contains $\{a_0,b_0,g\}$ or $\{a_0,c_0,z\}$, for some element $z$ in {   $\{a_2,b_2\}$. } 
Since $D$ is not $\{f,g,a_0,b_0\}$, the latter holds.  
Then \ort\ implies that it is contained in $T_0\cup T_2$, and $\lambda (T_0\cup T_1\cup T_2)\leq 2$; a \cn.  
Thus~\ref{badthings} holds.

\begin{sublemma}
\label{withg}
$M\ba c_0,c_1,\dots ,c_n,g$ is \sfc.
\end{sublemma}

To see this, first observe that, by~\cite[Lemma~7.1]{cmoVI}, $M\ba c_0,c_1,\dots ,c_n$ is \sfc.   
The last matroid has $g$ in a triangle, so either it has $g$ in a triad, or  $M\ba c_0,c_1,\dots ,c_n,g$ is \thc.  
The former implies, by \ort\ with $\{g,b_0,a_1\}$, that $\{b_0,g\}$ or $\{a_1,g\}$ is contained in a triad of $M\ba c_0,c_1,\dots ,c_n$.  
Then {    \ref{mgn} implies that} $M\ba g,c_0,c_1,\dots ,c_n/x$ has an $N$-minor for some $x\in\{b_0,a_1\}$; a \cn\ to~\ref{badthings}.  
Thus $M\ba c_0,c_1,\dots, c_n,g$ is \thc.  
Suppose this matroid has a \ns\ \ths.  
Without loss of generality, we may assume that the triad $\{b_0,a_1,b_1\}$ is contained on one side of the \ths, and we can add $g$ to that side to get a \ns\ \ths\ of $M\ba c_0,c_1,\dots ,c_n$; a \cn.  
Thus~\ref{withg} holds.  

Next we show that
\begin{sublemma}
\label{newh}
$M$ has an element $h$ such that $\{a_0,f,h\}$ is a triangle   and $M\ba g,h$ has an $N$-minor.  
\end{sublemma}

{   Since $M$ has  no open-rotor-chain win, $M\ba c_0,c_1,\ldots,c_n,g$ has a $4$-fan $(1,2,3,4)$, so $M$ has a cocircuit $C^*$ such that
$\{2,3,4\} \subsetneqq C^* \subseteq \{2,3,4,c_0,c_1,\ldots,c_n,g\}$. By \ort\ with the cocircuits in Figure~\ref{3not}, we deduce that $\{1,2,3\}$ can only meet the set of elements in that figure if it contains
$\{d_{n-1},d_n\}, \{c_n,d_n\}, \{f,g\},$ or $\{a_0,f\}$. The first three possibilities have been excluded, so either $\{1,2,3\}$ avoids the set of elements in Figure~\ref{3not}, or $\{a_0,f\} \subseteq \{1,2,3\}$. Suppose the latter occurs. Then $M$ has a triangle of the form $\{a_0,f,h\}$, so $M\ba g$ has $(h,f,a_0,b_0)$ as a $4$-fan. By \ref{badthings}, $M\ba g/b_0$ does not have an $N$-minor. Thus $M\ba g,h$ has an $N$-minor, so \ref{newh} holds. We may now assume that $\{1,2,3\}$ avoids the set of elements in Figure~\ref{3not}.  Suppose $c_i \in C^*$ for some $i$ in $\{0,1,\ldots,n-1\}$. Then, since $c_i$ is in two triangles in Figure~\ref{3not}, \ort\ with these triangles implies that $\{2,3,4\}$ contains two elements in this figure. Hence $\{2,3\}$ contains an element in the figure; a \cn. Thus $C^*$ avoids $\{c_0,c_1,\ldots,c_{n-1}\}$, so $C^* \subseteq \{2,3,4,c_n,g\}$. }
Furthermore, \ort\ with $\{g,b_0,a_1\}$ and $T_n$ implies that either $C^*=\{2,3,4,g\}$ and $4\in\{b_0,a_1\}$,  or $C^*=\{2,3,4,c_n\}$ and $4\in\{a_n,b_n\}$.  
Thus $C^*$ meets one of the triangles $T_0,T_1, \{c_{n-1},d_{n-1},a_n\}$, or $\{c_{n-2},b_{n-1},b_n\}$ in a single element; a \cn\ to \ort.  
Thus~\ref{newh} holds.  

By \ort\ with the cocircuits in Figure~\ref{3not}, either $h$ differs from all the elements in that figure,  or $d_n=f$ and $h\in\{c_n,d_{n-1}\}$.  
By~\ref{non+1} and~\ref{cndnno}, we deduce that $h$ is a new element.  

We now show that 
\begin{sublemma}
\label{hffsc}
$M\ba h$ is \ffsc, and every $4$-fan of this matroid has $f$  as its coguts element.  
\end{sublemma}

Let $(z,y,x,w)$ be a $4$-fan in $M\ba h$.  
Then $\{w,x,y,h\}$ is a cocircuit of $M$, and \ort\ with $\{a_0,f,h\}$ implies that $\{w,x,y\}$ meets $\{a_0,f\}$ in a single element.  
By~\ref{badthings},  $a_0 \not\in \{w,x,y\}$, so $f\in\{w,x,y\}$.  
Suppose $f\in\{x,y,z\}$. Then \ort\ with $\{f,g,a_0,b_0\}$ implies that $\{x,y,z\}$ meets $\{g,a_0,b_0\}$.  
{   By \ref{efgnope}, $g \not\in \{x,y,z\}$. As $\{x,y,z\}$ does not contain $\{e_0,f\}$, it must contain $\{f,b_0\}$.} 
 By \ort\ with $D_0$, the triangle {  $\{x,y,z\}$} is $\{f,b_0,b_1\}$; a \cn\ to \ort\ with $D_1$.  
Thus $f=w$,   
that is, $f$ is the coguts element of every $4$-fan of $M\ba h$.  
Thus $M\ba h$ has no $5$-fan.  
It follows by Lemma~\ref{6.3rsv} that  $M\ba h$ is \ffsc. Thus~\ref{hffsc} holds.  

Since $M$ has no quick win, $M\ba h$ has a $4$-fan $(z_0,y_0,x_0,f)$.
Thus $(\{a_0,f,h\},\{x_0,y_0,z_0\},\{f,h,x_0,y_0\})$ is a  bowtie.

\begin{sublemma}
\label{mhf} $M\ba h/f$ has an $N$-minor.
\end{sublemma}

To show this, we assume the contrary. Then $M\ba h,z_0$ has an $N$-minor. Extend the bowtie $(\{a_0,f,h\},\{x_0,y_0,z_0\},\{f,h,x_0,y_0\})$ to a right-maximal bowtie string
$\{a_0,f,h\},\{f,h,x_0,y_0\},\{x_0,y_0,z_0\},\ldots, \{x_k,y_k,z_k\}$. Now $M\ba h,z_0/y_0 \cong M\ba h, x_0/y_0 \cong M\ba h,x_0/f$.
Thus $M\ba h,z_0/y_0$ has no $N$-minor. Therefore, by Lemma~\ref{stringybark}, $M\ba h,z_0,z_1,\ldots,z_k$ has an $N$-minor.

To complete the proof of \ref{mhf}, we aim to apply \cite[Lemma 10.1]{cmoVI}, but first we need to show that

\begin{sublemma}
\label{gammaray} $(\{x_k,y_k,z_k\}, \{a_0,f,h\}, \{\gamma,z_k,a_0,f\})$ is not a bowtie for all $\gamma$ in $\{x_k,y_k\}$.
\end{sublemma}

Assume that $M$ contains such a bowtie. Then, by \ref{badthings},
$\{\gamma,z_k\} = \{g,b_0\}$. By Lemma~\ref{bowwow}, $k > 0$. Suppose $z_k = b_0$. Then $M\ba b_0$ has an $N$-minor. This matroid has $(c_1,a_1,b_1,c_0,z)$ as a $5$-fan where $z = a_2$ unless $n= 2$, in which case $z = b_2$. By Lemma~\ref{2.2}, $M\ba b_0,c_1,z$ or $M\ba b_0,c_1/a_1$ has an $N$-minor. The former implies that $M/b_1$ has an $N$-minor, so in both cases we get a \cn\ to Lemma~\ref{rotorwin}. We conclude that $z_k \neq b_0$, so $z_k = g$. Now $M\ba h, z_0,z_1,\dots ,z_k$ has an $N$-minor, and, by~\ref{hffsc}, $M\ba h$ is \ffsc. 
By Hypothesis VII, it follows that $M\ba z_0$ is \ffsc,  that  
 $M\ba z_i$ is \ffsc\ for all $i$ in $\{1,2,\dots ,k-1\}$, and that $M\ba g$ is \ffsc. 
But $M\ba g$ has a $5$-fan; a \cn. 
Thus~\ref{gammaray} holds.

We can now apply~\cite[Lemma~10.1]{cmoVI}.  
None of (i), (ii), or (v) of that lemma holds.  
If (iii) holds, then $a_0$ is in a $4$-cocircuit with $h$; a \cn\ to~\ref{badthings}.  
Thus (iv) holds, so $M\ba h,z_0/y_0$ or $M\ba h,z_0/x_0$ has an $N$-minor.  
Lemma~\ref{stringswitch} implies that $M\ba h/f$ has an $N$-minor; a \cn.  {   We conclude that \ref{mhf} holds. }

{   Since $M$ has no quick win, we know that $M\ba h/f$  is not \ifc.  We now apply
Lemma~\ref{realclaim1} to the bowtie string $\{a_1,g,b_0\}, \{b_0,g,a_0,f\},\{a_0,f,h\}, \{f,h,x_0,y_0\}$ to obtain that $M$ has $\{a_0,h\}$ contained in a $4$-cocircuit; 
 a \cn\ to~\ref{badthings}.}  
\end{proof}

\begin{figure}[tb]
\center
\includegraphics{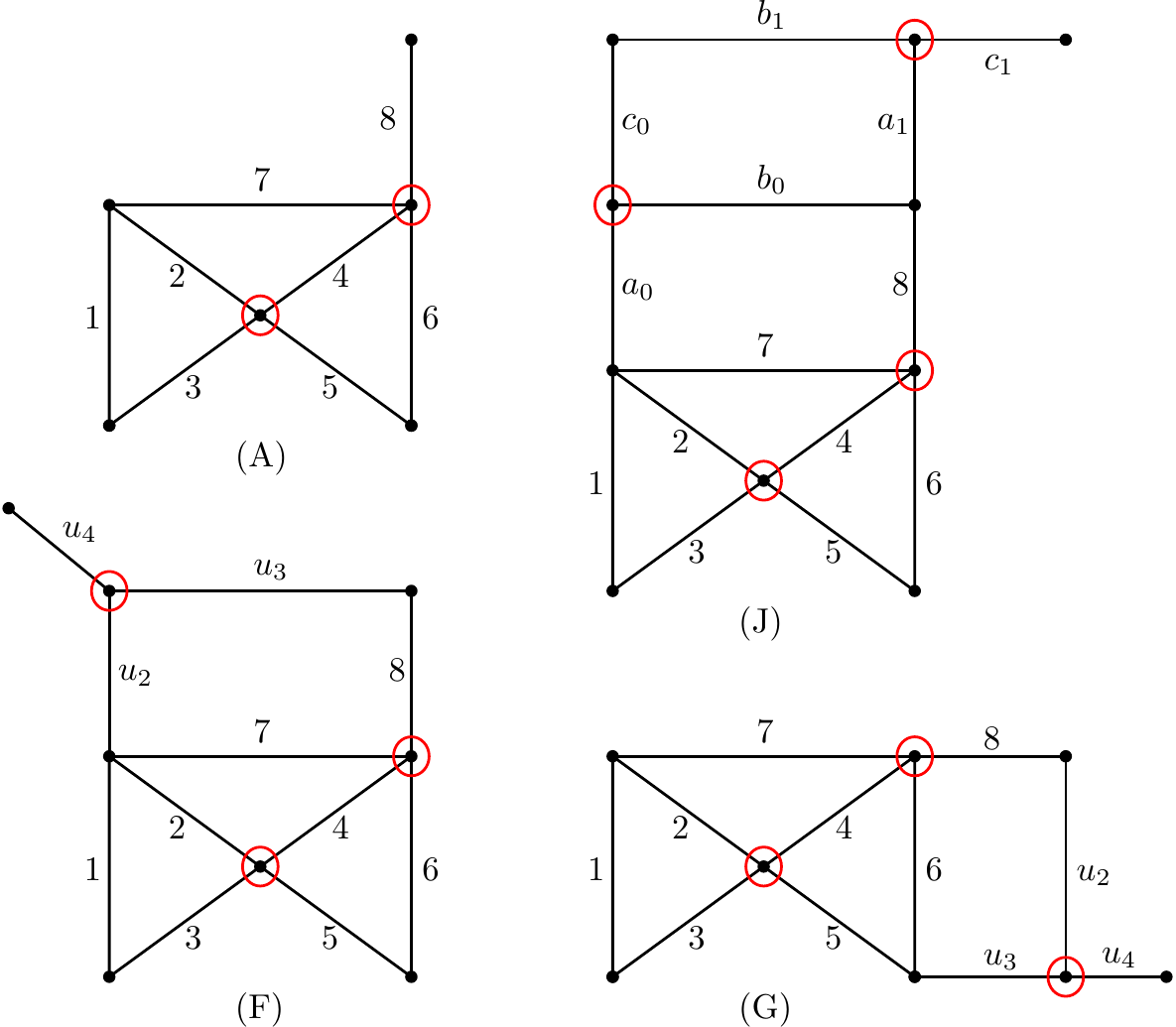}
\caption{}
\label{Afiga}
\end{figure}

\begin{lemma}
\label{AconfigBOOM}
Let $M$ and $N$ be \ifc\ binary matroids such that $|E(M)|\geq 16$ and $|E(N)|\geq 7$.  
Suppose Hypothesis VII holds.  
Suppose that $M$ contains configuration (A) labelled as in Figure~\ref{Afiga} where $M\ba 4$ is \ffsc\ with an $N$-minor, but $N\not \preceq M\ba 1,4$. Then 
\begin{itemize}
\item[(i)] $M$ has a quick win; or 
\item[(ii)] $M^*$ has an open-rotor-chain win, a ladder win, or an enhanced-ladder win; or
\item[(iii)] deleting the central cocircuit of an augmented $4$-wheel in $M^*$ gives an \ifc\ matroid with an $N^*$-minor.  
\end{itemize}
\end{lemma}

\begin{proof} 
Assume the lemma does not hold.  
First we show that 
\begin{sublemma}
\label{mcon8}
$M\ba 6/8$ and $M/8$ are \thc, and $M/5\ba 6/8$ is \thc\ with an $N$-minor.
\end{sublemma}

{   Clearly $M/5\ba 4 \cong M/5 \ba 6$. 
As $N \not \preceq M\ba 1,4$, it follows by Lemma~\ref{airplane} that $M/5\ba 6$ is \ffsc\ having an $N$-minor.} 
Since $M/5\ba 6$ has $(2,4,7,8)$ as a $4$-fan, both $M/5\ba 6/8$ and $M/5\ba 6\ba 2$ are \thc, and at least one of these matroids has an $N$-minor. 
As $M/5\ba 6,2 \cong M/5\ba 4,2 \cong M\ba 4,2/3 \cong M/3 \ba 4,1$, and  $N \not \preceq M\ba 1,4$, we deduce that $N\preceq M/5\ba 6/8$. 

By Lemma~\ref{ABCDE}, we know that $M\ba 6$ is \ffsc. As $M\ba 6$ has $(2,4,7,8)$ as a $4$-fan, $M\ba 6/8$ is \thc.  Thus $M/8$ is \thc\ unless $M$ has a triangle containing 
$\{6,8\}$.  In the exceptional case,  $M\ba 4/5$ has a $5$-fan containing $\{2,6,7,8\}$. This 
\cn\  completes the proof of \ref{mcon8}.   
 
Next we note that 

\begin{sublemma}
\label{mcon8sfc}
$M/8$ is \sfc.
\end{sublemma}

Suppose that  $M/8$ has a \ns\ \ths\ $(U,V)$. Then, by \cite[Lemma 3.3]{cmoV}, we may assume that $\{1,2,\ldots,7\} \subseteq U$. Thus we can adjoin $8$ to $U$ to obtain a \ns\ \ths\ of $M$; a \cn. Thus \ref{mcon8sfc} holds. 

Next we show that 
\begin{sublemma}
\label{effgee}
$M$ contains one of the configurations (F) and (G) in Figure~\ref{Afiga}, 
where all of the elements shown are distinct, and every $4$-fan in $M/8$ has its guts element in $\{6,7\}$. 
\end{sublemma}

To see this, note that, since $N \preceq M/8$ but $M$ has no quick win, $M/8$ has a $4$-fan $(u_1,u_2,u_3,u_4)$. Thus $M$ has $\{u_1,u_2,u_3,8\}$ as a circuit and has $\{u_2,u_3,u_4\}$ as a triad. By \ort, $\{u_1,u_2,u_3\}$ meets $\{4,6,7\}$, so $u_1 \in \{4,6,7\}$. If $u_1 = 4$, then, by \ort, $2, 3$, or $5$ is in $\{u_2,u_3,u_4\}$, so $M$ has a $4$-fan; a \cn. Thus $u_1 \in \{6,7\}$. 
By construction, $8\notin\{u_2,u_3,u_4\}$.  
As $\{u_2,u_3,u_4\}$ is a triad, this set also avoids $\{1,2,\dots ,7\}$.  
Hence \ref{effgee} holds.

Next we show the following. 

\begin{sublemma}
\label{m68sfc}
If $M$ contains configuration (G), then $M\ba 6/8$ is \sfc.  
\end{sublemma}

By~\ref{mcon8}, we know that $M\ba 6/8$ is \thc.  
Let $(U,V)$ be a \ns\ \ths\ of $M\ba 6/8$. Then we may assume that $\{u_2,u_3,u_4\} \subseteq U$. Thus $(U \cup 6,V)$ is a \ns\ \ths\ of $M/8$; a \cn\ to \ref{mcon8sfc}. Hence \ref{m68sfc} holds.

\begin{sublemma}
\label{mgf}
If $M$ contains configuration (G), then $M$ contains configuration (F). 
\end{sublemma}

To see this, observe, since 
 $M\ba 6/8$ is \sfc\ with an $N$-minor,    
but (i) does not hold, $M\ba 6/8$  has a $4$-fan $(v_1,v_2,v_3,v_4)$. Suppose first that 
$\{v_2,v_3,v_4\}$ is a triad of $M$. 
Then the $4$-fan is a fan in $M/8$, and~\ref{effgee} implies that $v_1=7$.  
Hence $M$ contains configuration (F) where $v_i$ replaces $u_i$ for each $i$ in $\{1,2,3\}$.

We may now assume that $\{v_2,v_3,v_4,6\}$ is a cocircuit of $M$. 
Lemma~\ref{bowwow} implies that $\{v_2,v_3,v_4\}$ avoids $\{1,2,3\}$.  
Orthogonality implies that $\{4,5\}$ meets $\{v_2,v_3,v_4\}$. Suppose $4 \in \{v_2,v_3,v_4\}$. Then, by \ort, $\{2,7\}$ meets $\{v_2,v_3,v_4\}$. 
Thus $\{v_2,v_3,v_4\}$  contains $\{4,7\}$, so $\{6,v_2,v_3,v_4\} = \{6,4,7,8\}$; a \cn. We conclude that $4 \not \in \{v_2,v_3,v_4\}$. Thus $5 \in \{v_2,v_3,v_4\}$. 

By Lemma~\ref{ABCDE}, we know that $\{4,5,6\}$ is the only triangle containing $5$.  
Suppose that  $5\in\{v_2,v_3\}$. Then, without loss of generality, $\{v_1,v_2,5,8\}$ is a circuit.  
In this case, \ort\ with $\{2,3,4,5\}$ and $\{4,6,7,8\}$ implies that either $4\in\{v_1,v_2\}$,  or $\{v_1,v_2\}=\{7,z\}$, for some $z$ in $\{2,3\}$.  
If $4\in\{v_1,v_2\}$, then $\{v_1,v_2,5,8\}\btu\{4,5,8,u_2,u_3\}$ is a triangle meeting $\{u_2,u_3\}$; a \cn.  
Thus $\{v_1,v_2\}=\{7,z\}$, for some $z\in\{2,3\}$, and $\lambda (\{1,2,\dots ,8\})\leq 2$; a \cn.  
We conclude that  $5\notin\{v_2,v_3\}$,  so $5=v_4$.  
Moreover, by \ort\ between $\{6,8,u_2,u_3\}$ and $\{v_2,v_3,5,6\}$, we deduce that $\{u_2,u_3\}$ meets $\{v_2,v_3\}$. 
Thus, we may assume that $v_3=u_3$.  
As $u_3$ is not in a triangle of $M$, the set $\{v_1,v_2,u_3,8\}$ is a circuit.  
By \ort, $\{4,7\}$ and $\{u_2,u_4\}$ meet $\{v_1,v_2\}$.  
Hence $v_2\in\{4,7,u_2,u_4\}$.  
By \ort\ between $\{v_2,v_3,5,6\}$ and the  circuits $\{2,4,7\}$ and $\{6,8,u_2,u_3\}$, we know that $v_2\notin\{{   4,}7,u_2\}$.  
Hence $v_2=u_4$.  
Then $\{u_4,u_3,5,6\}\btu \{u_2,u_3,u_4\}$, which is $\{u_2,5,6\}$, is a triad; a \cn.  
We conclude that \ref{mgf} holds. 

We may now assume that $M$ contains configuration (F).  
We relabel $(u_2,u_3,u_4)$ as $(a_0,b_0,c_0)$, for reasons that will become clear later.  
We show next that 
\begin{sublemma}
\label{f568}
{   $M/5\ba 6/8/c_0$ has an $N$-minor but  $M/5\ba 6/8\ba 7$ does not.}
\end{sublemma}

As $M/5\ba 6/8$  has an $N$-minor and has  $(7,a_0,b_0,c_0)$ as a $4$-fan, we deduce that $N \preceq M/5\ba 6/8\ba 7$  or $N \preceq M/5\ba 6/8/c_0$. In the first case, as 
$M/5\ba 6/8\ba 7 \cong M/5\ba 6,7/4$, we see that $N\preceq M/5,4$. Thus, by Lemma~\ref{con45}, we deduce that $N\preceq M\ba 4,1$; a \cn. We conclude that \ref{f568} holds.

From \ref{f568}, $N \preceq M/c_0$.  
Since $c_0$ is in a triad of $M$, we see that $M/c_0$ is $3$-connected. 

\begin{sublemma}
\label{fu4}
$M/c_0$ is \sfc. 
\end{sublemma}

To show this, suppose that $(U,V)$ is a \ns\ \ths\ of $M/c_0$. Then, by \cite[Lemma 3.3]{cmoV}, we may assume that $\{1,2,\ldots,7\} \subseteq U$. Hence we may also assume that $8 \in U$. Now if $a_0$ or $b_0$ is in $U$, then we may assume that both are in $U$, 
in which case 
  we can adjoin $c_0$ to $U$ to get a \ns\ \ths\ of $M$; a \cn. Hence $\{a_0,b_0\} \subseteq V$ and we can adjoin $c_0$ to $V$ to get a \ns\ \ths\ of $M$; a \cn.  Thus \ref{fu4} holds. 

{   Next we observe the following.

\begin{sublemma}
\label{jif}
If $a_0$ or $b_0$ is the guts element of a $4$-fan of $M/c_0$, then, up to switching the labels on $a_0$ and $b_0$, the matroid $M$ contains structure (J) in Figure~\ref{Afiga} where all of the elements shown are distinct.
\end{sublemma}

We need only check that the elements are distinct. Clearly $\{a_1,b_1,c_1\}$ avoids $\{1,2,\ldots,7,b_0,c_0\}$. If $\{a_1,b_1,c_1\}$ meets $\{a_0,8\}$, then it contains this set. Now $a_0 \not\in \{a_1,b_1\}$ as $M$ is binary, so $a_0 = c_1$. Then $8 \in \{a_1,b_1\}$ and we get a \cn\ to \ort. Hence \ref{jif} holds.

We now show that

\begin{sublemma}
\label{jif2}
$M$ contains structure (J) in Figure~\ref{Afiga} where all of the elements shown are distinct and the labels on $a_0$ and $b_0$ may be interchanged.
\end{sublemma}

As $M$ has no quick win, by \ref{fu4}, $M/c_0$ has a $4$-fan $(\alpha,\beta,\gamma,\delta)$. Thus $M$ has $\{\alpha,\beta,\gamma,c_0\}$ as a circuit. By \ort, $\{a_0,b_0\}$ meets $\{\alpha,\beta,\gamma\}$. If $\alpha \in \{a_0,b_0\}$, then the result follows by \ref{jif}.  Thus we may assume that $\gamma = b_0$. Then $\{\beta,\gamma,\delta\}$ contains exactly two elements of $\{a_0,b_0,7,8\}$. Now $7 \not\in \{\beta,\gamma,\delta\}$, and $\{\beta,\gamma,\delta\}$ does not contain $\{a_0,b_0\}$, so $8 \in \{\beta,\delta\}$.

As the next step towards proving \ref{jif2}, we now show that

\begin{sublemma}
\label{jif3}
$M$ does not have $\{2,5,a_0,c_0\}$ or $\{6,8,b_0,c_0\}$ as a circuit.
\end{sublemma}

Since $M$ has $\{a_0,b_0,2,5,6,8\}$ as a circuit and
$\{a_0,b_0,2,5,6,8\} \bigtriangleup \{2,5,a_0,c_0\} = \{6,8,b_0,c_0\}$, it suffices to prove that $\{2,5,a_0,c_0\}$ is not a circuit. Assume otherwise. By \ref{f568}, we know that $M/5,8,c_0\ba 6$ has an $N$-minor. Hence so does $M/5,8,c_0\ba 6,2$. But
$M/5,8,c_0\ba 6,2\cong M/5,8,c_0\ba 4,2 \cong  M\ba 4,2/3,8,c_0
\cong  M/3,8,c_0\ba 4,1$, so $N\preceq M\ba 1,4$; a \cn. Hence \ref{jif3} holds.

Recall that $8 \in \{\beta,\delta\}$. Suppose that $8 = \beta$. Then, by \ort\ between $\{\alpha,8,b_0,c_0\}$ and the cocircuits $\{4,6,7,8\}$ and $\{2,3,4,5\}$, it follows that $\alpha \in \{6,7\}$. As $\{7,8,a_0,b_0\}$ is also a circuit, $\alpha \neq 7$, so $\alpha = 6$ and we have a \cn\ to \ref{jif3}. We deduce that $8 \neq \beta$, so $8 = \delta$. Hence $M/c_0$ has $(\alpha,\beta,b_0,8)$ as a $4$-fan.

\begin{sublemma}
\label{jif5}
$M$ has $\{\alpha,\beta,b_0,c_0\}$ as a circuit and has $\{\beta,b_0,8\}$ and $\{a_0,b_0,c_0\}$ as triads, these being its only triads containing $b_0$.
\end{sublemma}

The first part of this is immediate. By \ort, a triad containing $b_0$ must contain $a_0$ or $8$, so the last part also holds.

By \ref{f568}, we see that $N \preceq M/c_0,8$. Still aiming at obtaining \ref{jif2}, we show next that

\begin{sublemma}
\label{jif4}
$M/c_0,8$ is $3$-connected.
\end{sublemma}

Assume the contrary. As $M/c_0$ is $3$-connected having $(\alpha,\beta,b_0,8)$ as a $4$-fan, we deduce that $M/c_0$ has $\{8,b_0\}$ or $\{8,\beta\}$ in a triangle. Suppose $\{8,\beta\}$ is in a triangle of $M/c_0$. Then $\{8,\beta,c_0\}$ is contained in a $4$-circuit of $M$, which, by \ort, must be $\{8,\beta,c_0,a_0\}$ or $\{8,\beta,c_0,b_0\}$. But $M$ has $\{\alpha,\beta,b_0,c_0\}$ as a circuit. As $\alpha \neq 8$, it follows that $\{8,\beta,c_0,a_0\}$ is a circuit. By taking the symmetric difference of the last two circuits, we deduce that
$\{\alpha,8,a_0,b_0\}$ is a circuit, so $\alpha = 7$. Then $\{7,\beta,b_0,c_0\}$ is a circuit, so $\beta \in \{4,6\}$; a \cn\ to \ort. We conclude that $\{8,\beta\}$ is not in a triangle of $M/c_0$. Thus $M/c_0$ has $\{8,b_0\}$ contained in a triangle, so $M$ has $\{8,b_0,c_0\}$ contained in a $4$-circuit. By \ort, this $4$-circuit is $\{8,b_0,c_0,6\}$; a \cn\ to \ref{jif3}. We conclude that \ref{jif4} holds.

By \ref{fu4}, $M/c_0$ is sequentially $4$-connected. It follows that $M/c_0,8$ is sequentially $4$-connected for if $(U,V)$ is a non-sequential $3$-separation of the last matroid, then we may assume that $\{\alpha,\beta,b_0\} \subseteq U$, so $(U \cup 8,V)$ is a non-sequential $3$-separation of $M/c_0$; a \cn.

Since $M$ has no quick win, $M/c_0,8$ has a $4$-fan $(s_1,s_2,s_3,s_4)$. Thus $M$ has a circuit $C$ such that
$\{s_1,s_2,s_3\} \subsetneqq C \subseteq \{s_1,s_2,s_3,c_0,8\}$,  and $\{s_2,s_3,s_4\}$ avoids $\{1,2,\ldots,7,8,c_0\}$. Suppose $c_0 \not\in C$. Then $C = \{s_1,s_2,s_3,8\}$. By \ort, $s_1 \in \{4,6,7\}$. But $s_1 \neq 4$ otherwise $\{s_2,s_3\}$ meets $\{2,3,5\}$. Thus $s_1 \in \{6,7\}$. Now $M$ has $\{\beta,b_0,8\}$ as a triad, so $\{s_2,s_3\}$ meets $\{\beta,b_0\}$ in a single element. The triad $\{s_2,s_3,s_4\}$ avoids $\{c_0,8\}$ so differs from both $\{a_0,b_0,c_0\}$ and $\{b_0,\beta,8\}$. Hence, by \ref{jif5}, $b_0 \not\in \{s_2,s_3,s_4\}$, so
$\beta \in \{s_2,s_3\}$. Thus, by \ort\ between $\{s_2,s_3,s_4\}$ and $\{\alpha,\beta,b_0,c_0\}$, we deduce that $\alpha \in \{s_2,s_3,s_4\}$. It follows that $M/c_0$ has a $4$-fan having $b_0$ as its guts element and $\{s_2,s_3,s_4\}$ as its triad; a \cn\ to \ref{jif}. We conclude that $c_0 \in C$.

Suppose $C = \{s_1,s_2,s_3,c_0\}$. By \ort\ and \ref{jif}, $\{s_2,s_3\}$ meets $\{a_0,b_0\}$. As $\{s_2,s_3,s_4\}$ is not $\{a_0,b_0,c_0\}$, it follows by \ort\ that $\{s_2,s_3,s_4\}$ meets $\{7,8\}$; a \cn. We deduce that $C \neq \{s_1,s_2,s_3,c_0\}$, so $C = \{s_1,s_2,s_3,c_0,8\}$. Then, by \ort, $s_1 \in \{4,6,7\}$, and $\{s_2,s_3\}$ meets $\{a_0,b_0\}$. By \ref{jif5}, $b_0 \not\in \{s_2,s_3,s_4\}$ so $a_0 \in \{s_2,s_3\}$. Thus, by \ort\ between $\{s_2,s_3,s_4\}$ and $\{a_0,b_0,7,8\}$, we obtain a \cn. We conclude that \ref{jif2} holds. }

{   We may now assume that $M^*$ has $(T_0,T_1,D_0)$ as a bowtie.
In $M^*$, take a right-maximal bowtie string $T_0,D_0,T_1,D_2,\dots ,T_n$ noting that  $n$ may equal $1$. Next we show the following.

\begin{sublemma}
\label{distink}
The elements in $\{1,2,\dots ,8\}\cup T_0\cup T_1\cup\dots\cup T_n$ are  distinct.
\end{sublemma}

The  triangles in this string are triads in $M$, so the elements in the string avoid $\{1,2,\dots ,7\}$. Therefore \ref{distink} holds unless either $a_0 = c_n$, or $8$ is in the bowtie string. Suppose $8 \in T_i$. Then, since $8 \not\in T_0$, \ort\ between $T_i$ and $\{7,8,a_0,b_0\}$ implies that $a_0 = c_i$ and $i = n$. Then $8 \in \{a_n,b_n\}$. By \ort, $\{b_{n-1},c_{n-1},a_n,b_n\}$ meets $\{4,6,7\}$,
 so a triangle in $M$ meets a triangle in $M^*$; a \cn.  We conclude that $8$ is not in the bowtie string. By \ort\ between $T_n$ and the cocircuit $\{7,8,a_0,b_0\}$ in $M^*$, we see that $c_n \neq a_0$. Thus \ref{distink} holds. }

{  We want to apply Lemma~\ref{beachtheend} to $M^*$. Since $M^*$ contains both triangles and triads, it is not isomorphic to the cycle matroid of a quartic M\"{o}bius ladder. Moreover, if $M^*$ is isomorphic to the cycle matroid of a terrahawk, then so is $M$, and therefore $M$ has a second triangle containing $5$; a \cn\ to Lemma~\ref{ABCDE}. Let $X= \{5,8\}$, let $Y= \{6\}$, and let $M' = M^*\ba X/Y$. In $M'$, each $T_i$ is a disjoint union of circuits while each $D_j$ is a disjoint union of cocircuits. Since $M'$ is $3$-connected, it follows that $T_0,D_0,T_1,D_2,\dots ,T_n$ is a bowtie string in $M'$. }

\begin{sublemma}
\label{whatn*}
$M'\ba c_0,c_1,\dots ,c_n$ has an $N^*$-minor and $M'\ba c_0,c_1,\dots ,c_i/e$ has no $N^*$-minor  for all $e$ in $\{a_i,b_i\}$ where $i\in\{1,2,\dots ,n\}$.  
\end{sublemma}

{   To see this, observe first that, by \ref{f568}, $M'\ba c_0$ has an $N^*$-minor but $M'/7$ does not. It follows from this that $M'\ba a_0/b_0$ has no $N^*$-minor since $M'\ba a_0$ has $\{b_0,7\}$ as a cocircuit, so $M'\ba a_0/b_0 \cong M'\ba a_0/7$. }
 Suppose $M'\ba c_0,c_1,\dots ,c_i$ has an $N^*$-minor for some $i$ in $\{0,1,\dots ,n-1\}$.  
As this matroid has $(c_{i+1},a_{i+1},b_{i+1},b_i)$ as a $4$-fan, either $M'\ba c_0,c_1,\dots ,c_{i+1}$ has an $N^*$-minor, or  $M'\ba c_0,c_1,\dots ,c_i/b_i$ has an $N^*$-minor.  
{   By Lemma~\ref{stringswitch}, the latter implies that $M'\ba a_0,a_1,\dots ,a_i/b_0$ has an $N^*$-minor. Hence so does $M'\ba a_0/b_0$; a \cn.  
 It follows by induction that $M'\ba c_0,c_1,\dots ,c_n$ has an $N^*$-minor. }
Furthermore, if $M'\ba c_0,c_1,\dots ,c_i/a_i$ has an $N^*$-minor for some $i$  in $\{1,2,\dots ,n\}$, then Lemma~\ref{stringswitch} implies that $M'\ba a_0,a_1,\dots ,a_{i-1},b_i/b_0$ has an $N^*$-minor, so $M'\ba a_0/b_0$ has an $N^*$-minor; a \cn. Thus~\ref{whatn*} holds.

Clearly none of (i), (ii), and (vii) of Lemma~\ref{beachtheend} holds in $M^*$.  

\begin{sublemma}
\label{not3or5}
{  Neither (iii) nor (v) of Lemma~\ref{beachtheend} holds in $M^*$.} 
\end{sublemma}

To see this, observe that, by~\ref{whatn*}, it follows that  (v) of Lemma~\ref{beachtheend} does not hold. 
Suppose (iii) of Lemma~\ref{beachtheend} holds in $M^*$. Then $\{a_0,b_0,z,c_n\}$ is a $4$-circuit of $M$ for some $z$ in $\{a_n,b_n\}$.  
Taking the symmetric difference of this circuit with $\{7,8,a_0,b_0\}$, we see that $\{7,8,z,c_n\}$ is   a  circuit of $M$.  
But $M/5\ba 6/8,c_n$ has an $N$-minor and has $7$ in a $2$-circuit. Hence $M/5\ba 6/8,c_n\ba 7$ has an $N$-minor; a \cn\ to~\ref{f568}.  {  We deduce that \ref{not3or5} holds.}

Suppose $M^*$ contains the structure in Figure~\ref{goofy} where $d_1\in Y$ and either $d_0\in X$ or $M'\ba c_0,c_1,d_0$ has an $N^*$-minor.  
Then $d_1=6$, and $\{d_0,a_1,c_1,6\}$ is a circuit in $M$.  
By \ort\ with $\{4,6,7,8\}$, we know that $d_0\in\{4,7,8\}$.  
But $\{b_0,d_0,a_1\}$ is a triad of $M$, so it avoids $\{4,7\}$.  
Thus $d_0=8$.  
The symmetric difference $\{4,5,6\}\btu\{6,8,a_1,c_1\}$ is  $\{4,5,8,a_1,c_1\}$, which must be a circuit  in $M$.  
Now $M/5\ba 6/8/c_0/c_1$ has $\{4,a_1\}$ as a circuit.  
Hence $M/5\ba 6/8/c_0,c_1\ba a_1$ has an $N$-minor and (v) of Lemma~\ref{beachtheend} holds; a \cn.  {   Thus $M^*$ does not contain the structure in Figure~\ref{goofy}.}

{  Suppose $M^*$ contains the structure in Figure~\ref{fign=2}(a).  
Then $M$ has $\{p,t,a_0\},\{a_0,b_0,c_0\}$, and $\{b_0,b_1,q\}$ as distinct cocircuits. } 
By \ort\ with the circuit $\{7,8,a_0,b_0\}$, each of $q$ and $\{p,t\}$ meets $\{7,8\}$ in a unique element.  
As $7$ is in no triad of $M$, we deduce that $q=8$ and $8\in\{p,t\}$. Hence $q\in\{p,t\}$; a \cn.  
Suppose next that $M^*$ contains the structure in Figure~\ref{fign=2}(b).  
Then $\{b_0,b_1,q\}$ and $\{s_1,s_2,s_3\}$ are disjoint cocircuits of $M$ and $\{b_1,c_1,q,s_2,s_3\}$ is a circuit of $M$.  
By \ort\ with the circuit $\{7,8,a_0,b_0\}$, it follows that $q=8$.  
By \ort\ with the cocircuit $\{4,6,7,8\}$, we see  that $\{b_1,c_1,s_2,s_3\}$ meets $\{4,6,7\}$; a \cn.  {  We conclude that $M^*$   contains neither of the structures in 
Figure~\ref{fign=2}.}

By Lemma~\ref{beachtheend}, it follows that $M^*$ contains the structure in Figure~\ref{drossfigiinouv} and $\{d_{n-1},d_n\}$ is not contained in a triangle of $M^*$.   
By~\cite[Lemma~6.4]{cmoVI}, as $a_0\neq c_n$ and $\{d_{n-1},d_n\}$ is not contained in a triangle, we know that the elements in this figure are all distinct with the possible exception that $\al$ and $\be$ may be repeated elements.  
Thus $\{\al,\be,a_0\}$ is a triangle of $M^*$ distinct from $T_0$, and $\{\al,a_0,c_0,d_0\}$ or $\{\al,a_0,c_0,a_1,c_1\}$ is a cocircuit of $M^*$ for some elements $\al,\be$, and $d_0$.  
Furthermore, $\{c_0,d_0,a_1\}$ is a triangle of $M^*$.  
Since $M^*$ has $\{7,8,a_0,b_0\}$ as a cocircuit, \ort\ implies that $\{\al,\be\}$ meets $\{7,8\}$.  
Clearly $7$ avoids $\{\al,\be\}$, so $8\in\{\al,\be\}$.  

{  
\begin{sublemma}
\label{beta8}
$\beta = 8$.
\end{sublemma}

To show this, suppose that $\alpha  = 8$. 
Then \ort\ implies that $\{4,6,7\}$ meets $\{a_0,c_0,d_0\}$ or $\{a_0,c_0,a_1,c_1\}$. Thus $M^*$ is not \ifc; a \cn.   
Thus \ref{beta8} holds. }
  
We relabel $7$ as $\ga$ to see that $M^*$ contains the structure in Figure~\ref{drossfig2}, where the elements are all distinct with the possible exception that $\al,\be$, and $\ga$ may be repeated elements.  

{  In preparation for applying Lemma~\ref{moreteeth}, we now show the following.}

\begin{sublemma}
\label{distink2}
The elements in Figure~\ref{drossfig2} are   distinct  except that $\ga$ and $d_n$ may be equal.
\end{sublemma} 

As $\ga$ is in a triad of $M^*$, we know that $\ga$ avoids all of the other elements in Figure~\ref{drossfig2} with the possible exception of $d_n$.  
By \ort\ between the triangles in this figure and {  the cocircuits $\{\be,\ga,a_0,b_0\}$ 
and $\{\alpha,a_0,c_0,d_0\}$, we deduce that $\be$ and $\alpha$ avoid all of the elements with the possible exception of $d_n$.}  
Thus the elements are all distinct except that $d_n$ may be in $\{\al,\be,\ga\}$.  
But \ort\ between $\{d_{n-1},a_n,c_n,d_n\}$ and $\{\al,\be, a_0\}$ implies that $d_n\notin\{\al,\be\}$.  
Thus \ref{distink2} holds. 

By \ref{not3or5} and Lemma~\ref{beachtheend}, $M^*$ has no triangle $T_{n+1}$ such that $\{x,c_n,a_{n+1},b_{n+1}\}$ is a $4$-cocircuit for any $x$ in $\{a_n,b_n\}$.  
Then Lemma~\ref{6.3rsv} implies that $M^*\ba c_n$ is \ffsc.  

By \ref{whatn*},  $M^*\ba \be ,c_0,c_1,\dots ,c_n$ has an $N$-minor.  
By \ref{distink2}, we are now in a position to apply Lemma~\ref{moreteeth} to get  that $M^*\ba \be,c_0,c_1,\dots ,c_n$ is \ffsc\ and every $4$-fan of this matroid is either a $4$-fan in $M^*\ba c_n$ with $b_n$ as its coguts element,  or is a $4$-fan in $M^*\ba \be$ with $\al$ as its coguts element.  
Let $(u,v,w,x)$ be such a $4$-fan.  
{  Suppose first that $x = b_n$ and this $4$-fan is a $4$-fan in $M^*\ba c_n$. Then 
 $\{v,w,b_n,c_n\}$ is   a $4$-cocircuit of $M^*$ and, taking $T_{n+1} = \{u,v,w\}$, we get a \cn\ to the previous paragraph.  }
It follows that  $(u,v,w,\al)$ is a $4$-fan in $M^*\ba \be$.  
By \ref{beta8}, we know that $\be =8$. Thus $(\al,w,v,u)$ is a $4$-fan in $M/8$; a \cn\ to~\ref{effgee}.  
We conclude that $M^*\ba \be,c_0,c_1,\dots ,c_n$ has no $4$-fans and so is \ifc.  
Thus (ii) holds; a \cn.  
\end{proof}
}

This completes our analysis of the case when $M$ contains configuration (A) in Figure~\ref{ABCDEfig}. 

\section{Configuration (B)}
\label{Bconfig}
 
In this section, we deal with the case when $M$ contains configuration (B) in Figure~\ref{ABCDEfig}. The results from the last two sections mean that if we find that $M$ contains configuration (C) or (A) from Figure~\ref{ABCDEfig}, then we are guaranteed to get one of the desired outcomes from the main theorem.

\begin{figure}[htb]
\center
\includegraphics{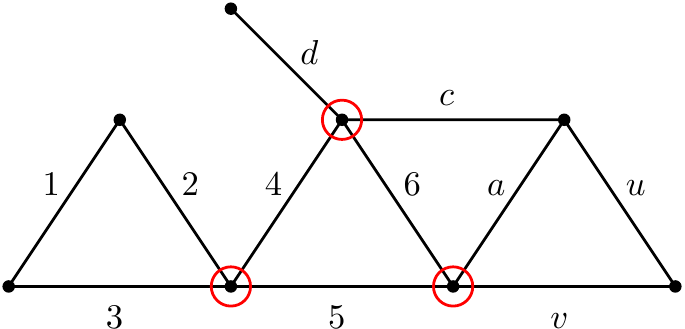}
\caption{All of the elements are distinct except that $u$ may be $1$.}
\label{BAfiga}
\end{figure}

\begin{lemma}
\label{BAlemma} 
Suppose $M$ and $N$ are \ifc\ binary matroids, $|E(M)|\geq 13$ and $|E(N)|\geq 7$, and $M$ contains structure (B) in Figure~\ref{ABCDEfig} where all of the elements are distinct except that $1$ may be $a$. 
Suppose $M\ba 4$ is \ffsc\ with an $N$-minor but $M\ba 1,4$ does not have an $N$-minor. 
Then 
\begin{itemize}
\item[(i)] 
$M$ has a triangle  $\{7,8,9\}$ where $\{4,6,7,8\}$ is a cocircuit and the elements in $\{1,2,\dots ,9\}$ are distinct except that $1$ may be $9$; or 
\item[(ii)] $M$ contains the configuration in Figure~\ref{BAfiga}, where all of the elements are distinct except that $u$ may be $1$, and $M\ba 6$ is \ffsc\ with an $N$-minor; or 
\item[(iii)] $M\ba 6$ is \ifc\ with an $N$-minor.
\end{itemize}
\end{lemma}

\begin{proof}
{  Suppose that the lemma fails.  }
As $N\not \preceq M\ba 1,4$, it follows by Lemma~\ref{airplane} that $N\preceq M\ba 4/5$, and $M\ba 4/5$ is \ffsc. Since $M/5\ba 4 \cong M/5\ba 6$, we deduce that $N\preceq M\ba 6$.

\begin{sublemma}
\label{m6ffsc} 
$M\ba 6$ is \ffsc\ with an $N$-minor.
\end{sublemma}

Assume that~\ref{m6ffsc} fails. Then, by Lemma~\ref{6.3rsv}, $\{4,5,6\}$ is the central triangle of a quasi rotor.   Moreover, since $M\ba 4$ is \ffsc, the central element of this quasi rotor is $4$, and Lemma~\ref{6.3rsv} specifies that this quasi rotor is $(\{1,2,3\},\{4,5,6\},\{7,8,9\},\{2,3,4,5\},\{4,6,7,8\},\{x,4,7\})$ for some $x$ in $\{2,3\}$, so (i) holds; a \cn.  
We conclude that \ref{m6ffsc} holds.

Since $M\ba 6$ is not \ifc,  
it has a $4$-fan $(u,v,w,x)$. Thus $M$ has $\{v,w,x,6\}$ as a cocircuit. Hence $\{v,w,x\}$ meets $\{4,5\}$ and $\{a,c\}$. 
Clearly $\{u,v,w\} \neq \{4,5,6\}$, so {  Lemma~\ref{ABCDE} implies that $5 \not\in \{u,v,w\}$}. 
{  Suppose $4 \in \{u,v,w\}$. Then \ort\ between $\{u,v,w\}$ and   $\{4,6,c,d\}$ implies that $\{c,d\}$ meets $\{u,v,w\}$.  
If $\{4,c\}$ is in a triangle, then the symmetric difference of this triangle with $\{4,5,a,c\}$ is a triangle other than $\{4,5,6\}$ that contains $5$; a \cn\ to Lemma~\ref{ABCDE}.  
Thus $\{4,d\}\subseteq \{u,v,w\}$, so $M\ba 6/5$ has a $5$-fan; a \cn.  
 We conclude that $4 \not\in \{u,v,w\}$, so $x \in \{4,5\}$ and, without loss of generality, $w$ 
 in $\{a,c\}$.

\begin{sublemma}
\label{123uvw}
Either $\{1,2,3\}$ avoids $\{u,v,w\}$, or $\{1,2,3\} \cap \{u,v,w\} = \{1\} = \{u\}$.
\end{sublemma}

To see this, observe that, as $(\{u,v,w\},\{4,5,6\},\{v,w,x,6\})$ is a bowtie, Lemma~\ref{bowwow} implies that the cocircuit $\{2,3,4,5\}$ avoids $\{u,v,w\}$.  
Furthermore, Lemma~\ref{bowwow} applied to the bowtie $(\{1,2,3\},\{4,5,6\}, \{2,3,4,5\})$ implies that $\{v,w,x,6\}$ avoids $\{1,2,3\}$.  
It follows that \ref{123uvw} holds. }

By \ref{123uvw}, if $x = 4$, then {  (i)} of the lemma holds; a \cn.  Thus we may assume that $x = 5$. 
Suppose $w=c$. 
Then $\{c,d\}$ is not in a triangle, as {  (i)} does not hold, so \ort\ implies that {  $4\in\{u,v\}$, which we have already shown does not occur.}  
We conclude that $a = w$. 
Then $\{v,a,5,6\}$ is a cocircuit of $M$, so $M$ contains the configuration in Figure~\ref{BAfiga}.  

{  By \ref{123uvw}, since $a = w$, we see that $a\neq 1$.} 
Thus the elements in Figure~\ref{ABCDEfig}(B) are distinct.  
{  Lemma~\ref{bowwow} implies that $\{u,v\}$ avoids $\{2,3,4,5\}$ and $\{4,6,c,d\}$.  
Hence the elements in Figure~\ref{BAfiga} are distinct {  with the possible exception that} $1\in\{u,v\}$.  
{  By \ref{123uvw}, $1\neq v$. Thus} 
(ii) holds.  }
\end{proof}

If (i) of the last lemma holds, then, possibly after a minor relabelling, we see that $M$ contains structure (C) from Figure~\ref{ABCDEfig}. Since this case has already been treated, it remains for us to consider the 
 structure in Figure~\ref{BAfiga}.  

\begin{lemma}
\label{preBA2}
Suppose $M$ and $N$ are \ifc\ binary matroids, $|E(M)|\geq 16$ and $|E(N)|\geq 7$, and $M$ contains the structure in Figure~\ref{BAfiga} where all of the elements are distinct except that $u$ may be $1$.  
Suppose Hypothesis {  VII} holds.  
If $M\ba 6$ is \ffsc\ with an $N$-minor but $M\ba 6,u$ has no $N$-minor, then 
\begin{itemize}
\item[(i)] $M$ has a quick win; or 
{  
\item[(ii)] $M^*$ has an open-rotor-chain win, a ladder win,  or an enhanced-ladder win; or
\item[(iii)] deleting the central cocircuit of an augmented $4$-wheel in $M^*$ gives an \ifc\ matroid with an {magenta $N^*$}-minor.  
}
\end{itemize}
\end{lemma}

\begin{proof}
As $M\ba 6,u$ has no $N$-minor, we relabel the elements $1,2,3,4,5,6,a,v,u,c$, and $d$ 
in Figure~\ref{BAfiga}, 
as $x,y,z,6,5,4,2,3,1,c$, and $d$. We then restrict our attention to the configuration in $M$ that is the same as that in Figure~\ref{ABCDEfig}(A).  
Now $M\ba 1,4$ has no $N$-minor.  
Lemma~\ref{AconfigBOOM} implies that {  the result holds.}  
\end{proof}

We may now assume that $M$ contains the structure in Figure~\ref{BAfiga} and $M\ba 6,u$ has an $N$-minor.  
Recall that $M\ba 4/5$ has an $N$-minor.  
As this matroid has $(a,6,c,d)$ as a $4$-fan, we may delete $a$ or contract $d$ keeping an $N$-minor.  
In the next lemma, we consider the second possibility.   

\begin{lemma}
\label{BAlemma2} 
Suppose $M$ is an \ifc\ binary matroid containing at least thirteen elements, and $M$ has the structure in Figure~\ref{BAfiga} where all of the elements are distinct except that $u$ may be $1$. 
Suppose that both $M\ba 4$ and $M\ba 6$ are   \ffsc. 
Then 
\begin{itemize}
\item[(i)] $M/d$ is \ffsc\ and every \ftv\ of $M/d$ is a $4$-fan with $c$ as its guts element; or
\item[(ii)] $M$ contains the structure in Figure~\ref{ABCDEfig}(A).  
\end{itemize}
\end{lemma}

\begin{proof}  Assume that (ii) does not hold. First we show that 

\begin{sublemma}
\label{dnotry} 
$M$ has no triangle containing $d$.
\end{sublemma}

Assume that $M$ has a triangle $T$ containing $d$. Then, by \ort, $T$ contains $c$, $6$, or $4$. 
In the first two cases, we obtain the \cn\ that $M\ba 4$ is not \ffsc. 
We conclude that $4 \in T$. Then, by \ort\ and symmetry, we may assume that $2 \in T$, so (ii) holds.  
This \cn\ implies  that~\ref{dnotry} holds.

\begin{sublemma}
\label{mdwelby} 
$M/d$ is \sfc.
\end{sublemma}

By \ref{dnotry}, $M/d$ is \thc. Suppose that $M/d$ has a \ns\ \ths\ $(U,V)$. Then, by \cite[Lemma 3.3]{cmoV}, we may assume that $\{4,5,6,a,u,v,c\} \subseteq U$. Then $(U \cup d,V)$ is a \ns\ \ths\ of $M$; a \cn. Thus \ref{mdwelby} holds. 

Now suppose that $M/d$ has a $4$-fan $(\al,\be,\ga,\de)$. Then $M$ has $\{\al,\be,\ga,d\}$ as a circuit. By \ort, 
$\{\al,\be,\ga\}$ meets $\{4,6,c\}$. 
Since $M$ has no $4$-fan, $\{\be,\ga\}$ avoids the triangles of $M$. Thus $\al \in \{4,6,c\}$. By \ort\ between $\{\al,\be,\ga,d\}$ and the cocircuits $\{2,3,4,5\}$ and $\{5,6,a,v\}$, we see that $\al \not \in \{4,6\}$. Thus $\al = c$. 

We now know that every $4$-fan in $M/d$ has $c$ as its guts element. It follows easily that $M$ has no $5$-fan and no $5$-cofan, so $M/d$ is \ffsc\ as required.
\end{proof}

We have not eliminated the case that $M$ contains the structure in Figure~\ref{BAfiga} and $M/d$ has an $N$-minor, but we have built up more structure, which will assist us in our later analysis.  

\begin{figure}[htb]
\center
\includegraphics{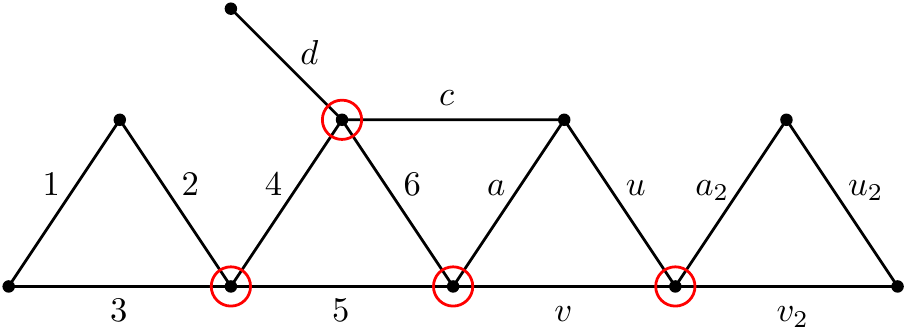}
\caption{All of the elements are distinct except that $u_2$ may be the same as $1$, {  or $\{1,2,3\}$ may be $\{a_2,u_2,v_2\}$}.}
\label{band}
\end{figure}

\begin{lemma}
\label{newtry}
Suppose $M$ and $N$ are \ifc\ binary matroids, $|E(M)|\geq 13$ and $|E(N)|\geq 7$, and $M$ contains the structure in Figure~\ref{BAfiga} where all of the elements are distinct except that $u$ may be $1$.  
Suppose that Hypothesis {  VII} holds. 
Suppose that $M\ba 4$ and $M\ba 6$ are each \ffsc\ with an $N$-minor, and $M\ba 1,4$ has no $N$-minor.  
Then 
\begin{itemize}
\item[(i)] $\{4,5,6\}$ is the only triangle of $M$ containing $4$, and $M$ contains the structure in Figure~\ref{band} where all of the elements are distinct except that $u_2$ may be the same as $1$, or $\{1,2,3\}$ may equal $\{a_2,u_2,v_2\}$; or
\item[(ii)] $M$ has a quick win; or
\item[(iii)] $M$ has a ladder win. 
\end{itemize}
\end{lemma}

\begin{proof}
We assume that {  neither (ii) nor (iii)}  holds.  We show first that 

\begin{sublemma}
\label{4try}
$\{4,5,6\}$ is the only triangle that meets $\{4,5\}$.  
\end{sublemma}

{  Lemma~\ref{ABCDE} implies that} $\{4,5,6\}$ is the only triangle that contains $5$. 
Let $T$ be a triangle that contains $4$ but differs from $\{4,5,6\}$.  
Then \ort\ implies that $T$ meets $\{2,3,5\}$ and $\{6,c,d\}$, so, up to switching the labels on $2$ and $3$, the triangle is $\{2,4,d\}$ or $\{2,4,c\}$,  so {  $M\ba 4/5$ has a $5$-fan; a \cn\ to Lemma~\ref{airplane}.  }
We deduce that~\ref{4try} holds.  

Lemma~\ref{preBA2} implies that $M\ba 6,u$ has an $N$-minor.  
From Lemma~\ref{ILK}, we know that $M\ba u$ is \ffsc\ and, as $M\ba 4$ is not \ifc, either 
\begin{itemize}
\item[(I)] $M$ has a triangle $\{a_2,u_2,v_2\}$ and a cocircuit $\{x,u,a_2,v_2\}$ where $x\in\{a,v\}$ and $|\{4,5,6,a,u,v,a_2,u_2,v_2\}|=9$, or 
\item[(II)] every \ftv\ of $M\ba u$ is a $4$-fan of the form $(6,y_2,y_3,y_4)$ where  $y_2 \in\{a,v\}$.  
\end{itemize}

Suppose (II) holds.  
Since $y_2\in\{a,v\}$, \ort\ implies that $(y_1,y_2,y_3)$ is $(6,v,d)$ or $(6,a,c)$.  
If $\{6,v,d\}$ is a triangle, then $\lambda(\{4,5,6,a,v,u,c,d\})\leq 2$; a \cn.  Thus $y_2=a$ and $y_3 =c$.  
We now consider $M\ba 6,u$.  
By~\cite[Lemma~6.1]{cmoVI}, since $M\ba 6,u$ is not \ifc, one of the following occurs: $\{c,y_4\}$ is contained in a triangle; or $\{5,v\}$ is contained in a triangle; or $M$ has a triangle that contains $4$ but avoids $\{5,6,a,c,u,v\}$; or $M\ba 6,u$ is \ffsc\ and $v$ is  the coguts element of every $4$-fan in it.  
If $\{c,y_4\}$ is contained in a triangle, then \ort\ implies that the third element of this triangle is in $\{4,6,d\}$ and so $\lambda (\{4,5,6,a,u,v,c,d,y_4\})\leq 2$; a \cn.  
By \ref{4try}, $\{5,v\}$ is not contained in a triangle and  $M$ has no triangle that contains $4$ but avoids $\{5,6\}$. We deduce  that $M\ba 6,u$ has a $4$-fan of the form $(z_1,z_2,z_3,v)$.  
As every $4$-fan of $M\ba u$ contains $6$, we see that $(z_1,z_2,z_3,v)$ is not a $4$-fan of $M\ba u$. Hence $\{6,z_2,z_3,v\}$ or $\{6,u,z_2,z_3,v\}$ is a cocircuit of $M$.  
Then \ort\ implies that $\{z_2,z_3\}$ meets $\{4,5\}$, a \cn\ to \ref{4try}.   
We conclude that (II)  does not hold.  Therefore (I) holds.

If the triangle $\{a_2,u_2,v_2\}$ meets $\{c,d\}$, then \ort\ {  with the cocircuit $\{4,6,c,d\}$} implies that $\{c,d\}\subseteq \{a_2,u_2,v_2\}$, so $M\ba 4$ has a $5$-fan; a \cn.  
Thus the elements in $\{4,5,6,a,u,v,a_2,u_2,v_2,c,d\}$ are   distinct.  Now consider the cocircuit $\{x,u,a_2,v_2\}$ recalling that $x \in \{a,v\}$. 
If $x=a$, then \ort\ implies that $\{u,a_2,v_2\}$ meets $\{6,c\}$; a \cn.  
Thus $x=v$.  
By hypothesis, $\{2,3\}$ avoids $\{4,5,6,a,c,d,u,v\}$.  
Suppose $\{2,3\}$ meets $\{a_2,u_2,v_2\}$.  
Then \ort\ with $\{2,3,4,5\}$ implies that $\{2,3\}\subseteq\{a_2,u_2,v_2\}$, so 
$\{1,2,3\} =\{a_2,u_2,v_2\}$, and (i) holds.  
Now suppose that  $\{2,3\}$ avoids $\{a_2,u_2,v_2\}$. Then the elements in $\{2,3,4,5,6,a,u,v,a_2,u_2,v_2,c,d\}$ are  distinct.  Finally, \ort\ implies that $1$ can only be in the last set if it equals $u_2$. Thus (i) holds. 
\end{proof}

When $M$ contains the configuration in Figure~\ref{band} and $M\ba 4$ is \ffsc\ with an $N$-minor but 
$M\ba 1,4$ does not have an $N$-minor, we know, by Lemma~\ref{airplane}, that $M\ba 6/5$ is \ffsc\ with an $N$-minor. Since $M\ba 6/5$ has $(a,4,c,d)$ as a $4$-fan, it follows that either 
\begin{itemize}
\item[(i)] $N\preceq M\ba 6/5\ba a$; or 
\item[(ii)] $N\preceq M\ba 6/5/d$.
\end{itemize}

As we showed in Lemma~\ref{BAlemma2}, we are able to find a new triad in the case that (ii) holds.  
In the following lemma, we dispense with the case that (i) holds.  


\begin{lemma}
\label{parti}
Suppose $M$ and $N$ are \ifc\ binary matroids, $|E(M)|\geq 15$ and $|E(N)|\geq 7$, and $M$ contains the structure in Figure~\ref{band} where all of the elements are distinct except that $u_2$ may be $1$, or $\{1,2,3\}$ may equal $\{a_2,u_2,v_2\}$.   
Suppose that Hypothesis {  VII} holds and that $\{4,5,6\}$ is the only triangle that contains $4$.  
Suppose that $M\ba 4$ and $M\ba 6$ are each \ffsc\ with an $N$-minor and $M\ba 6/5\ba a$ has an $N$-minor but $M\ba 1,4$ has no $N$-minor.  
Then 
\begin{itemize}
\item[(i)] $M$ has a quick win; or
\item[(ii)] $M$ has a ladder win; or
\item[(iii)] $M$ has a mixed ladder win.  
\end{itemize}
\end{lemma}

\begin{proof} Assume that the lemma does not hold. 
We relabel the elements $4,5,6,a,v,u,d$, and $c$  in Figure~\ref{band} as $a_0,v_0,u_0,a_1,v_1,u_1,t_0,$ and $t_1$, respectively.  
Since $M\ba u_0/v_0\ba a_1$ is isomorphic to $M\ba u_0,u_1/v_1$, the second matroid has an $N$-minor.  
We take a right-maximal bowtie string $\{a_0,v_0,u_0\},\{v_0,u_0,a_1,v_1\},\{a_1,v_1,u_1\},\{v_1,u_1,a_2,v_2\},\dots,\{a_n,v_n,u_n\}$.  
Let $X$ be the set of elements in this bowtie string.  



\begin{figure}[htb]
\center
\includegraphics{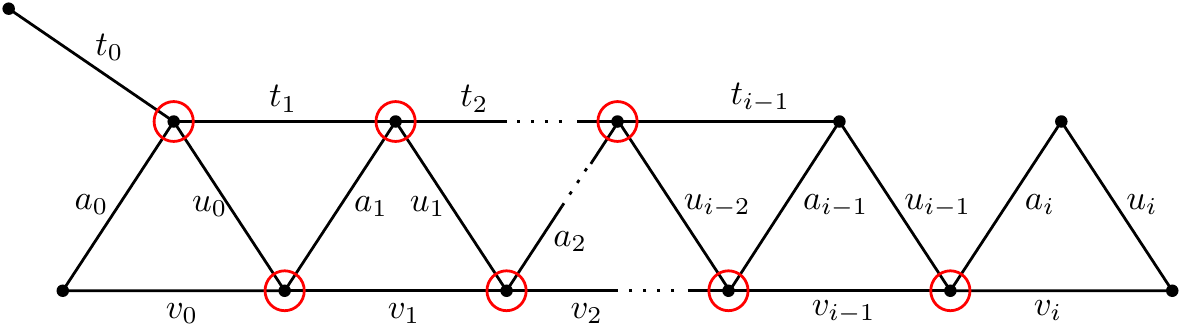}
\caption{All of the elements shown are distinct.}
\label{auvi}
\end{figure}

\begin{sublemma}
\label{newsub}
Suppose $M$ contains the structure in Figure~\ref{auvi}, where $i \ge 2$, all of the elements are distinct, and $\{t_0,t_1,\dots ,t_{i-1}\}$ avoids $X$.  
Suppose $M\ba u_0,u_1,\dots ,u_{i-1}/v_{i-1}$ has an $N$-minor. Then $M$ has an element $t_i$ 
that is not in $\{t_0,t_1,\ldots,t_{i-1}\} \cup X$
such that $\{t_{i-1},a_{i-1},u_{i-1},t_i\}$ is a cocircuit, $\{u_{i-1},t_i,a_i\}$ is a triangle,  and $M\ba u_0,u_1,\dots ,u_i/v_i$ has an $N$-minor.  Moreover, if $\{1,2,3\}$ meets 
$(\cup_{j=0}^i \{a_j,u_j,v_j\}) \cup \{t_0,t_1,\ldots,t_{i}\}$, then either $1 = u_i$, or $\{1,2,3\} = \{a_i,u_i,v_i\}$. 
\end{sublemma}

Since $M/v_{i-1}\ba u_{i-1}$ has an $N$-minor, Lemma~\ref{realclaim1} implies that $\{a_{i-1},u_{i-1}\}$ is contained in a $4$-cocircuit.  
Orthogonality implies that this cocircuit meets $\{u_{i-2},t_{i-1}\}$, and Lemma~\ref{bowwow} implies that it avoids $\{a_{i-2},u_{i-2},v_{i-2}\}$, hence it contains $t_{i-1}$.  
Let $t_i$ be the fourth element in this cocircuit.  
Orthogonality implies that $t_i$ avoids the triangles in Figure~\ref{auvi} and also avoids $X$.  Thus $t_i$ is a new element unless $t_i=t_0$.  
In the exceptional case, $\{t_0,a_0,u_0,t_1\}\btu\{t_{i-1},a_{i-1},u_{i-1},t_0\}$, which equals $\{a_0,u_0,t_1,t_{i-1},a_{i-1},u_{i-1}\}$, is a cocircuit. Hence the elements in Figure~\ref{auvi}, excluding $\{t_0,a_i,u_i,v_i\}$, comprise a $3$-separating set in $M$; a \cn.  
Thus $t_i\neq t_0$.  

Next we establish the last part of \ref{newsub}. 
{  Suppose that $\{1,2,3\}$ meets $(\cup_{j=0}^i \{a_j,u_j,v_j\}) \cup \{t_0,t_1,\ldots,t_{i}\}$.  
If $\{1,2,3\}$ avoids $X$, then \ort\ with the cocircuits in Figure~\ref{auvi} implies that $\{1,2,3\}$ contains $\{t_0,t_1,\dots ,t_i\}$.  
Hence $i=2$, and $\lambda (\{a_0,u_0,v_0,a_1,u_1,v_1,t_0,t_1,t_2\})\leq 2$; a \cn.  
We deduce that  $\{1,2,3\}$ meets $X$.  
By~\cite[Lemma~5.4]{cmoVI}, we see that $1=u_i$, as desired, or $\{1,2,3\}=\{a_k,u_k,v_k\}$ for some $k$ in $\{2,3,\dots ,i\}$.  
We assume the latter.  
By \ort\ with the cocircuit $\{2,3,a_0,v_0\}$, we see  that $\{2,3\}$ avoids $\{u_{k-1},t_k,a_k\}$, so $1=a_k$ and $\{2,3\}=\{u_k,v_k\}$.  
{  Now \ort\ with $\{2,3,a_0,v_0\}$ implies that $\{u_k,t_{k+1},a_{k+1}\}$ is not a triangle. }   
Hence $k=i$.}  
Thus  the last part of \ref{newsub} holds. 


We can now apply Lemma~\ref{rainbow} to our structure noting that, by assumption, (ii)(c) of that lemma cannot hold. Since $\{a_0,u_0,v_0\}$ is the unique triangle of $M$ containing $a_0$, it follows that $\{a_0,t_0\}$ is not in a triangle of $M$. Thus   
$M$ has a triangle containing  $\{u_{i-1},t_i\}$.  
By \ort,   this triangle meets $\{v_{i-1},a_i,v_i\}$.  
If $i<n$, then \ort\ with $\{v_i,u_i,a_{i+1},v_{i+1}\}$ implies that the third element of this triangle is $a_i$.  
If $i=n$, then, up to switching the labels on $a_n$ and $v_n$, we may assume that the third element of this triangle is $a_i$.  

To complete the proof of~\ref{newsub}, it remains only to show that $M\ba u_0,u_1,\dots ,u_i/v_i$ has an $N$-minor.  
Suppose not.  
Since $M\ba u_0,u_1,\dots ,u_{i-1}/v_{i-1}$ has an $N$-minor and has $(a_i,t_i,a_{i-1},t_{i-1})$ as a $4$-fan, we know that $M\ba u_0,u_1,\dots ,u_{i-1}/v_{i-1}\ba a_i$ or $M\ba u_0,u_1,\dots ,u_{i-1}/v_{i-1}/t_{i-1}$ has an $N$-minor.  
Since the first matroid is isomorphic to $M\ba u_0,u_1,\dots ,u_{i-1}/v_i\ba u_i$ by Lemma~\ref{stringswitch}, we may assume that the second matroid has an $N$-minor.  
Now Lemma~\ref{stringswitch} implies that $M\ba u_0,u_1,\dots ,u_{i-1}/v_{i-1}/t_{i-1}$ is isomorphic to $M\ba a_0,a_1,\dots ,a_{i-1}/v_{0}/t_{i-1}$. Applying Lemma~\ref{stringswitch} again, this time focusing on the bowtie string at the top of the diagram, we get that the last  matroid is isomorphic to $M\ba a_0,u_0,u_1,\dots ,u_{i-2}/v_{0}/t_{1}$.  
Thus $M\ba a_0,u_0$ has an $N$-minor.  
As this matroid has $\{t_0,t_1\}$ as a cocircuit, we deduce that $M/t_0$ has an $N$-minor. Because $\{a_0,u_0,v_0\}$ is the unique triangle containing $a_0$, it follows that $M$ does not contain the structure in Figure~\ref{ABCDEfig}(A).  Thus Lemma~\ref{BAlemma2} implies that $M/t_0$ has a $4$-fan of the form $(t_1,\be,\ga,\de)$. Hence $\{t_0,t_1,\be,\ga\}$ is a circuit of $M$, and $\{\be,\ga,\de\}$ is a triad of $M$.  
Orthogonality with $\{t_1,a_1,u_1,t_2\}$ implies that $\{\be,\ga\}$ meets a triangle of $M$; a \cn.  
Thus~\ref{newsub} holds.  

\begin{figure}[htb]
\center
\includegraphics{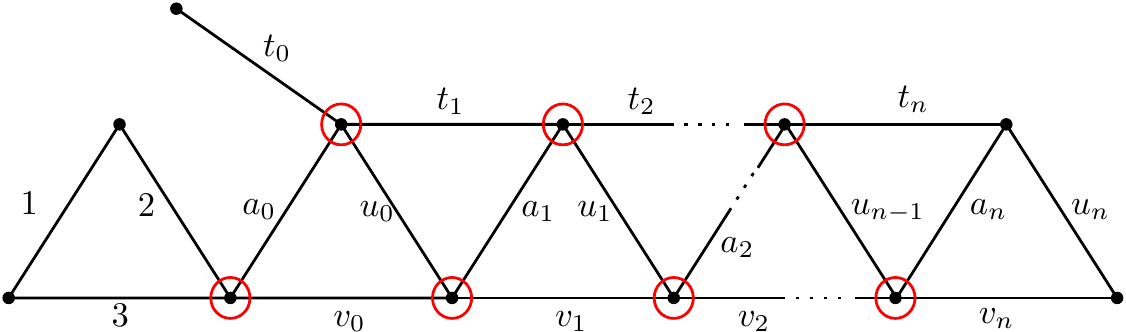}
\caption{$n\geq 2$ and all of the elements shown are distinct except that $1$ may be the same as $u_n$.}
\label{auvn}
\end{figure}

By repeatedly applying \ref{newsub} on our bowtie string, we deduce that $M$ contains the structure in Figure~\ref{auvn} and $M\ba u_0,u_1,\dots ,u_n/v_n$ has an $N$-minor.  By 
Lemma~\ref{stringswitch}, we see that 

\begin{sublemma}
\label{switcheroo}
$M\ba u_0,u_1,\dots ,u_n/v_n \cong M\ba a_0,a_1,\dots ,a_n/v_0$. 
\end{sublemma}

We also get from \ref{newsub} that  
 $\{1,2,3\}$ avoids the other elements in Figure~\ref{auvn} except that $1$ may be $u_n$, or $\{1,2,3\}$ may be $\{a_n,u_n,v_n\}$.  
If $\{1,2,3\}=\{a_n,u_n,v_n\}$, then \ort\ implies that $\{2,3\}=\{u_n,v_n\}$, so $1=a_n$.  By \ref{switcheroo},  we deduce, since $a_0 = 4$, that   $M\ba 1,4$ has an $N$-minor; a \cn.  
We conclude that all the elements in Figure~\ref{auvn} are distinct, except that $1$ may be $u_n$.  

Next we show that 
\begin{sublemma}
\label{vnnotcoguts}
$M\ba u_n$ has no $4$-fan with $v_n$ as its coguts element.  
\end{sublemma}

Suppose that $M\ba u_n$ has $(7,8,9,v_n)$ as a $4$-fan.  
{  Then $(\{a_n,u_n,v_n\},\{7,8,9\},\{8,9,v_n,u_n\})$ is a bowtie, and~\cite[Lemma~5.4]{cmoVI} implies that $\{7,8,9\}=\{a_j,u_j,v_j\}$ for some $j$ in $\{0,1,\dots ,n-2\}$.  
Then $\lambda (X\cup \{t_1,t_2,\dots ,t_n\})\leq 2$. {  Since $\{1,2,3,t_0\}$ avoids $X \cup \{t_1,t_2,\ldots,t_n\}$, to avoid a \cn, we must have that $1 = u_n$. Hence $\{1,2,3\}$ meets $\{8,9,v_n,u_n\}$ in a single element; a \cn. } }  
Thus~\ref{vnnotcoguts} holds.  

Next we show that 
\begin{sublemma}
\label{new4}
$\{a_n,u_n\}$ is contained in a $4$-cocircuit.  
\end{sublemma}

Since $M\ba u_0,u_1,\dots ,u_n/v_n\cong M\ba a_0,a_1,\dots ,a_n/v_0$ by \ref{switcheroo}, we deduce that the second matroid has an $N$-minor.  
Thus both $M\ba u_n$ and $M\ba a_{n-1}$ have $N$-minors.  
Moreover, by Hypothesis {  VII}, $M\ba u_n$ is \ffsc. 
As $M$ has no quick win, Lemma~\ref{6.3rsv} and \ref{vnnotcoguts} imply that~\ref{new4} holds unless $v_n$ is in a triangle $T$ with $u_{n-1}$ or $v_{n-1}$.  
In the exceptional case, \ort\ with the vertex cocircuits in Figure~\ref{auvn} implies that $T$ is $\{v_n,u_{n-1},t_n\}$; a \cn.  
Thus~\ref{new4} holds.  

Orthogonality implies that the $4$-cocircuit containing $\{a_n,u_n\}$ meets $\{u_{n-1},t_n\}$, and Lemma~\ref{bowwow} implies that the cocircuit avoids $\{a_{n-1},u_{n-1},v_{n-1}\}$.  Hence it contains $t_n$.  
Let $t_{n+1}$ be the fourth element in this cocircuit.  
Orthogonality with the triangles in Figure~\ref{auvn} implies that $t_{n+1}$ avoids all of the elements in the figure with the possible exception of $t_0$.  
Now we apply~\cite[Lemma~6.5]{cmoVI} and conclude that one of the following holds:
$M\ba u_0,u_1,\dots ,u_n$ is \ifc; or $M\ba u_0,u_1,\dots ,u_n$ is \ffsc\ and every $4$-fan in this matroid is also a $4$-fan of $M\ba u_n$ with $v_n$ as its coguts element  or is a $4$-fan of $M\ba u_0$ with $a_0$ in its triangle; or $M$ is a quartic M\"obius ladder with $a_0$ in two triangles.  
Since $\{a_0,u_0,v_0\}$ is the only triangle of $M$ containing $a_0$, we deduce that either $M$ has a ladder win, a \cn;  or   $M\ba u_n$ has a $4$-fan with $v_n$ as its coguts element, which  contradicts~\ref{vnnotcoguts}.  We conclude that the lemma holds.
\end{proof}


We continue to consider the case when $M$ contains the structure in Figure~\ref{band}, $M\ba 4$ is \ffsc\ with an $N$-minor, $M\ba 1,4$ does not have an $N$-minor, and $M\ba 6/5$ is \ffsc\ with an $N$-minor.  
Now $M\ba 6/5$ has $(a,4,c,d)$ as a $4$-fan, and the preceding lemma dealt with the case when $M\ba 6/5\ba a$ has an $N$-minor. Our final lemma deals with the case when the last matroid does not have an $N$-minor.  

\begin{lemma}
\label{BAkiller} 
Let $M$ and $N$ be \ifc\ binary matroids with $|E(M)| \ge 16$ and $|E(N)| \ge 7$. 
Suppose that $M$ contains the structure in Figure~\ref{band}, {  where the elements are all distinct except that $u_2$ may be $1$, or $\{a_2,u_2,v_2\}$ may be $\{1,2,3\}$.}  
Suppose that $M\ba 4$ and $M\ba 6$ are \ffsc\ having $N$-minors, and that $M\ba 1,4$ does not have an $N$-minor.  Suppose that Hypothesis {  VII} holds, that 
$\{4,5,6\}$ is the only triangle containing $4$, and that $M\ba 6/5\ba a$ does not have an $N$-minor.  
Then  
\begin{itemize}
\item[(i)] $M$ has a quick win; or 
\item[(ii)] {  $M$ or $M^*$} has  an open-rotor-chain win,  a bowtie-ring win, or a ladder win; or 
\item[(iii)] {  $M$ or $M^*$}  has an enhanced-ladder win; or 
{  
\item[(iv)] deleting the central cocircuit of an augmented $4$-wheel in $M^*$ gives an \ifc\ matroid with an {  $N^*$}-minor.  
}
\end{itemize}
\end{lemma}

\begin{proof} Assume that the lemma fails. Then, by Lemma~\ref{preBA2}, $N\preceq M\ba 6,u$. 
As before,  we relabel  $4,5,6,a,u,$ and $v$ as $a_0,v_0,u_0,a_1,u_1,$ and $v_1$, respectively.   
Let $\{a_0,v_0,u_0\},\{v_0,u_0,a_1,v_1\},\{a_1,v_1,u_1\},\{v_1,u_1,a_2,v_2\},\dots ,\{a_n,v_n,u_n\}$ be a right-maximal bowtie string.  We show next that 

\begin{sublemma}
\label{notminor}
$M\ba u_0,u_1,\dots, u_i/v_i$ has no $N$-minor for all $i$ in $\{1,2,\dots ,n\}$ and 
$M\ba u_0,u_1,\dots, u_j/a_j$ has no $N$-minor for all $j$ in $\{2,3,\dots ,n\}$, but 
$M\ba u_0,u_1,\dots ,u_n$ has an $N$-minor.  
\end{sublemma}

Since $M\ba u_0,u_1/v_1 \cong M\ba 6,a/5$, the first matroid does not have an $N$-minor. It follows by 
 Lemma~\ref{stringybark}  that~\ref{notminor} holds.  

By \ref{notminor} and Hypothesis {  VII}, it follows that $M\ba u_i$ is \ffsc\ for all $i$ in $\{0,1,\ldots,n\}$. 
Establishing the next assertion will occupy most of the rest of the proof of Lemma~\ref{BAkiller}. 
\begin{sublemma}
\label{notbowtie}
$M$ has no bowtie of the form $(\{a_n,u_n,v_n\},\{a_0,v_0,u_0\},\{x,u_n,a_0,v_0\})$ with $x$  in $\{a_n,v_n\}$.  
\end{sublemma}

Suppose that $M$ does have such a bowtie. 
Then $a_0\neq u_n$.  
By possibly interchanging  the labels on $a_n$ and $v_n$, we may assume that $x=v_n$. 
Next we show the following. 

\begin{sublemma}
\label{23alt}
Either $\{2,3\} = \{u_n,v_n\}$; or $\{u_n,v_n,2,3\}$ is a cocircuit of $M$.  
\end{sublemma}

To see this, observe that 
both  $\{u_n,v_n,a_0,v_0\}$ and  $\{2,3,a_0,v_0\}$ are cocircuits of $M$. Taking their symmetric difference, we immediately get \ref{23alt}. 

We now eliminate the first possibility in \ref{23alt}. 

\begin{sublemma}
\label{23unvn}
$\{2,3\} \neq \{u_n,v_n\}$.  
\end{sublemma}

Suppose $\{2,3\} = \{u_n,v_n\}$. Then $a_n = 1$ and $M$ has 
$(\{a_0,u_0,v_0\}, \{u_0,v_0,a_1,v_1\}, \linebreak \{a_1,u_1,v_1\},\ldots, \{a_n,u_n,v_n\}, \{u_n,v_n,a_0,v_0\})$ as a ring of bowties. We now apply \cite[Lemma {  5.5}]{cmoVI} noting that, since $M$ does not have a bowtie-ring win, part (i) of that lemma does not hold. Moreover, by \ref{notminor}, part (iii) of that lemma does not hold. Thus part (ii) of that lemma holds, that is, $M\ba u_0,u_1,\dots ,u_n$ is \sfc\ but not \ifc, and every $4$-fan of it has the form $(\alpha, \beta,\gamma, \delta)$ where $\{\alpha, \beta, \gamma\}$ avoids $\{a_0,u_0,v_0,a_1,u_1,v_1,\ldots,a_n,u_n,v_n\}$, and $M$ has a cocircuit $\{\beta,\gamma, \delta,u_i\}$ for some $i$ in $\{0,1,\ldots,n\}$ and some $\delta$ in $\{a_i,v_i\}$. 

With a view to applying Lemma 10.4 of \cite{cmoVI}, we show next that $i \neq 0$. Assume the contrary. Then $M$ has $\{\beta,\gamma, \delta,u_0\}$ as a cocircuit and $\{u_0,c,a_1\}$ as a triangle where $\delta \in \{a_0,v_0\}$. By \ort, $c \in \{\beta,\gamma\}$. Then \ort\ between $\{\alpha, \beta, \gamma\}$ and $\{d,a_0,u_0,c\}$ implies that 
$d \in \{\alpha, \beta, \gamma\}$. Hence $M\ba a_0$ has a $5$-fan; a \cn. We conclude that $i \neq 0$.

{  We now apply  \cite[Lemma 10.4]{cmoVI} noting that we get a \cn\ using \ref{notminor} unless $i = 1$ and $\delta = a_i$.  In the exceptional case, \ort\ between $\{u_0,c,a_1\}$ and $\{a_1,u_1,\beta,\gamma\}$ implies that $\{u_0,c\}$ meets $\{\beta, \gamma\}$. By construction, $u_0\notin \{\al,\be,\ga\}$. Hence
$c \in \{\beta,\gamma\}$. By \ort\ with $\{a_0,v_0,c,d\}$, the triangle $\{\alpha, \beta, \gamma\}$ contains $\{c,d\}$, so $M\ba 4$ has a $5$-fan; a \cn. Thus \ref{23unvn} holds.}

By \ref{23alt}, we now know that $M$ has $\{2,3,u_n,v_n\}$ as a cocircuit.  
If $\{1,2,3\}$ avoids $\{a_0,u_0,v_0,a_1,u_1,v_1,\ldots,a_n,u_n,v_n\}$, then we can adjoin $\{2,3,u_n,v_n\}$ and $\{1,2,3\}$ to our right-maximal bowtie string to get a \cn. Thus $\{1,2,3\}$ meets $\{a_0,u_0,v_0,a_1,u_1,v_1,\ldots,a_n,u_n,v_n\}$. By \cite[Lemma {  5.4}]{cmoVI}, 
$\{1,2,3\} = \{a_j,u_j,v_j\}$ for some $j$ with $0 \le j \le n-2$. Certainly $j \neq 0$. Moreover, $j \neq 1$ otherwise the cocircuit $\{2,3,a_0,v_0\}$ contradicts Lemma~\ref{bowwow}.  
{  If $u_j\in\{2,3\}$, then $M\ba u_0,u_1,\dots ,u_n$ has $v_n$ in a $2$-cocircuit; a \cn\ to~\ref{notminor}.  
Thus $\{2,3\}=\{a_j,v_j\}$ and $\{2,3,u_n,v_n\}\btu\{v_{j-1},u_{j-1},a_j,v_j\}$, which is   $\{v_{j-1},u_{j-1},u_n,v_n\}$, is a cocircuit in $M$.  
Again $M\ba u_0,u_1,\dots ,u_n$ has $v_n$ in a $2$-cocircuit; a \cn. 
{  Thus~\ref{notbowtie} holds. 

We can now apply~\cite[Lemma~10.1]{cmoVI}.  Since $n\geq 2$, we conclude that either $M\ba u_0,u_1/v_1$ has an $N$-minor, or $M$ has $a_0$ in a triangle other than $\{a_0,u_0,v_0\}$.  The former option gives a \cn\ to~\ref{notminor}, and the latter gives a \cn\ to the assumptions of the lemma.}}
\end{proof}

\section{The proof of the main theorem}
\label{pomt}

In this section, as we shall see, it is quite straightforward to  assemble the parts from earlier sections to complete the proof of the main result.

\begin{proof}[Proof of Theorem~\ref{mainguy}]
Assume that the theorem fails.   
Theorem~\ref{killcasek} implies that Hypothesis~VII holds.   
Now $M$ has a bowtie $(\{1,2,3\},\{4,5,6\},\{2,3,4,5\})$ where $M\ba 4$ is \ffsc\ with an $N$-minor and $M\ba 1,4$ has no $N$-minor.  
Lemma~\ref{airplane} implies that $M\ba 4/5$ is \ffsc\ with an $N$-minor.  
By Lemma~\ref{ABCDE}, we know that $M$ contains (A), (B), or (C) in Figure~\ref{ABCDEfig}, that $M\ba 6$ is \ffsc, and that $\{4,5,6\}$ is the only triangle in $M$ containing $5$.  
Moreover, the elements in (A), (B), or (C) are all distinct except that $a$ may equal $1$ in (B) or (C).  

Suppose that $M$ contains the configuration in Figure~\ref{ABCDEfig}(C).  
Lemma~\ref{ccrider} implies that $M\ba 4/5,6$ does not have an $N$-minor.  
Since $M\ba 4$ is \ffsc\ with an $N$-minor and has $(1,2,3,5)$ and $(a,b,c,6)$ as $4$-fans, and $M\ba 1,4$ does not have an $N$-minor, we deduce that $M\ba 4/5\ba a$ has an $N$-minor.  
Then Lemma~\ref{killtobler2} gives a \cn.  
We conclude  that $M$ does not contain the configuration in Figure~\ref{ABCDEfig}(C).  

{  If $M$ contains the configuration in Figure~\ref{ABCDEfig}(A),   
then Lemma~\ref{AconfigBOOM} {  gives a \cn.}  
Hence we may assume that} $M$ does not contain either of the configurations in Figure~\ref{ABCDEfig}(A) or 
Figure~\ref{ABCDEfig}(C). Thus  $M$ contains the configuration in Figure~\ref{ABCDEfig}(B). Lemma~\ref{BAlemma} implies that $M$ contains the configuration in Figure~\ref{BAfiga}, where $M\ba 6$ is \ffsc\ and all of the elements are distinct except that $1$ may be $u$ {  or $\{1,2,3\}$ may equal $\{a_2,u_2,v_2\}$}.   
By Lemma~\ref{newtry},   $M$ contains the configuration in Figure~\ref{band} where $\{4,5,6\}$ is the only triangle containing $4$, and all of the elements are distinct except that $u_2$ may be $1$.   
If $M\ba 6/5\ba a$ has an $N$-minor, then Lemma~\ref{parti} gives a \cn.  Thus we may assume that $N \not \preceq M\ba 6/5\ba a$. 
Now $M\ba 4/5\cong M\ba 6/5$, so $M\ba 6/5$ is \ffsc\ with an $N$-minor. 
Using Lemma~\ref{BAkiller}, we get a  \cn\ that completes the proof of the theorem.   
\end{proof}

\section*{Acknowledgements} 
The authors thank Dillon Mayhew for numerous helpful discussions.


\begin{thebibliography}{99}




\bibitem{cmochain} Chun, C., Mayhew, D., and Oxley, J.,
A chain theorem for internally $4$-connected binary matroids,
{\em J.~Combin.~Theory Ser.~B} {\bf 101} (2011), 141--189.



\bibitem{cmoIII} Chun, C., Mayhew, D., and Oxley, J.,
Towards a splitter theorem for internally $4$-connected binary matroids III, 
{\it Adv. in Appl. Math.} {\bf 51} (2013), 309--344.

\bibitem{cmoIV} Chun, C., Mayhew, D., and Oxley, J.,
Towards a splitter theorem for internally $4$-connected binary matroids IV,
{\it Adv. in Appl. Math.} {\bf 52} (2014), 1--59.

\bibitem{cmoV} Chun, C., Mayhew, D., and Oxley, J.,
Towards a splitter theorem for internally $4$-connected binary matroids V, 
{\it Adv. in Appl. Math.} {\bf 52} (2014), 60--81.

\bibitem{cmoVI} Chun, C. and Oxley, J.,
Towards a splitter theorem for internally $4$-connected binary matroids VI, 
submitted.

\bibitem{cmoVIII} Chun, C., Mayhew, D., and Oxley, J.,
Towards a splitter theorem for internally $4$-connected binary matroids VIII: The theorem,  
in preparation.



\bibitem{johtho} Johnson, T. and Thomas, R., Generating internally four-connected graphs, {\it J. Combin. Theory Ser. B} {\bf 85} (2002), 21--58.



{  \bibitem{kinglem} Kingan, S. and Lemos, M., 
Almost-graphic matroids,
{\it Adv. in Appl. Math.} {\bf 28} (2002), 438--477.}


\bibitem{mayroywhi} Mayhew, D., Royle, G., and Whittle, G., The binary matroids with no $M(K_{3,3})$-minor, {\it Mem. Amer. Math. Soc.}  {\bf 208} (2010), no. 981. 



\bibitem{oxrox} Oxley, J., {\it Matroid theory}, Second edition, Oxford University Press, New York, 2011.





\bibitem{seymour} Seymour, P. D.,
Decomposition of regular matroids,
{\em J.~Combin.~Theory Ser.~B} {\bf 28} (1980), 305--359.

\bibitem{wtt} Tutte, W. T.,
Connectivity in matroids,
{\em Canad.~J.~Math.} {\bf 18} (1966), 1301--1324.


\bibitem{zhou} Zhou, X., Generating an internally $4$-connected binary matroid from another, 
{\it Discrete Math.} {\bf 312} (2012), 2375--2387.

\end{thebibliography}
\end{document}